\documentclass[10pt]{article}
\usepackage{amsfonts,amssymb,amsmath,amsthm,mathtools}
\usepackage{mathrsfs}
\usepackage{enumitem}
\usepackage[margin=2cm]{geometry}
\usepackage[english]{babel}
\usepackage[T1]{fontenc}
\usepackage{lmodern}
\usepackage{microtype}
\usepackage{hyperref}
\usepackage[nameinlink,capitalize]{cleveref}
\usepackage{url}
\hypersetup{
  colorlinks=true,
  linkcolor=blue,
  citecolor=blue,
  urlcolor=blue,
  pdftitle={Commutator Estimates Uniform in the Screening Parameter and Mean-Field Limits for Yukawa Interactions},
  pdfauthor={Ning Jiang, Zhengyang Qiao, Juntao Wu, Jiangwei Zhang},
  pdfsubject={Yukawa interactions in arbitrary dimension, potential truncation, modulated energy, commutator estimates, mean-field limits, propagation of chaos, and the Coulomb limit},
  pdfkeywords={Yukawa interaction, screened Coulomb potential, modified Helmholtz operator, potential truncation, commutator estimate, stress-energy tensor, modulated energy, mean-field limit, propagation of chaos, Coulomb limit},
  pdfdisplaydoctitle=true
}
\numberwithin{equation}{section}

\newcommand{\R}{\mathbb R}
\newcommand{\dd}{\,\mathrm d}
\newcommand{\Sd}{|\mathbb S^{d-1}|}

\newcommand{\cF}{\mathcal F}
\newcommand{\cH}{\mathcal H}
\newcommand{\cQ}{\mathcal Q}
\newcommand{\cD}{\mathcal D}
\newcommand{\cS}{\mathcal S}

\newtheoremstyle{plainnormalnote}%
  {3pt}% Space above
  {3pt}% Space below
  {\itshape}% Body font
  {0pt}% Indent amount
  {\bfseries}% Theorem head font
  {.}% Punctuation after theorem head
  {.5em}% Space after theorem head
  {\thmname{#1}\thmnumber{ #2}\thmnote{ \normalfont(#3)}}% Theorem head spec
\theoremstyle{plainnormalnote}
\newtheorem{theorem}{Theorem}[section]
\newtheorem{corollary}[theorem]{Corollary}
\newtheorem{proposition}[theorem]{Proposition}
\newtheorem{lemma}[theorem]{Lemma}
\theoremstyle{definition}
\newtheorem{definition}{Definition}[section]
\newtheoremstyle{boldremark}%
  {3pt}% Space above
  {3pt}% Space below
  {\normalfont}% Body font
  {0pt}% Indent amount
  {\bfseries}% Theorem head font
  {.}% Punctuation after theorem head
  {.5em}% Space after theorem head
  {\thmname{#1}\thmnumber{ #2}\thmnote{ \textbf{(#3)}}}% Theorem head spec
\theoremstyle{boldremark}
\newtheorem{remark}{Remark}[section]
\theoremstyle{plainnormalnote}

\crefname{theorem}{Theorem}{Theorems}
\Crefname{theorem}{Theorem}{Theorems}
\crefname{proposition}{Proposition}{Propositions}
\Crefname{proposition}{Proposition}{Propositions}
\crefname{lemma}{Lemma}{Lemmas}
\Crefname{lemma}{Lemma}{Lemmas}
\crefname{corollary}{Corollary}{Corollaries}
\Crefname{corollary}{Corollary}{Corollaries}
\crefname{definition}{Definition}{Definitions}
\Crefname{definition}{Definition}{Definitions}
\crefname{equation}{equation}{equations}
\Crefname{equation}{Equation}{Equations}

\begin{document}

\title{\texorpdfstring{\Large\bfseries Commutator Estimates Uniform in the Screening Parameter and Mean-Field Limits for Yukawa Interactions}{Commutator Estimates Uniform in the Screening Parameter and Mean-Field Limits for Yukawa Interactions}}
\author{Ning Jiang$^\mathrm{a}$,\quad Zhengyang Qiao$^\mathrm{b}$,\quad Juntao Wu$^\mathrm{a}$,\quad Jiangwei Zhang$^\mathrm{c}$\\
{\small\itshape $^\mathrm{a}$ School of Mathematics and Statistics, Wuhan University, Wuhan, Hubei 430072, P. R. China}\\
{\small\itshape $^\mathrm{b}$ College of Intelligence Science and Technology, National University of Defense Technology,}\\
{\small\itshape Changsha 410073, P. R. China}\\
{\small\itshape $^\mathrm{c}$ Institute of Applied Physics and Computational Mathematics, Beijing 100088, P. R. China}\\[1mm]
{\small Corresponding author: Juntao Wu, \texttt{00036371@whu.edu.cn}}\\[1mm]
}
\date{}

\maketitle

\begin{abstract}
We study quantitative mean-field limits for classical particles with Yukawa
(screened Coulomb) interactions in every fixed dimension $d\ge2$, uniformly as
the screening parameter $\kappa$ tends to zero.  Our main result is a
first-order commutator estimate in the natural Yukawa modulated energy, with
an additive error of order $N^{-2/d}$ for $d\ge3$ and $(1+\log N)/N$ for
$d=2$.  It requires only a bounded reference density and a Lipschitz transport
field, with no uniform lower bound on interparticle distances and no negative
power of $\kappa$.

The modified Helmholtz operator $-\Delta+\kappa^2$ creates the main new
difficulty.  Truncating the potential to a constant inside each truncation
ball produces a surface charge and a positive volume charge whose total mass
is strictly less than one.  We keep the reference density unchanged and
control this loss of mass through an exact Green function representation and
renormalized energy identities.  A stress-energy identity with interface
terms and averaging over the truncation radii then give the uniform
commutator estimate.

Combined with the modulated energy dissipation identity and a normalized
quadratic transport cost, this estimate yields weak--strong stability,
propagation of chaos, and time-integrated control of the mean-square
difference between empirical and mean-field forces.  We also prove a
quantitative Yukawa-to-Coulomb limit.  For smooth product data, $N\to\infty$
and $\kappa\downarrow0$ may be taken simultaneously with no relation between
their rates; in dimension three, a direct comparison at the particle level
also holds for general symmetric initial laws with finite initial error.
\end{abstract}

\medskip
\noindent\textbf{Keywords:} Yukawa interaction,  commutator estimate,  modulated energy, mean-field limit, propagation of chaos, Coulomb limit.

\medskip
\noindent\textbf{MSC 2020:} Primary 35Q70; Secondary 35B35, 35J05, 49Q22, 82C22.

\section{Introduction}

\subsection{Yukawa interaction and the mean-field problem}

The Yukawa interaction is a standard model of screened Coulomb forces in
Debye--H\"uckel theory, plasma physics, and liquid-state theory
\cite{Yukawa1935,DebyeHuckel1923,Ichimaru1992,HansenMcDonald2013}.  Its
screening length is $\kappa^{-1}$, and the regime $\kappa\downarrow0$
corresponds to the loss of screening and the recovery of the long-range
Coulomb interaction.  Our goal is to develop a quantitative mean-field theory for the particle
system that remains stable throughout this transition.  We therefore treat
the mean-field limit $N\to\infty$ and the Coulomb limit $\kappa\downarrow0$
simultaneously, with constants that do not deteriorate as the screening
length diverges.

Fix an integer $d\ge2$.  For $\kappa>0$, the Yukawa potential is
\begin{equation*}
 g_{\kappa,d}(x)=(2\pi)^{-d/2}
 \left(\frac{\kappa}{|x|}\right)^{\nu_d}K_{\nu_d}(\kappa|x|),
 \qquad \nu_d:=\frac d2-1,
 \qquad (-\Delta+\kappa^2)g_{\kappa,d}=\delta_0,
\end{equation*}
where $K_\nu$ denotes the modified Bessel function of the second kind.  We use the standard recurrence relations, derivative formulas, and small- and
large-argument asymptotics for modified Bessel functions; see, for example,
\cite[Chapter~10]{NIST2010}.  In three dimensions,
$g_{\kappa,3}(x)=e^{-\kappa|x|}/(4\pi|x|)$.  We use the Coulomb representative
\begin{equation}\label{eq:coulomb_kernel_general_intro}
 g_{0,d}(x)=
 \begin{cases}
 -\dfrac1{2\pi}\log|x|+\dfrac{\log2-\gamma_E}{2\pi},&d=2,\\[1mm]
 \dfrac1{(d-2)|\mathbb S^{d-1}|}|x|^{2-d},&d\ge3.
 \end{cases}
\end{equation}
In dimension two this fixes the additive normalization used in the Coulomb
limit; additive constants do not affect the force or the energy of a neutral
measure.  The radial Bessel identity
\[
 \frac{\dd}{\dd z}\bigl(z^{-\nu}K_\nu(z)\bigr)
 =-z^{-\nu}K_{\nu+1}(z)
\]
gives
\begin{equation}\label{eq:yukawa_force_general_intro}
 g_{\kappa,d}'(r)
 =-(2\pi)^{-d/2}\kappa^{d/2}r^{1-d/2}K_{d/2}(\kappa r).
\end{equation}
Thus the Yukawa force has the same Coulomb singularity at the origin for every
$\kappa\ge0$, whereas for $\kappa>0$ the interaction is exponentially screened
at large distances.

We consider the repulsive first-order gradient flow
\begin{equation}\label{eq:particle_system_intro}
 \dot x_i(t)=-\frac1N\sum_{j\ne i}\nabla g_{\kappa,d}(x_i(t)-x_j(t)),
 \qquad i=1,\ldots,N,
\end{equation}
with the self-interaction omitted.  We set
\begin{equation*}
 E_{N,\kappa}(X_N)
 :=\frac1{2N^2}\sum_{i\ne j}g_{\kappa,d}(x_i-x_j),
 \qquad
 m_N(X_N):=\frac1N\sum_{i=1}^N|x_i|^2,
\end{equation*}
and write
$\mu_{X_N}:=N^{-1}\sum_i\delta_{x_i}$ for the empirical measure.  The formal
mean-field equation is
\begin{equation}\label{eq:pde_intro}
 \partial_t\rho+\nabla\cdot(\rho u_\rho^\kappa)=0,
 \qquad u_\rho^\kappa=-\nabla g_{\kappa,d}*\rho.
\end{equation}
The interaction energy is
\[
 \frac12\iint_{\R^d\times\R^d}
 g_{\kappa,d}(x-y)\rho(x)\rho(y)\,\dd x\dd y.
\]
For fixed $\kappa>0$, exponential screening improves the decay at infinity.
However, estimates containing a factor such as $\kappa^{-p}$ degenerate as
the screening length diverges.  We therefore work in the natural Yukawa
energy and seek constants that remain bounded for
$0<\kappa\le\kappa_*$, without regularizing the Coulomb singularity at short
distances.  The dissipative structures of
\eqref{eq:particle_system_intro} and \eqref{eq:pde_intro} are combined below
with a quadratic transport cost in the Wasserstein framework of
\cite{AmbrosioGigliSavare2008}.

A central difficulty is that screening also changes the source generated by
the truncation.  Let $g_{\kappa,r}$ agree with $g_{\kappa,d}$ outside $B_r$
and equal the constant $g_{\kappa,d}(r)$ inside $B_r$.  Then
Section~\ref{sec:truncation_energy} proves the exact identity
\begin{equation}\label{eq:intro_smeared_charge_preview}
 (-\Delta+\kappa^2)g_{\kappa,r}
 =A_d(\kappa r)\sigma_r
  +\kappa^2g_{\kappa,d}(r)\mathbf1_{B_r}\,\dd x,
\end{equation}
where $\sigma_r$ is the probability surface measure on $\partial B_r$.
Consequently the smeared charge contains both a surface part and a positive
volume part.  Its total mass $M_d(\kappa r)$ satisfies
\begin{equation}\label{eq:intro_mass_loss_preview}
 1-M_d(\kappa r)
 =\kappa^2\int_{B_r}
   \bigl(g_{\kappa,d}(|x|)-g_{\kappa,d}(r)\bigr)\,\dd x,
 \qquad
 0\le1-M_d(\kappa r)\le\frac{\kappa^2r^2}{2d}.
\end{equation}
Thus the truncation does not preserve the total mass of the source.  Although
the smeared source is no longer neutral, it has finite Yukawa energy for every
fixed $\kappa>0$.  We leave the background $\rho$ unchanged and account
explicitly for $1-M_d(\kappa r)$ in the renormalized energy and in the source
coefficients.  The exact Green function identity
$g_{\kappa,d}*\delta_0^{\kappa,r}=g_{\kappa,r}$ reduces the
particle--background truncation error to a local term and yields estimates
that remain uniform as $\kappa\downarrow0$, without any rescaling of the
background.

A separate obstruction occurs at low frequency.  For a neutral fluctuation
$\nu$, the Coulomb and Yukawa quadratic forms have Fourier weights
$|\xi|^{-2}$ and $(|\xi|^2+\kappa^2)^{-1}$.  Their ratio is
\[
 \frac{|\xi|^{-2}}{(|\xi|^2+\kappa^2)^{-1}}
 =1+\frac{\kappa^2}{|\xi|^2},
\]
which is unbounded as $|\xi|\downarrow0$.  Section~\ref{sec:coulomb_limit}
shows by dilation that, even for a fixed $\kappa>0$, there is no constant
$C_\kappa$ such that $I_0[\nu]\le C_\kappa I_\kappa[\nu]$ for all smooth
compactly supported neutral $\nu$.  Consequently, a Coulomb commutator bound together with a separate estimate of
the Yukawa--Coulomb remainder does not, by itself, close in the natural Yukawa
modulated energy.  We therefore establish the commutator estimate directly
for $L_\kappa=-\Delta+\kappa^2$.

Let $\rho_N^0$ be a symmetric probability measure on $(\R^d)^N$.  The
deterministic particle flow induces an $N$-particle law $\rho_N(t)$, with
$k$-particle marginal $\rho_{N:k}(t)$.  We use propagation of chaos in the sense of Kac and Sznitman
\cite{Kac1956,Sznitman1991,HaurayMischler2014,ChaintronDiezI2022}.  We first
prove stability for the full $N$-particle law by combining the shifted
modulated energy with the normalized quadratic Wasserstein cost.  Symmetry
then yields corresponding estimates for each fixed marginal.  We also
construct, for smooth compactly supported data, a common local class of
classical limiting solutions with bounds uniform for
$0\le\kappa\le\kappa_*$.

\subsection{Main results and contributions}

In the statements below, $W_2$ denotes the quadratic Wasserstein distance for
the Euclidean cost on the relevant finite-dimensional space; on $(\R^d)^N$
the cost is $\sum_{i=1}^N|x_i-y_i|^2$, and the factor $1/N$ is written
explicitly whenever a normalized cost is used.  Unless stated otherwise,
``uniform in $\kappa$'' means uniform for $0<\kappa\le\kappa_*$ with a fixed
$\kappa_*<\infty$.  The Coulomb case $\kappa=0$ is treated separately.
Generic constants may change from line to line; whenever uniformity is part of
a theorem, the allowed parameter dependence of the constant is stated
explicitly.

\paragraph{Main analytic input and consequences.}
Theorem~\ref{thm:yukawa_commutator} is the main analytic result: it gives a
first-order commutator estimate in the natural Yukawa modulated energy,
uniformly for $0<\kappa\le\kappa_*$.  The estimate is proved directly for the
modified Helmholtz operator rather than by perturbing a Coulomb estimate.
Sections~\ref{sec:yukawa_main}--\ref{sec:yukawa_commutator} contain this new
analytic argument.  Theorem~\ref{thm:yukawa_meanfield} then combines the
commutator estimate with the modulated energy dissipation identity and a
normalized quadratic transport cost to obtain weak--strong stability and
propagation of chaos.  Proposition~\ref{prop:yukawa_classical_solutions}
provides a common local class of classical solutions, and
Section~\ref{sec:coulomb_limit} treats the Coulomb limit.

For a collision-free configuration \(X_N\) and a bounded probability density
\(\rho\), we define the off-diagonal modulated energy and, for a globally
Lipschitz vector field \(u:\R^d\to\R^d\), the associated first-order
commutator by
\begin{align}
 \cF_N^{g_{\kappa,d}}(X_N,\rho)
 &:=\iint_{x\ne y}g_{\kappa,d}(x-y)\,
      \dd(\mu_{X_N}-\rho)(x)\dd(\mu_{X_N}-\rho)(y),
 \notag\\
 \cH_N^{g_{\kappa,d}}(X_N,\rho;u)
 &:=\iint_{x\ne y}(u(x)-u(y))\cdot\nabla g_{\kappa,d}(x-y)\,
      \dd(\mu_{X_N}-\rho)(x)\dd(\mu_{X_N}-\rho)(y).
 \label{eq:H_def_intro}
\end{align}
Here the off-diagonal convention removes the labeled particle
self-interactions.  Under the hypotheses of Theorem~\ref{thm:yukawa_commutator},
the Yukawa energy and commutator terms above are finite;
in particular, the commutator integrals are absolutely convergent.  This uses
only the local integrability of the Coulomb singularity after the Lipschitz
cancellation and the large-scale decay of the Yukawa kernel; a direct estimate
is recorded in \eqref{eq:point_commutator_explicit_lipschitz} and the paragraph
following it.  For globally Lipschitz transport fields,
$\|\nabla u\|_{L^\infty}$ denotes the usual Lipschitz seminorm.  Whenever a
different interaction kernel $g$ is specified and the displayed integrals are
well defined, $\cF_N^g$ and $\cH_N^g$ mean these same scalar functionals with
$g$ in place of $g_{\kappa,d}$.  Their explicit particle--particle,
particle--background, and background--background expansions are recorded in
Section~\ref{sec:yukawa_main}.

The only additional normalization needed in dimension two concerns the additive
constant in the logarithmic potential.  We fix it once and for all by
\begin{equation*}
 c_{\kappa,d}:=
 \begin{cases}
 0,&d\ge3,\\[1mm]
 \dfrac{\log\kappa}{2\pi},&d=2,
 \end{cases}
 \qquad
 g_{\kappa,d}^{\sharp}:=g_{\kappa,d}+c_{\kappa,d}.
\end{equation*}
With the representative fixed in \eqref{eq:coulomb_kernel_general_intro},
$g_{\kappa,2}^{\sharp}$ converges to $g_{0,2}$ as $\kappa\downarrow0$.
Additive constants do not affect the force, but they do fix the normalization
of the finite-$N$ off-diagonal energy.  We therefore define
\begin{equation*}
 \cF_N^{\kappa,\sharp}(X_N,\rho)
 :=\cF_N^{g_{\kappa,d}}(X_N,\rho)-\frac{c_{\kappa,d}}N.
\end{equation*}
Because the signed measure $\mu_{X_N}-\rho$ has total mass zero and the
off-diagonal convention removes exactly the $N$ labeled particle self-interaction terms,
this normalization is equivalently
\begin{equation}\label{eq:sharp_energy_exact_constant_shift}
 \cF_N^{\kappa,\sharp}(X_N,\rho)
 =\cF_N^{g_{\kappa,d}^{\sharp}}(X_N,\rho).
\end{equation}

The additive convention above affects only finite-$N$ off-diagonal energies.
All source identities, truncations, positive field energies, and stress tensors
continue to use the unshifted Yukawa kernel.  The precise identities for
constant shifts, including the distinction between an off-diagonal neutral
energy and the quadratic energy of a nonneutral smeared charge, are recorded in
Remarks~\ref{rem:additive_shift_bookkeeping}--\ref{rem:two_dimensional_normalization}
in Section~\ref{sec:yukawa_main}.

For statements involving probability measures on configuration space, all
singular off-diagonal modulated energies, including their Coulomb counterparts,
are extended by $+\infty$ on collision configurations.
We define
\begin{equation*}
 \ell_d(r):=
 \begin{cases}
 1+\log_+(1/r),&d=2,\\[1mm]
 r^{2-d},&d\ge3,
 \end{cases}
\end{equation*}
and
\begin{equation}\label{eq:intro_additive_error}
 \varepsilon_N^\kappa(\eta)
 :=\frac{\ell_d(\eta)}N+\eta^2+\kappa^2\eta^2.
\end{equation}
For $d\ge3$, this is
$N^{-1}\eta^{2-d}+\eta^2+\kappa^2\eta^2$, while in $d=2$ it is
$N^{-1}(1+\log_+(1/\eta))+\eta^2+\kappa^2\eta^2$.  The term
$\kappa^2\eta^2$ is dominated by $\eta^2$ for fixed $\kappa_*$, but we keep it
separate in order to track the contribution of $1-M_d(\kappa r)$.

Our main estimate is the following.

\begin{theorem}[Uniform Yukawa commutator estimate]\label{thm:yukawa_commutator}
Fix an integer $d\ge2$, $\kappa_*>0$, and $\Lambda\ge1$.  There exists
$\eta_0=\eta_0(d,\kappa_*)>0$, and there exist $C,B\ge1$ depending only
on $d$, $\kappa_*$ and $\Lambda$, such that for every $N\ge2$, every
$0<\kappa\le\kappa_*$, every
$0<\eta<\eta_0$, every bounded probability density $\rho$ satisfying
\begin{equation}\label{eq:static_background_bound}
 \|\rho\|_{L^\infty}\le\Lambda,
\end{equation}
every pairwise distinct configuration $X_N$, and every globally Lipschitz
vector field $u:\R^d\to\R^d$, one has
\begin{equation}\label{eq:yukawa_static_shift}
 \cF_N^{\kappa,\sharp}(X_N,\rho)
 +B\varepsilon_N^\kappa(\eta)
 \ge\varepsilon_N^\kappa(\eta)\ge0
\end{equation}
and
\begin{equation}\label{eq:yukawa_static_main}
 |\cH_N^{g_{\kappa,d}}(X_N,\rho;u)|
 \le C\|\nabla u\|_{L^\infty}
 \left(\cF_N^{\kappa,\sharp}(X_N,\rho)
 +B\varepsilon_N^\kappa(\eta)\right).
\end{equation}
The same constants $C,B$ and the same scale $\eta_0$ work simultaneously
for all $N\ge2$, all $0<\kappa\le\kappa_*$, all admissible $\eta$, and all
collision-free configurations.  They are independent of the individual
particle separations, and in particular contain no factor involving
$\min_{i\ne j}|x_i-x_j|^{-1}$ or any negative power of $\kappa$.  Apart from
the probability normalization and \eqref{eq:static_background_bound}, they do
not depend on moments, support bounds, or additional regularity of $\rho$;
the field $u$ enters only through $\|\nabla u\|_{L^\infty}$.
\end{theorem}

\smallskip
We retain the full term \(\varepsilon_N^\kappa(\eta)\) in
\eqref{eq:yukawa_static_shift} because its positive margin absorbs the
truncation, background, and mass-loss errors.

It is convenient to set
\begin{equation}\label{eq:dimension_dependent_N_rate}
 a_{N,d}:=
 \begin{cases}
 \dfrac{1+\log N}{N},&d=2,\\[1mm]
 N^{-2/d},&d\ge3.
 \end{cases}
\end{equation}

\begin{corollary}[Optimized finite-$N$ commutator estimate]
\label{cor:yukawa_commutator_optimized_scale}
Under the assumptions of Theorem~\ref{thm:yukawa_commutator}, set
\begin{equation}\label{eq:optimized_truncation_radius}
 \eta_N:=\frac12\min\{1,\eta_0\}
 \begin{cases}
 N^{-1/2},&d=2,\\
 N^{-1/d},&d\ge3.
 \end{cases}
\end{equation}
Then there exist $C=C(d,\kappa_*,\Lambda)\ge1$ and
$c=c(d,\kappa_*)>0$ such that
\begin{equation*}
 \cF_N^{\kappa,\sharp}(X_N,\rho)+Ca_{N,d}\ge c\,a_{N,d}>0
\end{equation*}
and
\begin{equation*}
 |\cH_N^{g_{\kappa,d}}(X_N,\rho;u)|
 \le C\|\nabla u\|_{L^\infty}
 \left(\cF_N^{\kappa,\sharp}(X_N,\rho)
 +Ca_{N,d}\right).
\end{equation*}
\end{corollary}

For $d\ge3$, the correction $N^{-2/d}$ has the same order
$N^{(d-2)/d-1}$ as in the Coulomb case at the Riesz exponent $s=d-2$.  In dimension
two the argument yields $(1+\log N)/N$; no optimality of the logarithmic factor
is asserted for fixed $\kappa>0$.

The commutator estimate implies the following weak--strong stability result.  
\begin{theorem}[Uniform weak--strong stability and propagation of chaos]
\label{thm:yukawa_meanfield}
Fix $d\ge2$, $\kappa_*>0$, and $0<\kappa\le\kappa_*$.  Let
$\rho$ solve \eqref{eq:pde_intro} on $[0,T]$ and assume
\[
\rho\in C([0,T];(\mathcal P_2(\R^d),W_2)),
\]
\begin{equation}\label{eq:dynamic_solution_bound}
	\Lambda_T:=1+\sup_{t\le T}\left(
	\|\rho(t)\|_{L^\infty}
	+\|u_\rho^\kappa(t)\|_{W^{1,\infty}}
	\right)<\infty.
\end{equation}
Let $\rho_N^0\in\mathcal P_2((\R^d)^N)$ be symmetric.  Let
$B=B(d,\kappa_*,\Lambda_T)$ and $\eta_0=\eta_0(d,\kappa_*)$ be chosen
as in Theorem~\ref{thm:yukawa_commutator}.  For $0<\eta<\eta_0$, we set
\begin{equation}\label{eq:intro_initial_error}
 \mathfrak E_N^{\kappa,0}(\eta)
 :=\frac1N W_2^2(\rho_N^0,\rho(0)^{\otimes N})
 +\int\left[\cF_N^{\kappa,\sharp}(X_N,\rho(0))
 +B\varepsilon_N^\kappa(\eta)\right]\dd\rho_N^0,
\end{equation}
and assume $\mathfrak E_N^{\kappa,0}(\eta)<\infty$.  This finiteness
already implies that $\rho_N^0$ gives zero mass to the collision set and has
finite expected microscopic Yukawa energy; this is verified in
Subsection~\ref{subsec:particle_distribution_preliminaries}.  Then the full dissipative Lyapunov estimate
\begin{equation}\label{eq:yukawa_meanfield_full_lyapunov}
\begin{aligned}
 &\sup_{0\le t\le T}\left\{
 \frac1N W_2^2(\rho_N(t),\rho(t)^{\otimes N})
 +\int\!\left[\cF_N^{\kappa,\sharp}(X_N,\rho(t))
       +B\varepsilon_N^\kappa(\eta)\right]\dd\rho_N(t)\right\}\\
 &\qquad
 +\int_0^T\!\int \cD_N^\kappa(X_N,\rho(t))\,\dd\rho_N(t)\,\dd t
 \le C_T\mathfrak E_N^{\kappa,0}(\eta)
\end{aligned}
\end{equation}
holds.  In particular,
\begin{equation}\label{eq:initial_infimum}
 \sup_{0\le t\le T}\frac1N
 W_2^2(\rho_N(t),\rho(t)^{\otimes N})
 \le C_T\mathfrak E_N^{\kappa,0}(\eta).
\end{equation}
For $1\le k\le N$, we also have
\begin{equation}\label{eq:yukawa_meanfield_k_marginal}
 \sup_{0\le t\le T}
 W_2^2(\rho_{N:k}(t),\rho(t)^{\otimes k})
 \le C_Tk\mathfrak E_N^{\kappa,0}(\eta),
\end{equation}
where
\begin{equation}\label{eq:intro_mean_square_force_error}
	\cD_N^\kappa(X_N,\rho)
	:=\frac1N\sum_{i=1}^N
	\left|\frac1N\sum_{j\ne i}\nabla g_{\kappa,d}(x_i-x_j)
	-\nabla g_{\kappa,d}*\rho(x_i)\right|^2.
\end{equation}
The same constant $C_T=C(d,T,\kappa_*,\Lambda_T)$ works for every
$N$, every $1\le k\le N$, and every admissible $\eta$.  In particular,
for a family of limiting solutions with a common value of $\Lambda_T$, the
constant is uniform over $0<\kappa\le\kappa_*$ and has no dependence on
the individual solution beyond that common bound.
\end{theorem}

In particular, if $N\to\infty$ along a sequence for which one can choose
admissible radii $\eta_N$ with
$\mathfrak E_N^{\kappa,0}(\eta_N)\to0$, then for every fixed $k$, we have
\[
 \sup_{0\le t\le T}
 W_2\bigl(\rho_{N:k}(t),\rho(t)^{\otimes k}\bigr)\longrightarrow0.
\]
Thus \eqref{eq:yukawa_meanfield_k_marginal} gives propagation of chaos in the
usual sense of convergence of each fixed marginal, uniformly on $[0,T]$.

Theorem~\ref{thm:yukawa_meanfield} is a weak--strong statement for a
prescribed regular limiting solution.  Section~\ref{sec:mean_field_dynamics}
constructs a local class of such solutions for smooth compactly supported data,
uniformly for $0\le\kappa\le\kappa_*$.  For $0<\kappa\le\kappa_*$, product
initial distributions are then well prepared at order $a_{N,d}$; the solution
corresponding to $\kappa=0$ in the same class is used in the Coulomb limit.

\medskip
\noindent\textbf{Product data and simultaneous mean-field and Coulomb limits.}
The uniform classical theory and the preceding estimates imply the following
consequence, proved in
Corollaries~\ref{cor:yukawa_classical_chaos} and~\ref{cor:classical_simultaneous_yukawa_coulomb}.  If
$\rho_0\in C_c^{1,\beta}(\R^d)$ is a probability density and
$\rho_N^0=\rho_0^{\otimes N}$, then on every compact subinterval
$[0,T]\subset[0,T_*)$ of the common classical existence interval,
\begin{equation}\label{eq:intro_product_yukawa_headline}
\begin{aligned}
 &\sup_{t\le T}\left\{
 \frac1N W_2^2\bigl(\rho_N^\kappa(t),(\rho^\kappa(t))^{\otimes N}\bigr)
 +\int\!\left[\cF_N^{\kappa,\sharp}(X_N,\rho^\kappa(t))
       +B\varepsilon_N^\kappa(\eta_N)\right]\dd\rho_N^\kappa(t)\right\}\\
 &\qquad
 +\int_0^T\!\int\cD_N^\kappa(X_N,\rho^\kappa(t))\,\dd\rho_N^\kappa(t)\,\dd t
 \le C_Ta_{N,d},
\end{aligned}
\end{equation}
uniformly for $0<\kappa\le\kappa_*$.  If $\kappa_N\downarrow0$, then
\begin{equation}\label{eq:intro_simultaneous_headline}
 \sup_{t\le T}\frac1N
 W_2^2\bigl(\rho_N^{\kappa_N}(t),(\rho^{\mathrm C}(t))^{\otimes N}\bigr)
 \le C_T\bigl(a_{N,d}+\alpha_d(\kappa_N)^2\bigr),
\end{equation}
where $\alpha_2(\kappa)=\kappa$ and
$\alpha_d(\kappa)=\kappa^2$ for $d\ge3$.  Hence no relation between the
rates of $N\to\infty$ and $\kappa_N\downarrow0$ is imposed.  Since \eqref{eq:intro_simultaneous_headline} is a bound for the \emph{squared}
Wasserstein distance, the continuum contribution is
$O(\kappa^2)$ in $W_2$ for $d\ge3$, because it is $O(\kappa^4)$ in $W_2^2$.

\subsection{Related work and scope}\label{subsec:intro_related_work}

For globally Lipschitz forces, quantitative mean-field theory goes back to
Braun--Hepp and Dobrushin \cite{BraunHepp1977,Dobrushin1979}; broad accounts
of propagation of chaos and singular mean-field limits include
\cite{Jabin2014,Golse2016,ChaintronDiezI2022,ChaintronDiezII2022}.  For
singular deterministic systems, Hauray and Jabin developed particle
approximations below the Coulomb threshold and treated stronger singularities
with a microscopic cutoff \cite{HaurayJabin2015}.  Relative entropy methods
provide a different route for several stochastic or diffusive singular
systems \cite{JabinWang2018,BreschJabinWang2019,BreschJabinWang2020}.

\paragraph{Coulomb and Riesz modulated energies and commutators.}
The present work follows the modulated energy approach initiated for singular
repulsive gradient flows by Duerinckx and developed for Coulomb and Riesz
interactions by Serfaty and subsequent works
\cite{Duerinckx2016,Serfaty2020,Nguyen2021,Rosenzweig2022Coulomb}.  The
first-order derivative of the modulated energy is naturally a transport
commutator.  Rosenzweig and Serfaty obtained sharp estimates of arbitrary order
for Coulomb and super-Coulomb Riesz energies
\cite{RosenzweigSerfaty2024}, and Hess-Childs, Rosenzweig, and Serfaty proved
the optimal first-order estimate with additive error $N^{s/d-1}$ throughout the Riesz family and
for a class of Riesz-type potentials
\cite{HessChildsRosenzweigSerfaty2025}.  Their later work clarifies the
regularity required of the transport field
\cite{HessChildsRosenzweigSerfaty2026Regularity}.  These are the closest
functional inequalities to Theorem~\ref{thm:yukawa_commutator}.

The Yukawa kernel falls outside these frameworks for two structural reasons.
First, the class in \cite{Nguyen2021} requires, at the Coulomb
exponent, a global Fourier comparison with $|\xi|^{-2}$, whereas
\[
 \widehat g_{\kappa,d}(\xi)=\frac1{|\xi|^2+\kappa^2},
 \qquad
 |\xi|^2\widehat g_{\kappa,d}(\xi)\longrightarrow0
 \quad(|\xi|\downarrow0).
\]
Thus the low-frequency part of the Yukawa energy is genuinely weaker than the
Coulomb energy.  In addition, the Coulomb and sub-Coulomb framework in
\cite{Nguyen2021} uses local superharmonicity in the smearing argument, whereas
away from the origin the Yukawa kernel satisfies
\[
 \Delta g_{\kappa,d}=\kappa^2 g_{\kappa,d}>0.
\]
Some analytic commutator ingredients in that work require only weaker Fourier
upper bounds.  At finite $N$, the renormalized estimate in the natural Yukawa
energy also depends on the smearing argument described above.  Second, the
Riesz-type admissible class of
\cite{HessChildsRosenzweigSerfaty2025} retains Riesz behavior at large scales:
at the Coulomb exponent it is globally comparable to $|x|^{2-d}$ for
$d\ge3$, while $g_{\kappa,d}$ decays exponentially; in dimension two an
admissible logarithmic representative has $-\log|x|+O(1)$ behavior, whereas
every fixed additive normalization of the Yukawa kernel has a finite limit at
infinity.  Proposition~\ref{prop:yukawa_not_riesz_type} in
Appendix~\ref{app:yukawa_riesz_admissible} verifies directly that the Yukawa
kernel does not satisfy the corresponding Riesz-type admissibility hypotheses.

\paragraph{Comparison with a perturbative decomposition around the Coulomb kernel.}
Serfaty's Coulomb theory allows the addition of a sufficiently regular
interaction \cite{Serfaty2020}.  For fixed $\kappa>0$ in dimension two this
can yield a mean-field statement formulated in a Coulomb modulated energy,
whereas Theorem~\ref{thm:yukawa_commutator} closes in the natural Yukawa
modulated energy.  In dimensions $d\ge3$, the Yukawa--Coulomb force correction
lacks the regularity required for that reduction; in every dimension, the
natural Coulomb quadratic form is not controlled by the Yukawa quadratic form.
More precisely, for neutral $\nu$,
the two Fourier weights are $|\xi|^{-2}$ and
$(|\xi|^2+\kappa^2)^{-1}$, whose ratio is unbounded at low frequency.  As
proved in Subsection~\ref{sec:coulomb_perturbation_comparison}, even for fixed
$\kappa>0$ there is no global estimate $I_0[\nu]\le C_\kappa I_\kappa[\nu]$
on smooth compactly supported neutral data.  Our commutator theorem instead closes directly in the Yukawa energy, with a
constant uniform as the screening parameter tends to zero.

\paragraph{Other metrics and the companion paper.}
Nguyen and Serfaty recently introduced a multiscale heat-kernel mollification
metric for a broad class of first-order singular interactions
\cite{NguyenSerfaty2026}.  Their method is not restricted to potential flows
and uses a different control functional, while the estimate here is formulated
in the natural Yukawa modulated energy.  Recent complementary
directions include time-dependent particle weights
\cite{BenPoratCarrilloJabin2026} and the dual-BBGKY derivation of the
second-order two-dimensional Vlasov--Poisson limit without microscopic cutoff
\cite{DuerinckxJabin2026}.

A companion paper by the present authors studies a different problem:
weak--strong stability for Coulomb flows with bounded density and for Riesz
flows \cite{JiangQiaoWuZhangCoulombRiesz2026}.  It combines existing Coulomb
and Riesz commutator estimates with dissipation and a normalized quadratic
transport cost.  In the Coulomb case, the remaining negative force-error term
is used to handle a non-Lipschitz reference velocity.  The present paper
instead assumes a Lipschitz reference velocity in its abstract weak--strong
theorem and develops the uniform Yukawa commutator estimate together with the
modified Helmholtz truncation needed for its proof.  Once this functional inequality is available, we use the same elementary
combination of transport and dissipation estimates as in the companion paper.
Thus the two papers share the dynamical stability argument, while their
analytic inputs are distinct: the companion paper uses existing Coulomb and
Riesz commutator estimates, whereas the present work establishes the uniform
Yukawa commutator estimate for the modified Helmholtz operator.

For the Coulomb endpoint we use the standard Coulomb energy and stress-energy
representations from \cite{Serfaty2020,Duerinckx2016}; the comparison with
sharp finite-$N$ Coulomb estimates is discussed in
Section~\ref{sec:coulomb_limit}.  Surveys and lecture notes on Coulomb and Riesz
gases and modulated energy methods include
\cite{SerfatyLectures2024,Rosenzweig2026Survey}.

\subsection{Strategy of the proof}

The proof follows the standard Coulomb and Riesz strategy of smearing point
charges, choosing microscopic radii, and estimating the first variation; see,
in particular,
\cite{Serfaty2020,Nguyen2021,RosenzweigSerfaty2024,HessChildsRosenzweigSerfaty2025}.
The modified Helmholtz operator changes the truncation source and its total
mass, so the corresponding terms must be controlled directly in the Yukawa
energy.  We first obtain positivity at a common truncation radius and only
then introduce nearest-neighbor radii.  This order avoids a circular use of the field energy associated with
variable truncation radii.

\smallskip
\noindent\emph{\textbf{I.\ Truncation for the modified Helmholtz operator and loss of mass.}}
We truncate the Yukawa potential at radius $r$ by making it constant inside
$B_r$ and compute $L_\kappa g_{\kappa,r}$ exactly.  The source is the sum of
the positive surface and volume measures in
\eqref{eq:intro_smeared_charge_preview}, and its total mass satisfies
\eqref{eq:intro_mass_loss_preview}.  The resulting signed measure is generally
nonneutral but has finite Yukawa energy for every $\kappa>0$, while the
contribution of $1-M_d(\kappa r)$ is controlled uniformly at the truncation
scale.

\smallskip
\noindent\emph{\textbf{II.\ Positivity at a common truncation radius and nearest-neighbor
diagonal correction.}}
Before introducing configuration-dependent radii, we compare the original
modulated energy with a positive field energy at one common truncation radius.
This yields the lower bound required for the subsequent argument without
using the estimate for the diagonal correction at variable radii.  We then set
\[
 r_i=\min\left\{\eta,\frac1{16}\min_{j\ne i}|x_i-x_j|\right\},
\]
so that the enlarged balls remain disjoint.  The corresponding
nearest-neighbor pairs control the diagonal correction term, which yields the
required estimate for the energy of the truncated field without any uniform
lower bound on particle separation.

\smallskip
\noindent\emph{\textbf{III.\ Stress-energy identity and averaging over truncation
radii.}}
For the finite-energy truncated field we use
\[
 T_\kappa[h]
 =2\nabla h\otimes\nabla h-(|\nabla h|^2+\kappa^2h^2)\mathrm{Id}.
\]
The surface part of the truncated source produces a jump in the normal
derivative, and integration by parts across the interface expresses the
surface contribution through the arithmetic mean of the one-sided gradient
traces.  After comparing the first variations of the point and smeared
charges, we average over a common multiplicative scaling of the truncation
radii.  Coarea and Cauchy--Schwarz then convert the surface terms into annular
bulk energy and yield the uniform commutator estimate.

\smallskip
\noindent\emph{\textbf{IV.\ Dissipative weak--strong stability.}}
Along the particle and mean-field dynamics, differentiation of the modulated
energy produces the commutator and the exact negative term $-2\cD_N$.
Differentiation of the normalized quadratic transport cost contributes at most
$+\cD_N$.  Their sum therefore retains $-\cD_N$, which controls the
mean-square discrepancy between empirical and mean-field forces.  Gronwall's inequality then gives Wasserstein stability for the full
$N$-particle law, and symmetry yields the corresponding estimates for each
fixed marginal.

\smallskip
\noindent\emph{\bf V. Uniform classical solutions and the Coulomb limit.}
Uniform kernel bounds yield a common local class of classical solutions for
$0\le\kappa\le\kappa_*$.  To compare the Yukawa and Coulomb solutions, we use
the Coulomb energy and stress-energy identity together with an elementary
$\dot H^{-1}$ chain rule and a quadratic transport estimate.  This gives
$W_2^2+I_0=O(\kappa^2)$ in $d=2$ and $W_2^2+I_0=O(\kappa^4)$ in $d\ge3$, as
well as the corresponding integrated Coulomb field estimate.  In $d=3$, the
bounded pointwise force correction also gives a direct comparison at the
particle level for general symmetric initial laws with finite initial error.

\section{Yukawa kernel, modulated energy, and stress-energy tensor}\label{sec:yukawa_main}

\subsection{Notation and energy conventions}

Symbols such as $\rho$, $\rho^\kappa$, and $\rho^{\mathrm C}$ denote
one-particle probability measures; whenever $L^p$ or H\"older regularity is
assumed, we identify the measure with its density.  By contrast, $\rho_N$
denotes a probability measure on $(\R^d)^N$ and $\rho_{N:k}$ its $k$-particle
marginal.  The quantities $\cF_N$, $\cH_N$, $\cD_N$, and $\cQ_N$ are
finite-$N$ functionals, while $\mathscr F$, $\mathscr D$, and $\mathscr Q$
are reserved for their continuum counterparts in
Section~\ref{sec:coulomb_limit}.  Signed measures or signed densities are
denoted by $\nu$ (with decorations when needed), and $h$ or $\phi$ denotes
the associated scalar potential.

Throughout this section, $0<\kappa\le\kappa_*$ and $g_{\kappa,d}$ denotes the
Yukawa kernel introduced above.  We use the $H^{-1}_\kappa$--$H^1_\kappa$
energy formulation for the commutator estimate.  After fixing the energy
normalization, we establish the background estimates and the stress-energy
identity with interface terms needed for truncated point charges.  The
modulated energy $\cF_N^{g_{\kappa,d}}$ and the commutator
$\cH_N^{g_{\kappa,d}}$ were defined in the Introduction.  We write
\[
  L_\kappa:=-\Delta+\kappa^2,\qquad
  V_\rho^\kappa:=g_{\kappa,d}*\rho,\qquad
  u_\rho^\kappa=-\nabla V_\rho^\kappa.
\]
The particle collision set is
\[
  \Delta_N:=\{X_N\in(\R^d)^N:\exists i\ne j,\ x_i=x_j\}.
\]
The symbol \(\Delta\) is reserved for the Laplacian.  Singular double
integrals are taken off the pair diagonal \(x=y\) for collision-free
configurations, where this agrees with deletion of the labeled terms
\(i=j\).  Whenever lower semicontinuity on the whole configuration space is
needed, the labeled particle--particle sum is used and the energy is set to
\(+\infty\) on \(\Delta_N\), as in
Section~\ref{subsec:particle_distribution_preliminaries}.  The identity matrix is
denoted by \(\mathrm{Id}\).

\begin{remark}[Effect of additive constants on the modulated energy]
\label{rem:additive_shift_bookkeeping}
Let $g$ be an interaction kernel for which the expressions below are finite,
let $c\in\R$ be a constant, let $\rho$ be a non-atomic probability measure,
and let $X_N$ be collision free.  Then the off-diagonal convention gives
\begin{equation}\label{eq:additive_shift_off_diagonal_ledger}
 \cF_N^{g+c}(X_N,\rho)=\cF_N^g(X_N,\rho)-\frac{c}{N}.
\end{equation}
By contrast, if $\nu$ is a finite signed measure of total mass
$m:=\nu(\R^d)$ and the full quadratic integrals are finite, then
\begin{equation}\label{eq:additive_shift_full_quadratic_ledger}
 \iint (g(x-y)+c)\,\dd\nu(x)\dd\nu(y)
 =\iint g(x-y)\,\dd\nu(x)\dd\nu(y)+c m^2.
\end{equation}
For later use, the two identities can be recorded together as
\begin{equation}\label{eq:additive_shift_ledger_summary}
\begin{aligned}
 \cF_N^{g+c}(X_N,\rho)-\cF_N^g(X_N,\rho)
 &= -\frac{c}{N},\\
 \iint (g(x-y)+c)\,\dd\nu(x)\dd\nu(y)
 -\iint g(x-y)\,\dd\nu(x)\dd\nu(y)
 &= c\,m^2,
 \qquad m=\nu(\R^d).
\end{aligned}
\end{equation}
Thus an additive constant changes the off-diagonal modulated energy by
$-c/N$, whereas it changes the full quadratic energy of a finite signed
measure of total mass $m$ by $c m^2$.  The distinction between the two
identities in \eqref{eq:additive_shift_ledger_summary} is essential once
point charges are replaced by Yukawa smeared charges, because the latter have
mass $M_d(\kappa r)<1$.
In particular, for a positive measure $\delta$ of mass $M$, the additive
contribution to its genuine shifted quadratic self-energy is $cM^2$, whereas
the additive contribution to an off-diagonal empirical modulated energy is
$-c/N$.  Indeed, the full quadratic contribution of the constant kernel is
$c\,\nu(\R^d)^2$.  Applied to the neutral signed measure
$\mu_{X_N}-\rho$, this contribution vanishes before the diagonal is deleted;
removing the $N$ labeled atomic diagonal terms subtracts
$c\sum_{i=1}^N N^{-2}=c/N$, which proves
\eqref{eq:additive_shift_off_diagonal_ledger}.
\end{remark}

\begin{remark}
\label{rem:two_dimensional_normalization}
When $d=2$, the unshifted Yukawa kernel $g_{\kappa,2}$ is used in every
$L_\kappa$-source identity, potential truncation, smeared charge, positive
field energy, stress tensor, and force formula.  The additively normalized kernel
$g_{\kappa,2}^{\sharp}=g_{\kappa,2}+c_{\kappa,2}$ is used only to fix the
additive convention in off-diagonal energies and in the two-dimensional
logarithmic Coulomb limit.  Although the added constant does not change the force, the operator
$L_\kappa=-\Delta+\kappa^2$ does not annihilate constants:
\[
 L_\kappa(g_{\kappa,2}+c_{\kappa,2})
 =\delta_0+\kappa^2c_{\kappa,2}\,\dd x.
\]
Accordingly, all source identities, truncated fields, and positivity
arguments use the unshifted kernel $g_{\kappa,2}$.  The additive constant is
introduced only algebraically, after the modified Helmholtz identities have
been established.  Remark~\ref{rem:additive_shift_bookkeeping}
records the corresponding normalization rules: an off-diagonal neutral
empirical energy acquires $-c/N$, whereas the full quadratic energy of a
nonneutral measure of mass $M$ acquires $cM^2$.  Accordingly,
$c_{\kappa,2}M_d(\kappa r)$ is the diagonal correction associated with the
off-diagonal convention, while the constant contribution to the quadratic
self-energy is $c_{\kappa,2}M_d(\kappa r)^2$.  Formula
\eqref{eq:true_shifted_smeared_self_energy} gives the relation between these
quantities explicitly.

All truncation identities below are therefore formulated for $g_{\kappa,2}$;
the superscript $\sharp$ appears only in diagonal correction terms and
off-diagonal energies.  Changing the normalization by a
$\kappa$-independent constant changes the finite-$N$ off-diagonal energy by
only $O(N^{-1})$, which is absorbed by $\varepsilon_N^\kappa(\eta)$.  Theorem~\ref{thm:yukawa_commutator}
requires no moment, compact-support, or tail assumption.  A logarithmic moment
is used only for the logarithmic Coulomb limit with the chosen normalization
in Section~\ref{sec:coulomb_limit}.
\end{remark}

We use the energy norm
\begin{equation}
\label{eq:yukawa_energy_norm}
  \|h\|_{H^1_\kappa}^2
  :=\int_{\R^d}(|\nabla h|^2+\kappa^2|h|^2)\,\dd x.
\end{equation}
For each fixed $\kappa>0$, this norm is equivalent to the standard $H^1$ norm,
with
\[
 \min\{1,\kappa^2\}\|h\|_{H^1}^2
 \le \|h\|_{H^1_\kappa}^2
 \le \max\{1,\kappa^2\}\|h\|_{H^1}^2.
\]
For each fixed $\kappa>0$, $H^{-1}_\kappa$ agrees with the usual $H^{-1}$ as
a set, although the norm-equivalence constants are not uniform as
$\kappa\downarrow0$.  Density and approximation arguments in these spaces are
therefore carried out at fixed $\kappa$.  Every estimate asserted to be
uniform in the screening parameter is written directly in the weighted norms
and does not use this norm equivalence.  A dot over a Sobolev space denotes
its homogeneous version.  Our Fourier transform
convention is
\begin{equation}
\label{eq:global_fourier_convention}
  \widehat f(\xi):=\int_{\R^d}e^{-ix\cdot\xi}f(x)\,\dd x,
  \qquad
  f(x)=(2\pi)^{-d}\int_{\R^d}e^{ix\cdot\xi}\widehat f(\xi)\,\dd\xi.
\end{equation}
The same convention is used for finite measures and tempered distributions.
For $\omega\in\mathscr S'(\R^d)$ such that
$(|\xi|^2+\kappa^2)^{-1/2}\widehat\omega\in L^2(\R^d)$, set
\[
 \|\omega\|_{H^{-1}_\kappa}^2
 :=(2\pi)^{-d}\int_{\R^d}
       \frac{|\widehat\omega(\xi)|^2}{|\xi|^2+\kappa^2}\,\dd\xi.
\]
If \(h=L_\kappa^{-1}\omega\) in the distributional sense, Plancherel's identity gives
\begin{equation}
\label{eq:yukawa_dual_norm}
  \|\omega\|_{H^{-1}_\kappa}^2
  =\|h\|_{H^1_\kappa}^2
  =\langle\omega,h\rangle_{H^{-1}_\kappa,H^1_\kappa}.
\end{equation}
For \(\nu\in H^{-1}_\kappa\), define
\begin{equation}
\label{eq:global_signed_yukawa_energy_convention}
 I_\kappa[\nu]
 :=\langle \nu,L_\kappa^{-1}\nu\rangle_{H^{-1}_\kappa,H^1_\kappa}
 =\|\nu\|_{H^{-1}_\kappa}^2
 =\|g_{\kappa,d}*\nu\|_{H^1_\kappa}^2,
\end{equation}
where the convolution is understood distributionally whenever it is not an
ordinary absolutely convergent integral.
For \(\nu,\omega\in H^{-1}_\kappa\), we define the mutual energy by polarization,
\begin{equation}
\label{eq:global_mutual_yukawa_energy_convention}
 I_\kappa[\nu,\omega]
 :=\frac14\bigl(I_\kappa[\nu+\omega]-I_\kappa[\nu-\omega]\bigr)
 =\langle \nu,L_\kappa^{-1}\omega\rangle.
\end{equation}
A finite signed Radon measure is called a \emph{finite-energy measure} when it
belongs to \(H^{-1}_\kappa\).  In this class the notation
\(\iint g_{\kappa,d}(x-y)\,\dd\nu(x)\dd\omega(y)\) means
\(I_\kappa[\nu,\omega]\) unless absolute convergence is stated explicitly.
For a bounded probability density \(\rho\), this agrees with the ordinary
double integral and
\[
 I_\kappa[\rho]
 =\iint g_{\kappa,d}(x-y)\rho(x)\rho(y)\,\dd x\dd y
 =\int V_\rho^\kappa\,\rho.
\]
Thus $I_\kappa[\cdot]$ always denotes a scalar quadratic energy; its argument
may be a density or, more generally, a finite-energy signed measure.

For the Poisson operator, whenever a neutral tempered distribution $\nu$
satisfies $|\xi|^{-1}\widehat\nu\in L^2$, we use the homogeneous notation
\begin{equation*}
 I_0[\nu]
 :=(2\pi)^{-d}\int_{\R^d}
       \frac{|\widehat\nu(\xi)|^2}{|\xi|^2}\,\dd\xi.
\end{equation*}
When $d\ge3$ and $\rho$ is a bounded probability density, we also write
\begin{equation}\label{eq:global_unscreened_background_convention}
 V_\rho^{\mathrm C}:=g_{0,d}*\rho,
 \qquad
 I_0[\rho]:=\iint g_{0,d}(x-y)\rho(x)\rho(y)\,\dd x\dd y
 =\int V_\rho^{\mathrm C}\,\rho,
\end{equation}
whenever these equivalent finite quantities are used.  In dimension two we
do not use $I_0[\rho]$ or $V_\rho^{\mathrm C}$ for a nonneutral probability density:
the logarithmic potential depends on the chosen additive normalization, and
the neutral homogeneous energy needed later is recorded separately in
Section~\ref{sec:coulomb_limit}.

The estimates below require only a bounded probability density.  For
\(d\ge3\), since \(0\le g_{\kappa,d}\le g_{0,d}\), a decomposition into the
near and far fields gives
\begin{equation}\label{eq:static_background_bound_dge3}
 \|V_\rho^\kappa\|_{L^\infty}+I_\kappa[\rho]
 \le C_d(1+\|\rho\|_{L^\infty}).
\end{equation}
For \(d=2\), the Yukawa background bound is
\eqref{eq:two_dimensional_background_uniform_bounds}.  The additional estimate
\eqref{eq:two_dimensional_normalized_background_energy_bound} is needed only
when the normalized background energy must be controlled uniformly, notably
for well-prepared product initial data and for the Coulomb limit.  In every
dimension,
$
 L_\kappa V_\rho^\kappa=\rho
$
in distributions.  Hence, for every \(1<p<\infty\),
\(V_\rho^\kappa\in W^{2,p}_{\mathrm{loc}}(\R^d)\) by local elliptic regularity,
with local bounds depending only on \(d,p,\kappa_*\), the local \(L^p\) norm
of \(V_\rho^\kappa\), and \(\|\rho\|_\infty\).  Taking \(p>d\) yields the
one-sided \(C^{1,\alpha}\) traces required by the stress-energy identity with interface terms.

All estimates below are uniform for \(0<\kappa\le\kappa_*\) under the
background bounds in Theorem~\ref{thm:yukawa_commutator}.  The unscreened
limit is treated in Section~\ref{sec:coulomb_limit}.  We use
\(\varepsilon_N^\kappa(\eta)\) from \eqref{eq:intro_additive_error}.
Throughout, all sums $\sum_{i\ne j}$ are over \emph{ordered} labeled pairs.
Accordingly, the microscopic interaction energy carries the factor
\(1/(2N^2)\), whereas the particle--particle term in the modulated energy
carries \(1/N^2\); the particle--background contribution carries the factor
\(-2/N\).  The same convention of summing over ordered pairs is used when differentiating the
interaction energy along a transport.

For \(X_N\notin\Delta_N\), a bounded probability density \(\rho\), and a
globally Lipschitz vector field \(u:\R^d\to\R^d\), the commutator in
\eqref{eq:H_def_intro} expands as
\begin{equation}
\label{eq:point_commutator_explicit_lipschitz}
\begin{aligned}
 \cH_N^{g_{\kappa,d}}(X_N,\rho;u)
 &=\frac1{N^2}\sum_{i\ne j}
   (u(x_i)-u(x_j))\cdot\nabla g_{\kappa,d}(x_i-x_j)\\
 &\quad-\frac2N\sum_i\int
   (u(x_i)-u(y))\cdot\nabla g_{\kappa,d}(x_i-y)\rho(y)\,\dd y\\
 &\quad+\iint
   (u(x)-u(y))\cdot\nabla g_{\kappa,d}(x-y)\rho(x)\rho(y)\,\dd x\,\dd y.
\end{aligned}
\end{equation}
Since \(\mu_{X_N}-\rho\) has infinite Yukawa self-energy, the stress-energy
form is applied only after the point charges have been truncated.  The three
terms in \eqref{eq:point_commutator_explicit_lipschitz} are nevertheless
absolutely convergent.  Indeed, for $z\ne0$,
\[
 |u(x+z)-u(x)|\,|\nabla g_{\kappa,d}(z)|
 \le C_{d,\kappa_*}\|\nabla u\|_{L^\infty}
 \bigl(|z|^{2-d}\mathbf1_{\{|z|\le1\}}+\mathbf1_{\{|z|>1\}}\bigr),
\]
which is locally integrable in every fixed dimension $d\ge2$.  A near- and far-field decomposition, using
$\rho\in L^\infty$ near the origin and $\rho\in L^1$ at infinity, also gives
\begin{equation*}
 |\cH_N^{g_{\kappa,d}}(X_N,\rho;u)|
 \le \frac{C\|\nabla u\|_{L^\infty}}{N^2}
       \sum_{i\ne j}\bigl(g_{\kappa,d}(x_i-x_j)+1\bigr)
      +C_{d,\kappa_*}\|\nabla u\|_{L^\infty}
       (1+\|\rho\|_{L^\infty}).
\end{equation*}
In particular, the commutator is defined entirely through off-diagonal
terms.

The finite-energy stress pairing for the smeared charges is introduced next.

\subsection{Stress-energy tensor and interface terms}\label{subsec:stress_energy_interfaces}

\begin{definition}[Stress-energy pairing for finite-energy signed measures]\label{def:lipschitz_yukawa_stress_form}
Let \(0<\kappa\le \kappa_*\), let \(\nu\) be a finite signed measure on
\(\R^d\) with finite Yukawa energy, of arbitrary total mass, and set
\(h=g_{\kappa,d}*\nu\).  At fixed $\kappa>0$, the energy space does not impose a neutrality condition, because the multiplier
\((|\xi|^2+\kappa^2)^{-1}\) is nonsingular at \(\xi=0\).  The Yukawa
stress-energy tensor is
\[
  T_\kappa[h]
  :=2\nabla h\otimes\nabla h-
  \bigl(|\nabla h|^2+\kappa^2h^2\bigr)\mathrm{Id}.
\]
The pointwise quadratic bound gives
\[
 \|T_\kappa[h]\|_{L^1}
 \le C\|h\|_{H^1_\kappa}^2<\infty.
\]
For every globally Lipschitz \(u:\R^d\to\R^d\), we define
\[
\mathcal B_\kappa(\nu;u)
  :=\int_{\R^d}\nabla u(x):T_\kappa[h](x)\dd x.
\]
For smooth compactly supported sources this pairing agrees with the first
variation written as a double integral, as made explicit below.
\end{definition}

The stress-energy representation of the first variation under transport is standard
for Coulomb and Riesz interactions; see, for example,
\cite[equations~(1.23)--(1.25) and Section~4.1]{Serfaty2020}.  For the
modified Helmholtz operator the same calculation contains the additional
zeroth-order term $-\kappa^2h^2\mathrm{Id}$ in the stress tensor.  Thus, if
$f\in C_c^\infty(\R^d)$, $h=g_{\kappa,d}*f$, and
$u:\R^d\to\R^d$ is globally Lipschitz, then
\begin{equation}
\label{eq:smooth_yukawa_stress}
 \iint (u(x)-u(y))\cdot\nabla g_{\kappa,d}(x-y)f(x)f(y)\,\dd x\dd y
 =\int_{\R^d}\nabla u:T_\kappa[h] \,\dd x.
\end{equation}
Indeed, $L_\kappa h=f$ and a direct computation gives
$\nabla\!\cdot T_\kappa[h]=-2f\nabla h$.  Since $f$ is compactly supported,
$h$ and $\nabla h$ have exponential Yukawa decay at infinity, while a globally
Lipschitz $u$ has at most linear growth.  Integration by parts therefore has no
boundary contribution and yields
\[
 \int_{\R^d}\nabla u:T_\kappa[h] \,\dd x
 =2\int_{\R^d}u\cdot\nabla h\,f\,\dd x.
\]
Using $\nabla h(x)=\int\nabla g_{\kappa,d}(x-y)f(y)\,\dd y$ and the oddness
of $\nabla g_{\kappa,d}$, the last expression is exactly the left-hand side of
\eqref{eq:smooth_yukawa_stress}.  Hence, for smooth sources, the stress pairing
coincides with the classical first variation under transport.

The pointwise bound on the stress tensor gives
\begin{equation}
\label{eq:finite_energy_stress_bound}
  |\mathcal B_\kappa(\nu;u)|
  \le C_d\|\nabla u\|_{L^\infty} I_\kappa[\nu].
\end{equation}
Indeed, $\|T_\kappa[h]\|_{L^1}\le C_d\|h\|_{H^1_\kappa}^2$
and $I_\kappa[\nu]=\|h\|_{H^1_\kappa}^2$, so the constant in
\eqref{eq:finite_energy_stress_bound} depends only on the dimension.
Moreover, if $h_n\to h$ strongly in $H^1_\kappa$, then
\[
 \|T_\kappa[h_n]-T_\kappa[h]\|_{L^1}
 \le C_d(\|h_n\|_{H^1_\kappa}+\|h\|_{H^1_\kappa})
       \|h_n-h\|_{H^1_\kappa}\longrightarrow0,
\]
so the stress pairing is continuous in the finite-energy topology.  The smooth
identity \eqref{eq:smooth_yukawa_stress} gives the classical interpretation
of this pairing; the quantitative estimates below use
\eqref{eq:finite_energy_stress_bound} directly and require no approximation
argument.

\subsubsection{Interface traces}

The truncated Yukawa source contains a surface measure.  Its potential is
continuous across each truncation sphere, while the normal derivative has
distinct one-sided traces.  The interface contribution to the stress identity
is expressed through their arithmetic mean.

\begin{definition}[Average of one-sided traces]
\label{def:arithmetic_mean_interface_trace}
Let \(\Sigma\subset\R^d\) be a compact embedded \(C^2\) surface, and let
\(n\) be a unit normal pointing from \(\Omega^-\) to \(\Omega^+\).  Suppose
that \(H|_{\Omega^\pm}\) admits a \(W^{2,p}\)-extension to
\(\overline{\Omega^\pm}\) in a tubular chart for some \(p>d\).  Write
\[
 \gamma^\pm_\Sigma\nabla H\in W^{1-1/p,p}(\Sigma;\R^d)
 \hookrightarrow C^{0,1-d/p}(\Sigma;\R^d)
\]
for the one-sided Sobolev traces and define
\begin{equation*}
 \langle \nabla H\rangle_\Sigma
 :=\frac12\bigl(\gamma^+_\Sigma\nabla H+
                 \gamma^-_\Sigma\nabla H\bigr).
\end{equation*}
Reversing \(n\) exchanges the labels \(+\) and \(-\) and leaves
\(\langle \nabla H\rangle_\Sigma\) unchanged.
\end{definition}

We use the following elementary interface version of the standard
stress-energy integration-by-parts identity.  It identifies the trace
combination associated with the surface component of the truncated Yukawa
source.

\begin{proposition}[Stress-energy identity with interface terms]
\label{prop:average_trace_yukawa_green_stress}
Let \(\Sigma_1,\dots,\Sigma_m\subset\R^d\) be pairwise disjoint compact
embedded \(C^2\) surfaces with unit normals \(n_\ell\).  Let \(p>d\) and
assume
\[
 H\in H^1_\kappa(\R^d)\cap C^0(\R^d),
 \qquad
 H|_\Omega\in W^{2,p}_{\mathrm{loc}}(\Omega)
 \quad\text{for every component }\Omega\subset
 \R^d\setminus\bigcup_{\ell=1}^m\Sigma_\ell,
\]
with \(W^{2,p}\) extensions to both closed sides of each interface.  We write
\[
 \nabla H^\pm:=\gamma_{\Sigma_\ell}^\pm\nabla H,
 \qquad
 \langle \nabla H\rangle:=\frac12(\nabla H^++\nabla H^-).
\]
Assume
\begin{equation}
\label{eq:interface_source_general}
 L_\kappa H
 =f\,\dd x+
 \sum_{\ell=1}^m\lambda_\ell\,
 \mathcal H^{d-1}\lfloor_{\Sigma_\ell},
 \qquad
 f\in L^1_{\mathrm{loc}},
 \qquad
 \lambda_\ell\in L^1(\Sigma_\ell).
\end{equation}
Then we have
\begin{equation}
\label{eq:normal_jump_average_trace}
 \lambda_\ell
 =\partial_{n_\ell}H^- -\partial_{n_\ell}H^+
 \qquad\text{for }\mathcal H^{d-1}\text{-a.e.\ point of }\Sigma_\ell,
\end{equation}
and, for every compactly supported globally Lipschitz vector field \(u:\R^d\to\R^d\),
\begin{equation}
\label{eq:average_trace_stress_identity}
\begin{aligned}
 \int_{\R^d}\nabla u:T_\kappa[H]\,\dd x
 =2\int_{\R^d\setminus\cup_\ell\Sigma_\ell}
 u\cdot\nabla H\,f\,\dd x+2\sum_{\ell=1}^m\int_{\Sigma_\ell}
 u\cdot\langle \nabla H\rangle\,\lambda_\ell\,\dd\mathcal H^{d-1}.
\end{aligned}
\end{equation}
Both \eqref{eq:normal_jump_average_trace} and
\eqref{eq:average_trace_stress_identity} are independent of the choice of
orientation of \(n_\ell\).
\end{proposition}

\begin{proof}
The jump calculation determines the arithmetic mean in
\eqref{eq:average_trace_stress_identity}: if
\(\nabla H^\pm=\tau+a_\pm n\) on an interface and
\(\lambda=a_- -a_+\), then
\[
 T_\kappa[H]^+n-T_\kappa[H]^-n
 =-2\lambda\,\frac{\nabla H^++\nabla H^-}{2}.
\]
Every term in the identity is locally integrable under the stated hypotheses.
Set
\(\alpha:=1-d/p>0\).  The assumed one-sided \(W^{2,p}\)-extensions and
Sobolev embedding give, on a tubular neighborhood of each interface,
\[
 H|_{\overline{\Omega^\pm}}\in C^{1,\alpha},
 \qquad
 \gamma_{\Sigma_\ell}^\pm\nabla H\in C^{0,\alpha}(\Sigma_\ell).
\]
In particular, \(\nabla H\) is bounded on compact subsets of each closed
side of an interface.  Since \(f\in L^1_{\mathrm{loc}}\), for every compact
set \(K\subset\R^d\) one has
\[
 \int_{K\setminus\cup_\ell\Sigma_\ell}|f|\,|\nabla H|\,\dd x<\infty.
\]
Moreover, each \(\Sigma_\ell\) is compact,
\(\lambda_\ell\in L^1(\Sigma_\ell)\), and
\(\langle \nabla H\rangle\in C^{0,\alpha}(\Sigma_\ell)\); hence
$
 \lambda_\ell\langle \nabla H\rangle\in L^1(\Sigma_\ell).
$
Finally, \(H\in H^1_\kappa\) implies
\[
 T_\kappa[H]\in L^1(\R^d),
 \qquad
 \|T_\kappa[H]\|_{L^1}
 \le C\|H\|_{H^1_\kappa}^2.
\]
Thus all volume, surface, and stress pairings below are absolutely
integrable on the support of a compactly supported transport field.

It suffices first to take \(u\in C_c^\infty(\R^d;\R^d)\).  On
\(\Sigma_\ell\), write
\[
 a_\pm:=\partial_{n_\ell}H^\pm,
 \qquad
 \nabla H^\pm=\tau+a_\pm n_\ell,
 \qquad
 \tau\cdot n_\ell=0.
\]
The two one-sided $C^{1,\alpha}$ traces restrict to the same
$C^{1,\alpha}$ function $H|_{\Sigma_\ell}$; differentiating this common trace
tangentially gives the same tangential gradient $\tau$ on both sides.
Testing \eqref{eq:interface_source_general} in a tubular neighborhood and
integrating by parts on the two sides gives
\[
 \int_{\Sigma_\ell}(a_- -a_+)\varphi\,\dd\mathcal H^{d-1}
 =\int_{\Sigma_\ell}\lambda_\ell\varphi\,\dd\mathcal H^{d-1},
\]
hence \eqref{eq:normal_jump_average_trace}.

Away from the interfaces,
\begin{equation}
\label{eq:yukawa_stress_divergence_piecewise}
\begin{aligned}
 \partial_j(T_\kappa[H])_{ij}
 &=2\partial_{ij}H\,\partial_jH+2\partial_iH\,\Delta H
   -\partial_i(|\nabla H|^2+\kappa^2H^2)\\
 &=2(\Delta H-\kappa^2H)\partial_iH
 =-2f\,\partial_iH.
\end{aligned}
\end{equation}
Moreover, we have
\[
 T_\kappa[H]^\pm n_\ell
 =2a_\pm\tau+
 \bigl(a_\pm^2-|\tau|^2-\kappa^2H^2\bigr)n_\ell,
\]
so, by \eqref{eq:normal_jump_average_trace},
\begin{equation}
\label{eq:stress_jump_average_trace}
\begin{aligned}
 T_\kappa[H]^+n_\ell-T_\kappa[H]^-n_\ell
 &=2(a_+-a_-)\tau+(a_+^2-a_-^2)n_\ell\\
 &=-2\lambda_\ell
 \left(\tau+\frac{a_++a_-}{2}n_\ell\right)\\
 &=-2\lambda_\ell\langle \nabla H\rangle.
\end{aligned}
\end{equation}
This computation also fixes the sign and the factor of two in the global
identity.  Indeed,
$n_\ell$ is the outward normal of the ``minus'' side and $-n_\ell$ is the
outward normal of the ``plus'' side.  Hence the two boundary contributions
from a common interface sum to
\[
 u\cdot\bigl(T_\kappa[H]^-n_\ell-T_\kappa[H]^+n_\ell\bigr)
 =2\lambda_\ell\,u\cdot\langle \nabla H\rangle.
\]
Together with $\nabla\cdot T_\kappa[H]=-2f\nabla H$ on each open side, this
produces the two positive terms on the right-hand side of
\eqref{eq:average_trace_stress_identity}.  Reversing $n_\ell$ exchanges the
labels $+$ and $-$ and leaves this summed contribution unchanged.

For every bounded connected component \(\Omega\) of the complement,
\[
 \int_\Omega\nabla u:T_\kappa[H] \,\dd x
 =\int_{\partial\Omega}u\cdot T_\kappa[H]n_\Omega\,\dd\mathcal H^{d-1}
  -\int_\Omega u\cdot\nabla\cdot T_\kappa[H] \,\dd x.
\]
For the unbounded component, choose a ball whose interior contains
\(\operatorname{supp}u\) and apply the same identity to the intersection with
that ball.  The additional outer boundary term vanishes because \(u=0\)
there.  Summing over all components and using
\eqref{eq:yukawa_stress_divergence_piecewise} and
\eqref{eq:stress_jump_average_trace} gives
\eqref{eq:average_trace_stress_identity}.

For compactly supported globally Lipschitz \(u\), choose
\(u_n\in C_c^\infty\) with a common compact support such that
\[
 u_n\to u\ \text{locally uniformly},
 \qquad
 \nabla u_n\to\nabla u\ \text{a.e.},
 \qquad
 \sup_n\bigl(\|u_n\|_{L^\infty}+\|\nabla u_n\|_{L^\infty}\bigr)<\infty.
\]
The integrability established at the beginning of the proof, together with
these uniform bounds, supplies a common integrable majorant for the volume,
surface, and stress terms.  Dominated convergence therefore gives
\eqref{eq:average_trace_stress_identity} for \(u\).
Replacing \(n_\ell\) by \(-n_\ell\) exchanges \(+\) and \(-\) and leaves
both displayed identities unchanged.
\end{proof}

\section{Potential truncation and renormalized energy estimates}\label{sec:truncation_energy}

Throughout this section, each point charge is replaced by the smeared charge
associated with the truncated Yukawa potential.  We first use a common
truncation radius to obtain a lower bound for the modulated energy.  Only after
this positivity has been established do we choose the radii at the
nearest-neighbor scale.  We keep the loss of total smeared mass explicit
rather than rescaling the reference density.  For fixed $\kappa>0$, the
resulting nonneutral signed source belongs to $H^{-1}_\kappa$.  The uniform
estimates never use the zero-mode bound
$\|\cdot\|_{H^{-1}_\kappa}\lesssim\kappa^{-1}\|\cdot\|_{L^2}$.  Instead,
the exact energy expansion and the Green function identity reduce the
background errors to local terms, while a short-range interaction estimate
controls $s_{\kappa,d}^{\sharp}(r_i)$.

\subsection{Potential truncation and smeared charges}

Throughout this section $d\ge2$ is fixed, $0<\kappa\le\kappa_*$,
$L_\kappa=-\Delta+\kappa^2$, and the $H^1_\kappa$ norm is the one fixed in
\eqref{eq:yukawa_energy_norm}.

Unlike the Newtonian kernel, a uniformly charged spherical shell for the
Yukawa interaction does not generate a constant interior potential
\cite[Lemma~3]{CundenFacchiLigaboVivo2019}.  We therefore prescribe a constant
interior potential and apply $L_\kappa=-\Delta+\kappa^2$ to the truncated
potential.  The zeroth-order term $\kappa^2$ then produces a positive volume
charge density inside the ball.

\begin{lemma}[Potential truncation and smeared charges]
\label{lem:yukawa_operator_truncation}
Fix $\kappa>0$.  For \(r>0\), let
\[
 \sigma_r:=(\Sd r^{d-1})^{-1}\mathcal H^{d-1}|_{\partial B_r},
 \qquad
 \sigma_{r,x}:=(y\mapsto x+y)_\#\sigma_r
\]
be the probability surface measure on \(\partial B_r\) and its translate to
\(\partial B(x,r)\).  Define
\begin{equation*}
 g_{\kappa,r}(x):=\begin{cases}
 g_{\kappa,d}(x),&|x|\ge r,\\
 g_{\kappa,d}(r),&|x|<r,
 \end{cases}
 \qquad
 \delta_0^{\kappa,r}:=L_\kappa g_{\kappa,r}.
\end{equation*}
Then \(\delta_0^{\kappa,r}\) is the positive measure with a surface component
and an interior volume component
\begin{equation}\label{eq:general_dimensional_smeared_charge}
 \delta_0^{\kappa,r}
 =A_d(\kappa r)\sigma_r
 +\kappa^2g_{\kappa,d}(r)\mathbf1_{B_r}\,\dd x,
\end{equation}
where
\begin{equation*}
 A_d(z)=\frac{z^{d/2}K_{d/2}(z)}{2^{d/2-1}\Gamma(d/2)},
\end{equation*}
and the total mass of the interior component is
\begin{equation*}
 B_d(z)=\frac{z^{d/2+1}K_{d/2-1}(z)}{2^{d/2}\Gamma(d/2+1)}.
\end{equation*}
Thus
\begin{equation}\label{eq:Md_general}
 \delta_0^{\kappa,r}(\R^d)=M_d(\kappa r):=A_d(\kappa r)+B_d(\kappa r).
\end{equation}
Moreover,
\begin{equation}\label{eq:truncated_charge_mass_loss_identity}
 1-M_d(\kappa r)
 =\kappa^2\int_{B_r}
 \bigl(g_{\kappa,d}(|x|)-g_{\kappa,d}(r)\bigr)\,\dd x\ge0,
\end{equation}
so $0<M_d(\kappa r)\le1$.  The following quadratic upper bound is global:
for every $\kappa>0$ and $r>0$,
\begin{equation}\label{eq:truncated_charge_mass_loss_bound}
 0\le1-M_d(\kappa r)\le \frac{\kappa^2r^2}{2d}.
\end{equation}
As $\kappa r\downarrow0$, we also have
\begin{equation}\label{eq:truncated_charge_mass_loss_asymptotic}
 1-M_d(\kappa r)=\frac{\kappa^2r^2}{2d}+o(\kappa^2r^2)
 \qquad(\kappa r\downarrow0).
\end{equation}
The potential generated by the smeared charge is exactly the truncated potential,
\begin{equation}\label{eq:smeared_charge_resolvent}
 g_{\kappa,d}*\delta_0^{\kappa,r}=g_{\kappa,r},
\end{equation}
and its self-energy is
\begin{equation}\label{eq:general_dimensional_self_energy}
 s_{\kappa,d}(r):=\iint g_{\kappa,d}(x-y)\,\dd\delta_0^{\kappa,r}(x)
 \dd\delta_0^{\kappa,r}(y)
 =M_d(\kappa r)g_{\kappa,d}(r).
\end{equation}
Define the diagonal correction term
\begin{equation}\label{eq:normalized_self_counterterm_definition}
 s_{\kappa,d}^{\sharp}(r)
 :=M_d(\kappa r)g_{\kappa,d}^{\sharp}(r)
 =s_{\kappa,d}(r)+c_{\kappa,d}M_d(\kappa r),
\end{equation}
and the associated correction due to the additive normalization
\begin{equation}\label{eq:additive_normalization_defect_definition}
 \mathfrak b_{\kappa,d}(r)
 :=c_{\kappa,d}\bigl(1-M_d(\kappa r)\bigr).
\end{equation}
For $d\ge3$, $c_{\kappa,d}=0$ and $s_{\kappa,d}^{\sharp}=s_{\kappa,d}$ is
the genuine Yukawa self-energy.  In dimension two, $s_{\kappa,2}^{\sharp}$ is a
\emph{diagonal correction term} associated with the off-diagonal normalization
of $\cF_N^{\kappa,\sharp}$, whereas the quadratic self-energy of
$\delta_0^{\kappa,r}$ computed with the additively normalized kernel is given
by
\begin{equation}\label{eq:true_shifted_smeared_self_energy}
 \iint g_{\kappa,d}^{\sharp}(x-y)\,
   \dd\delta_0^{\kappa,r}(x)\dd\delta_0^{\kappa,r}(y)
 =s_{\kappa,d}(r)+c_{\kappa,d}M_d(\kappa r)^2
 =s_{\kappa,d}^{\sharp}(r)
  -M_d(\kappa r)\mathfrak b_{\kappa,d}(r).
\end{equation}
Thus the linear factor $M_d$ in
\eqref{eq:normalized_self_counterterm_definition} gives the diagonal correction
compatible with omitting the particle self-interactions, while the full
quadratic self-energy retains the factor $M_d^2$ displayed in
\eqref{eq:true_shifted_smeared_self_energy}.
Moreover,
\eqref{eq:truncated_charge_mass_loss_bound} yields
\begin{equation}\label{eq:additive_normalization_defect_bound}
 |\mathfrak b_{\kappa,d}(r)|
 \le \frac{|c_{\kappa,d}|\kappa^2r^2}{2d}
 \le C_{d,\kappa_*}r^2
 \qquad(0<\kappa\le\kappa_*),
\end{equation}
where the second inequality uses
$\sup_{0<\kappa\le\kappa_*}\kappa^2|\log\kappa|<\infty$ in dimension two.
In particular, for every fixed $r>0$,
$\mathfrak b_{\kappa,2}(r)\to0$ as $\kappa\downarrow0$.  None of these
normalization identities changes the source or force formulas.
For every fixed $d\ge3$ there are $0<c_d<C_d<\infty$ such that
\begin{equation}\label{eq:general_dimensional_self_energy_power_bounds}
 c_dr^{2-d}\le s_{\kappa,d}(r)\le C_dr^{2-d}
 \qquad(0<\kappa r\le1),
\end{equation}
and, when $\kappa r\le1/32$, we have
\begin{equation}\label{eq:general_dimensional_self_energy_close_pair}
 s_{\kappa,d}(r)\le C_d g_{\kappa,d}(16r).
\end{equation}
The translated smeared charge is denoted by
$\delta_x^{\kappa,r}:=\delta_0^{\kappa,r}(\cdot-x)$.
\end{lemma}

\begin{proof}
The source formula follows from a direct distributional computation with all
dimension-dependent factors kept explicit.  Set
\[
 g:=g_{\kappa,d},\qquad
 \Omega_r:=\R^d\setminus\overline{B_r},\qquad
 n(x):=\frac{x}{|x|}\quad(x\in\partial B_r).
\]
\medskip
\noindent\emph{Step 1: the first distributional derivative of the truncated potential.}
Let $\mathbf u\in C_c^\infty(\R^d;\R^d)$.  By the definition of the
distributional gradient and by the fact that $g_{\kappa,r}$ is constant on
$B_r$ and equal to $g$ on $\Omega_r$,
\begin{align*}
 \langle\nabla g_{\kappa,r},\mathbf u\rangle
 =-\int_{\R^d}g_{\kappa,r}\,\operatorname{div}\mathbf u\,\dd x\notag=-g(r)\int_{B_r}\operatorname{div}\mathbf u\,\dd x
   -\int_{\Omega_r}g\,\operatorname{div}\mathbf u\,\dd x.
\end{align*}
The divergence theorem in the ball gives
\begin{equation}
 -g(r)\int_{B_r}\operatorname{div}\mathbf u\,\dd x
 =-g(r)\int_{\partial B_r}\mathbf u\cdot n\,\dd\mathcal H^{d-1}.
 \label{eq:general_trunc_distributional_gradient_inside}
\end{equation}
Choose $R>r$ so large that $\operatorname{supp}\mathbf u\subset B_R$ and
integrate by parts on $B_R\setminus\overline{B_r}$.  The outer boundary term
vanishes, whereas the outward unit normal of the exterior domain $\Omega_r$
on its inner boundary is $-n$.  Hence
\begin{align}
 -\int_{\Omega_r}g\,\operatorname{div}\mathbf u\,\dd x
 &=\int_{\Omega_r}\nabla g\cdot\mathbf u\,\dd x
   -\int_{\partial\Omega_r}g\,\mathbf u\cdot\nu_{\Omega_r}\,
       \dd\mathcal H^{d-1}\notag\\
 &=\int_{\Omega_r}\nabla g\cdot\mathbf u\,\dd x
   +g(r)\int_{\partial B_r}\mathbf u\cdot n\,\dd\mathcal H^{d-1}.
 \label{eq:general_trunc_distributional_gradient_outside}
\end{align}
The two boundary terms in
\eqref{eq:general_trunc_distributional_gradient_inside} and
\eqref{eq:general_trunc_distributional_gradient_outside} cancel exactly.
Therefore
\begin{equation}\label{eq:truncated_potential_weak_gradient}
 \nabla g_{\kappa,r}=\mathbf1_{\Omega_r}\nabla g
 \qquad\text{in }\mathcal D'(\R^d;\R^d).
\end{equation}
Thus the continuity of the truncated potential prevents a surface measure from
appearing in its first distributional derivative.  The surface measure appears
only after differentiating once more, through the jump of the normal
derivative.

\medskip
\noindent\emph{Step 2: the surface and volume parts of the smeared charge and their masses.}
Let $\varphi\in C_c^\infty(\R^d)$.  From
\eqref{eq:truncated_potential_weak_gradient},
\begin{align}
 \big\langle L_\kappa g_{\kappa,r},\varphi\big\rangle
 &=\int_{\Omega_r}\nabla g\cdot\nabla\varphi\,\dd x
   +\kappa^2\int_{\Omega_r}g\varphi\,\dd x
   +\kappa^2g(r)\int_{B_r}\varphi\,\dd x.
 \label{eq:general_trunc_distributional_operator_start}
\end{align}
Since the origin does not belong to $\Omega_r$,
$L_\kappa g=(-\Delta+\kappa^2)g=0$ classically on $\Omega_r$.  Integrating
the first two terms in \eqref{eq:general_trunc_distributional_operator_start}
by parts gives
\begin{align*}
 \int_{\Omega_r}\bigl(\nabla g\cdot\nabla\varphi+\kappa^2g\varphi\bigr)
   \,\dd x\notag&=\int_{\Omega_r}(-\Delta g+\kappa^2g)\varphi\,\dd x
   +\int_{\partial B_r}\varphi\,\partial_{\nu_{\Omega_r}}g\,
       \dd\mathcal H^{d-1}\notag\\
 &=-g'(r)\int_{\partial B_r}\varphi\,\dd\mathcal H^{d-1},
\end{align*}
because $\nu_{\Omega_r}=-n$ on $\partial B_r$.  Substitution into
\eqref{eq:general_trunc_distributional_operator_start} yields the exact
source identity with surface measure,
\begin{equation}
 L_\kappa g_{\kappa,r}
 =-g_{\kappa,d}'(r)\,\mathcal H^{d-1}\lfloor_{\partial B_r}
  +\kappa^2g_{\kappa,d}(r)\mathbf1_{B_r}\,\dd x.
 \label{eq:general_trunc_source_area_measure}
\end{equation}
Both pieces are nonnegative because $g_{\kappa,d}>0$ and
$g_{\kappa,d}'<0$.

By the radial Bessel derivative formula \eqref{eq:yukawa_force_general_intro},
with $z=\kappa r$,
\begin{align*}
 -|\mathbb S^{d-1}|r^{d-1}g_{\kappa,d}'(r)
 =|\mathbb S^{d-1}|(2\pi)^{-d/2}
   \kappa^{d/2}r^{d/2}K_{d/2}(z)\notag=\frac{z^{d/2}K_{d/2}(z)}{2^{d/2-1}\Gamma(d/2)}
 =A_d(z).
\end{align*}
Since
$\sigma_r=(|\mathbb S^{d-1}|r^{d-1})^{-1}
\mathcal H^{d-1}\lfloor_{\partial B_r}$, the first term in
\eqref{eq:general_trunc_source_area_measure} is precisely
$A_d(\kappa r)\sigma_r$.

The volume charge density in the ball is the constant $\kappa^2g_{\kappa,d}(r)$.  Its mass is
\begin{align*}
 \kappa^2g_{\kappa,d}(r)|B_r|
 &=\kappa^2(2\pi)^{-d/2}
   \left(\frac\kappa r\right)^{d/2-1}K_{d/2-1}(z)
   \frac{|\mathbb S^{d-1}|}{d}r^d\notag\\
 &=\frac{z^{d/2+1}K_{d/2-1}(z)}{2^{d/2}\Gamma(d/2+1)}
 =B_d(z),
\end{align*}
where we used
$|\mathbb S^{d-1}|=2\pi^{d/2}/\Gamma(d/2)$ and
$\Gamma(d/2+1)=(d/2)\Gamma(d/2)$.  This proves
\eqref{eq:general_dimensional_smeared_charge}--\eqref{eq:Md_general}, with
$M_d(\kappa r)=A_d(\kappa r)+B_d(\kappa r)$.

\medskip
\noindent\emph{Step 3: the loss of total mass, its quadratic bound, and its leading coefficient.}
Set
$
 w:=g_{\kappa,d}-g_{\kappa,r}.
$
The radial monotonicity of the Yukawa kernel gives $w\ge0$,
$\operatorname{supp}w\subset\overline{B_r}$, and
$w(x)=g_{\kappa,d}(|x|)-g_{\kappa,d}(r)$ for $|x|<r$.  Moreover,
\begin{equation}
 L_\kappa w=\delta_0-\delta_0^{\kappa,r}
 \qquad\text{in }\mathcal D'(\R^d).
 \label{eq:truncated_charge_mass_difference_equation}
\end{equation}
Choose $\chi\in C_c^\infty(\R^d)$ equal to one on a neighborhood of
$\overline{B_r}$.  Testing \eqref{eq:truncated_charge_mass_difference_equation} against
$\chi$ gives, on the one hand,
$
 \big\langle L_\kappa w,\chi\big\rangle=1-M_d(\kappa r),
$
and, on the other hand, because $\nabla\chi=0$ on $\operatorname{supp}w$,
$
 \big\langle L_\kappa w,\chi\big\rangle
 =\kappa^2\int_{B_r}w(x)\,\dd x.
$
Thus
\begin{equation}
 1-M_d(\kappa r)
 =\kappa^2\int_{B_r}
   \bigl(g_{\kappa,d}(|x|)-g_{\kappa,d}(r)\bigr)\,\dd x,
 \label{eq:truncated_charge_mass_identity_detailed}
\end{equation}
which is \eqref{eq:truncated_charge_mass_loss_identity}.  In particular
$M_d(\kappa r)<1$ for $\kappa>0$ and $r>0$.

The source formula from Step~2 and the mass identity above also give a useful
comparison with the Coulomb kernel.  Indeed,
$A_d(\kappa s)\le M_d(\kappa s)\le1$, and hence
\begin{equation}\label{eq:yukawa_force_bound}
 0\le -g_{\kappa,d}'(s)
 =\frac{A_d(\kappa s)}{|\mathbb S^{d-1}|s^{d-1}}
 \le \frac1{|\mathbb S^{d-1}|s^{d-1}}
 =-g_{0,d}'(s),\qquad s>0.
\end{equation}
Here the last identity is valid for the Newtonian kernel when $d\ge3$ and
for the logarithmic kernel when $d=2$.  Therefore, for $0<s<r$, we obtain
\[
 0\le g_{\kappa,d}(s)-g_{\kappa,d}(r)
 \le g_{0,d}(s)-g_{0,d}(r).
\]
A direct radial calculation, valid for every $d\ge2$, yields
\begin{equation}\label{eq:coulomb_truncation_integral}
 \int_{B_r}\bigl(g_{0,d}(|x|)-g_{0,d}(r)\bigr)\,\dd x
 =\frac{r^2}{2d}.
\end{equation}
For $d\ge3$ this is the usual Newtonian computation, while for $d=2$ it is
$\int_0^r s\log(r/s)\,\dd s=r^2/4$.  Combining the last two displays with
\eqref{eq:truncated_charge_mass_identity_detailed} proves the global bound
\eqref{eq:truncated_charge_mass_loss_bound} simultaneously in all dimensions.

The same comparison also gives the leading coefficient.  Set $z=\kappa r$
and, after the change of variables $x=ry$, define
\[
 F_z(y):=r^{d-2}
 \bigl(g_{\kappa,d}(r|y|)-g_{\kappa,d}(r)\bigr),
 \qquad 0<|y|<1.
\]
With $\nu_d=d/2-1$, the Bessel representation gives the single formula
\[
 F_z(y)=(2\pi)^{-d/2}z^{\nu_d}
 \left(|y|^{-\nu_d}K_{\nu_d}(z|y|)-K_{\nu_d}(z)\right).
\]
The small-argument asymptotics of $K_{\nu_d}$ (including
$K_0(\zeta)=-\log(\zeta/2)-\gamma_E+o(1)$ when $d=2$) imply
\[
 F_z(y)\longrightarrow g_{0,d}(y)-g_{0,d}(1)
 \qquad (z\downarrow0)
\]
for every $0<|y|<1$.  The preceding Yukawa--Coulomb comparison gives the
integrable domination
$0\le F_z(y)\le g_{0,d}(y)-g_{0,d}(1)$.  Dominated convergence and
\eqref{eq:coulomb_truncation_integral} at $r=1$ therefore yield
\[
 \lim_{z\downarrow0}\frac{1-M_d(z)}{z^2}=\frac1{2d},
\]
which proves \eqref{eq:truncated_charge_mass_loss_asymptotic} for every $d\ge2$.

\medskip
\noindent\emph{Step 4: the Green function identity and the exact self-energy.}
By the source formula, $\delta_0^{\kappa,r}$ is a finite positive measure
supported in $\overline{B_r}$.  The truncated potential belongs to
$H^1_\kappa(\R^d)$: the singularity at the origin has been removed, while the
Yukawa tail decays exponentially.  Since
$L_\kappa g_{\kappa,r}=\delta_0^{\kappa,r}$ in distributions, we also have
$\delta_0^{\kappa,r}\in H^{-1}_\kappa(\R^d)$.  Consequently,
$L_\kappa^{-1}\delta_0^{\kappa,r}
=g_{\kappa,d}*\delta_0^{\kappa,r}$ belongs to $H^1_\kappa(\R^d)$, and both
this potential and $g_{\kappa,r}$ solve
$
 L_\kappa h=\delta_0^{\kappa,r}.
$
Their difference $H$ satisfies $L_\kappa H=0$ in $\R^d$ and
$H\in H^1_\kappa(\R^d)$.  Testing against $H$ gives
\[
 \int_{\R^d}\bigl(|\nabla H|^2+\kappa^2|H|^2\bigr)\,\dd x=0,
\]
so $H=0$.  Hence
$
 g_{\kappa,d}*\delta_0^{\kappa,r}=g_{\kappa,r},
$
which is \eqref{eq:smeared_charge_resolvent}.

Tonelli's theorem applies because both the kernel and the source are
nonnegative.  Since $g_{\kappa,r}$ is constant with value $g_{\kappa,d}(r)$
on the support of $\delta_0^{\kappa,r}$,
\begin{align*}
 s_{\kappa,d}(r)
 =\int_{\R^d}(g_{\kappa,d}*\delta_0^{\kappa,r})(x)
   \,\dd\delta_0^{\kappa,r}(x)=\int_{\R^d}g_{\kappa,r}(x)\,\dd\delta_0^{\kappa,r}(x)
 =M_d(\kappa r)g_{\kappa,d}(r),
\end{align*}
which proves \eqref{eq:general_dimensional_self_energy}.

\medskip
\noindent\emph{Step 5: power-law self-energy bounds in dimensions $d\ge3$.}
Assume first $d\ge3$ and set $z=\kappa r$.  The exact scaled identity is
\begin{equation*}
 r^{d-2}g_{\kappa,d}(r)
 =(2\pi)^{-d/2}z^{d/2-1}K_{d/2-1}(z).
\end{equation*}
The right-hand side extends continuously to $z=0$ with a strictly positive
limit and is positive on $[0,1]$.  Together with the continuity,
positivity, and convergence $M_d(z)\to1$, this proves
\eqref{eq:general_dimensional_self_energy_power_bounds}.  If
$\kappa r\le1/32$, then both $\kappa r$ and $16\kappa r$ remain in a fixed
compact subset of the small-argument regime.  The same Bessel bounds imply
$g_{\kappa,d}(r)\le C_dg_{\kappa,d}(16r)$; since $M_d\le1$, this proves
\eqref{eq:general_dimensional_self_energy_close_pair}.

\end{proof}

\begin{lemma}[Exterior potential identity for the truncated Yukawa charge]
\label{lem:yukawa_smeared_exterior_exactness}
Let $d\ge2$, $\kappa>0$, and $r>0$.  For every $x,y\in\R^d$ with
$|x-y|\ge r$,
\begin{equation}\label{eq:yukawa_smeared_exterior_exactness}
 \int_{\R^d} g_{\kappa,d}(y-z)\,\dd\delta_x^{\kappa,r}(z)
 =g_{\kappa,d}(y-x).
\end{equation}
Consequently, if $r_i,r_j>0$ and
$|x_i-x_j|\ge r_i+r_j$, then
\begin{equation}\label{eq:yukawa_disjoint_smeared_mutual_energy}
 I_\kappa[\delta_{x_i}^{\kappa,r_i},\delta_{x_j}^{\kappa,r_j}]
 =g_{\kappa,d}(x_i-x_j).
\end{equation}
In particular, the loss of total mass
$M_d(\kappa r)<1$ does not alter the exterior Yukawa potential of the
truncated charge.
\end{lemma}

\begin{proof}
By the Green function identity \eqref{eq:smeared_charge_resolvent} and translation
invariance,
\[
 g_{\kappa,d}*\delta_x^{\kappa,r}
 =g_{\kappa,r}(\cdot-x).
\]
If $|x-y|\ge r$, the definition of the truncated potential gives
$g_{\kappa,r}(y-x)=g_{\kappa,d}(y-x)$, proving
\eqref{eq:yukawa_smeared_exterior_exactness}.

Now assume $|x_i-x_j|\ge r_i+r_j$.  If
$z\in\operatorname{supp}\delta_{x_i}^{\kappa,r_i}\subset
\overline{B(x_i,r_i)}$, then
$|z-x_j|\ge |x_i-x_j|-r_i\ge r_j$.  Hence, using
\eqref{eq:yukawa_smeared_exterior_exactness} first for the $j$th charge,
\[
 I_\kappa[\delta_{x_i}^{\kappa,r_i},\delta_{x_j}^{\kappa,r_j}]
 =\int g_{\kappa,d}(z-x_j)\,\dd\delta_{x_i}^{\kappa,r_i}(z).
\]
Applying \eqref{eq:yukawa_smeared_exterior_exactness} once more, now to the
$i$th charge evaluated at $x_j$, gives
\[
 \int g_{\kappa,d}(z-x_j)\,\dd\delta_{x_i}^{\kappa,r_i}(z)
 =g_{\kappa,d}(x_i-x_j),
\]
which proves \eqref{eq:yukawa_disjoint_smeared_mutual_energy}.  No shell theorem
or normalization by the total mass of the smeared charge is used; the result
follows directly from the exact Green function identity for the modified
Helmholtz operator.
\end{proof}

A direct consequence of the Green function identity and the truncation integral in
Lemma~\ref{lem:yukawa_operator_truncation} is the background truncation estimate:
for every bounded probability density $\rho$, every $x\in\R^d$, and every $r>0$,
\begin{equation}\label{eq:background_truncation_bound}
 \begin{aligned}
 0\le V_\rho^\kappa(x)-\int V_\rho^\kappa\,\dd\delta_x^{\kappa,r}
 &=\int_{B(x,r)}\bigl(g_{\kappa,d}(x-y)-g_{\kappa,d}(r)\bigr)\rho(y)\,\dd y
 \le\frac{\|\rho\|_\infty}{2d}r^2.
 \end{aligned}
\end{equation}

\medskip
\noindent\emph{The logarithmic normalization in dimension two.}
With the convention fixed in Remark~\ref{rem:two_dimensional_normalization},
\begin{equation}\label{eq:renormalized_2d_limit}
 g_{\kappa,2}^{\sharp}(x)\longrightarrow
 g_{0,2}(x)
 =-\frac1{2\pi}\log|x|+\frac{\log2-\gamma_E}{2\pi}
 \qquad(\kappa\downarrow0)
\end{equation}
locally away from the origin.  The standard bounds for $K_0$ also give the
global majorant
\begin{equation}\label{eq:renormalized_2d_global_majorant}
 \sup_{0<\kappa\le\kappa_*}
 |g_{\kappa,2}^{\sharp}(x)|
 \le C_{\kappa_*}\bigl(1+|\log|x||\bigr),
 \qquad x\ne0.
\end{equation}
Indeed, when $\kappa|x|\le1$ the logarithms of $\kappa$ cancel, while when
$\kappa|x|\ge1$ one uses the decay of $K_0$ together with
$|\log\kappa|\le |\log|x||+C_{\kappa_*}$.  All source, truncation, and
stress-energy identities continue to use the unshifted kernel
$g_{\kappa,2}$.

\medskip
\noindent\emph{Consequences at the preliminary truncation scale.}
Fix $d\ge2$ and $\kappa_*>0$.  Choose a preliminary truncation scale
$\eta_{\mathrm{se}}=\eta_{\mathrm{se}}(d,\kappa_*)>0$ small enough that the
self-energy estimates in Lemma~\ref{lem:yukawa_operator_truncation} and the
small-argument asymptotics of $K_0$ in dimension two give constants
$0<c<C<\infty$ such that, uniformly for $0<\kappa\le\kappa_*$ and
$0<r\le32\eta_{\mathrm{se}}$,
\begin{equation}\label{eq:microscopic_self_energy_bound}
 c\,\ell_d(r)\le s_{\kappa,d}^{\sharp}(r)\le C\,\ell_d(r).
\end{equation}
The constants in this comparison depend only on $d$ and $\kappa_*$.  Indeed,
for $d\ge3$ this is \eqref{eq:general_dimensional_self_energy_power_bounds},
while for $d=2$ it follows from
$K_0(z)+\log z\to\log2-\gamma_E$ and $M_2(z)\to1$ as $z\downarrow0$.
For $d\ge3$, we also require
$32\kappa_*\eta_{\mathrm{se}}\le1$; then
\eqref{eq:general_dimensional_self_energy_close_pair} gives
\begin{equation}\label{eq:close_pair_self_energy_bound}
 s_{\kappa,d}(r)\le C_d g_{\kappa,d}(16r)
 \qquad(0<r\le\eta_{\mathrm{se}}).
\end{equation}
The final truncation scale $\eta_0$ used in the commutator theorem is chosen
below, after fixing the constants in the common-radius and short-range
estimates.  This separates the small-argument Bessel asymptotics from the
later geometric restrictions.

\medskip
\noindent\emph{Two-dimensional background bounds.}
Fix $\kappa_*>0$.  Uniformly for $0<\kappa\le\kappa_*$, every bounded
probability density $\rho$ satisfies
\begin{equation}\label{eq:two_dimensional_background_uniform_bounds}
 \|V_\rho^\kappa\|_{L^\infty}+I_\kappa[\rho]
 \le C_{\kappa_*}\bigl(1+\|\rho\|_{L^\infty}\bigr)
       (1+|\log\kappa|).
\end{equation}
If in addition
\[
 \int_{\R^2}\log(2+|x|)\rho(x)\,\dd x<\infty,
\]
then
\begin{equation}\label{eq:two_dimensional_normalized_background_energy_bound}
 |I_{g_{\kappa,2}^{\sharp}}[\rho]|
 \le C_{\kappa_*}\left(
 1+\|\rho\|_{L^\infty}
 +\int_{\R^2}\log(2+|x|)\rho(x)\,\dd x\right).
\end{equation}
The first estimate requires no moment assumption.  Both bounds follow from the
standard small- and large-argument estimates for $K_0$; we include the short
argument for completeness.

The bounds for $K_0$ imply, uniformly for $0<\kappa\le\kappa_*$,
\[
 0\le K_0(\kappa r)
 \le C_{\kappa_*}\bigl(1+|\log\kappa|+|\log r|\bigr)
 \quad(0<r\le1),
\]
and $K_0(\kappa r)$ is bounded on $r\ge1$ by a constant times
$1+|\log\kappa|$, with exponential decay once $\kappa r\ge1$.  Splitting
at unit distance therefore yields, uniformly in $x$,
\[
 |V_\rho^\kappa(x)|
 \le C_{\kappa_*}\left[
 (1+|\log\kappa|)\|\rho\|_{L^1}
 +\|\rho\|_{L^\infty}
   \int_{|z|\le1}|\log|z||\,\dd z\right].
\]
This proves the potential bound, and
$I_\kappa[\rho]\le\|V_\rho^\kappa\|_{L^\infty}$ because $\rho$ is a
probability density.

For the normalized kernel, the uniform majorant
\eqref{eq:renormalized_2d_global_majorant} gives an integrable logarithmic
singularity near the diagonal.  In the far field,
$|\log|x-y||\le C+\log(2+|x|)+\log(2+|y|)$, so the assumed logarithmic
moment controls the double integral.  This proves
\eqref{eq:two_dimensional_normalized_background_energy_bound}.

\smallskip
The regularity required to apply Proposition~\ref{prop:average_trace_yukawa_green_stress}
to the truncated fields follows directly from the source formula
\eqref{eq:general_dimensional_smeared_charge}; we verify it once, at the point
where the interface identity is first used in the commutator argument.

The next lemma compares the modulated energy of the point configuration with
the energy obtained by smearing all particles at a common truncation radius.  It is
proved before the individual radii are chosen from nearest-neighbor distances,
so the lower bound for the modulated energy, up to the additive error, does not
rely on those distances.  No disjointness of the balls is assumed here: the
interaction between two smeared charges is represented by $G_{\kappa,R}$, and
the argument remains valid even when $R$ exceeds some interparticle distances.

\begin{lemma}\label{lem:common_truncation_energy_comparison}
Fix $d\ge2$ and $\kappa_*>0$.  There exists $r_*\in(0,1]$,
depending only on $d$ and $\kappa_*$, with the following property.  Let $0<\kappa\le\kappa_*$, $0<R\le r_*$, let $\rho$ be a
bounded probability density, and let $X_N\notin\Delta_N$.  Define
\[
 G_{\kappa,R}(z):=\iint g_{\kappa,d}(z+a-b)\,
 \dd\delta_0^{\kappa,R}(a)\dd\delta_0^{\kappa,R}(b),
 \qquad M_R:=M_d(\kappa R),
\]
\[
 \mathcal E_R:=I_\kappa\!\left[
 \frac1N\sum_i\delta_{x_i}^{\kappa,R}-\rho\right]\ge0,
 \qquad
 L_R:=\frac1{N^2}\sum_{i\ne j}
 \bigl(g_{\kappa,d}-G_{\kappa,R}\bigr)(x_i-x_j).
\]
Then
\begin{equation}\label{eq:common_truncation_kernel_order}
 0\le G_{\kappa,R}\le g_{\kappa,R}\le g_{\kappa,d},
 \qquad L_R\ge0.
\end{equation}
Writing
\[
 J_i^R:=\int V_\rho^\kappa\,\dd\delta_{x_i}^{\kappa,R},
 \qquad V_i:=V_\rho^\kappa(x_i),
\]
one has the exact identity
\begin{equation}\label{eq:common_truncation_energy_identity}
\begin{aligned}
 L_R
 =\cF_N^{\kappa,\sharp}(X_N,\rho)-\mathcal E_R
   +\frac{s_{\kappa,d}^{\sharp}(R)}N
   +\frac{\mathfrak b_{\kappa,d}(R)}N
   +\frac2N\sum_i\bigl(V_i-J_i^R\bigr).
\end{aligned}
\end{equation}
Moreover, with $C=C(d,\kappa_*,\|\rho\|_{L^\infty})$,
\begin{equation}\label{eq:common_truncation_energy_bound}
 0\le L_R\le
 \cF_N^{\kappa,\sharp}(X_N,\rho)
 +C\left(\frac{\ell_d(R)}N+R^2+\kappa^2R^2\right).
\end{equation}
In particular, the expression on the right-hand side is nonnegative.
\end{lemma}

\begin{proof}
The Green function identity and positivity of the smeared charge give
$G_{\kappa,R}=g_{\kappa,R}*\delta_0^{\kappa,R}\ge0$.  Since
$g_{\kappa,R}\le g_{\kappa,d}$ and $\delta_0^{\kappa,R}$ is positive,
\[
 G_{\kappa,R}=g_{\kappa,R}*\delta_0^{\kappa,R}
 \le g_{\kappa,d}*\delta_0^{\kappa,R}
 =g_{\kappa,R}\le g_{\kappa,d},
\]
which proves \eqref{eq:common_truncation_kernel_order}.

Expanding the (generally non-neutral) energy of the smeared signed measure gives
\[
 \mathcal E_R
 =\frac1{N^2}\sum_{i\ne j}G_{\kappa,R}(x_i-x_j)
  +\frac{s_{\kappa,d}(R)}N
  -\frac2N\sum_iJ_i^R+I_\kappa[\rho].
\]
This energy is finite because each smeared charge and the bounded background
belong to $H^{-1}_\kappa$ at fixed $\kappa>0$.  Subtract the last identity from the definition of
$\cF_N^{g_{\kappa,d}}$.  Using
$\cF_N^{\kappa,\sharp}=\cF_N^{g_{\kappa,d}}-c_{\kappa,d}/N$,
$s_{\kappa,d}=s_{\kappa,d}^{\sharp}-c_{\kappa,d}M_R$, and
\eqref{eq:additive_normalization_defect_definition} gives
\eqref{eq:common_truncation_energy_identity}.  Thus the only correction from
the additive normalization that remains after accounting for the omitted
particle self-interactions is $\mathfrak b_{\kappa,d}(R)/N$.

The background truncation estimate \eqref{eq:background_truncation_bound} gives
\[
 0\le V_i-J_i^R\le C\|\rho\|_{L^\infty}R^2,
 \qquad
 0\le1-M_R\le\frac{\kappa^2R^2}{2d}.
\]
By \eqref{eq:additive_normalization_defect_bound} and the preceding
background estimate,
\[
 \frac{|\mathfrak b_{\kappa,d}(R)|}{N}
 +\frac2N\sum_i|V_i-J_i^R|
 \le C R^2
 \le C(R^2+\kappa^2R^2).
\]
Choose
\begin{equation}\label{eq:common_radius_scale_choice}
 0<r_*\le\min\{1,\eta_{\mathrm{se}}\}
\end{equation}
depending only on $d$ and $\kappa_*$, so that
\eqref{eq:microscopic_self_energy_bound} applies throughout $0<R\le r_*$.  Then
$s_{\kappa,d}^{\sharp}(R)\le C\ell_d(R)$ for $R\le r_*$.  Dropping the
nonnegative energy $\mathcal E_R$ in
\eqref{eq:common_truncation_energy_identity} proves
\eqref{eq:common_truncation_energy_bound}.
\end{proof}

For $d\ge3$, we now choose the parameters used in the short-range estimate.
In dimension two, the common-radius comparison is used directly.

\medskip
\noindent\emph{Choice of parameters for $d\ge3$.}
Choose $c_d>0$ so that
\begin{equation*}
 g_{\kappa,d}(z)\ge c_d|z|^{2-d}
 \qquad\text{whenever }0\le\kappa|z|\le\frac12.
\end{equation*}
For the kernel associated with the common truncation radius $G_{\kappa,R}$, positivity and the Green function identity give
$0\le G_{\kappa,R}\le g_{\kappa,R}\le g_{\kappa,d}$.  Moreover, for
$d\ge3$ and $0<R\le1$,
\begin{equation}\label{eq:GR_uniform_power_bound}
 G_{\kappa,R}(z)\le C_dR^{2-d}
 \qquad(z\in\R^d),
\end{equation}
uniformly in $\kappa>0$.  Indeed, decompose
$\delta_0^{\kappa,R}$ into its surface and volume parts from
Lemma~\ref{lem:yukawa_operator_truncation}; both have mass at most one.
Since $g_{\kappa,d}\le g_{0,d}$, Newton's shell theorem bounds the
surface--surface and surface--volume terms by $C_dR^{2-d}$.  If
$b_{\kappa,R}=\kappa^2g_{\kappa,d}(R)$ denotes the volume density, then
$b_{\kappa,R}R^d\le C_d$ and
\[
 \sup_z\iint_{B_R\times B_R}|z+a-b|^{2-d}\,\dd a\dd b
 \le C_dR^{d+2},
\]
so the volume--volume contribution is likewise bounded by $C_dR^{2-d}$.
Choose $A=A(d)\ge1$ so large that
\[
 2C_dA^{2-d}\le c_d16^{2-d}.
\]
For dimensions $d\ge3$, define the final microscopic truncation scale by
\begin{equation}\label{eq:final_microscopic_scale_dge3}
 \eta_0
 :=\min\left\{
 \eta_{\mathrm{se}},\frac{r_*}{A},\frac1A,
 \frac1{32\kappa_*}
 \right\}.
\end{equation}
Thus $\eta_0=\eta_0(d,\kappa_*)$ and, in particular,
$A\eta_0\le\min\{1,r_*\}$ and
$\kappa_*\eta_0\le1/32$.  For $R=A\eta$,
$0<\eta<\eta_0$, and $|z|\le16\eta$, the preceding bounds give
\[
 g_{\kappa,d}(z)
 \ge c_d(16\eta)^{2-d}
 \ge2C_dR^{2-d}
 \ge2G_{\kappa,R}(z).
\]
Thus
\begin{equation*}
 G_{\kappa,R}(z)\le\frac12g_{\kappa,d}(z)
 \qquad(0<\kappa\le\kappa_*,\ 0<\eta<\eta_0,\ |z|\le16\eta).
\end{equation*}
The same hierarchy guarantees $\kappa r\le1/32$ for every truncation radius
$r\le\eta$.  All energy estimates in this section are stated for
\(0<\kappa\le\kappa_*\); the Coulomb limit is treated in
Section~\ref{sec:coulomb_limit}.

\subsection{Truncation at the nearest-neighbor scale and self-energy estimates}

We next choose the truncation radii at the nearest-neighbor scale.  This
preserves disjointness of the truncation balls and yields the energy estimate
below.
For $d\ge3$, $\eta_0$ is the scale in
\eqref{eq:final_microscopic_scale_dge3}.  In dimension two, define instead
\begin{equation}\label{eq:final_microscopic_scale_d2}
 \eta_0
 :=\min\left\{
 \eta_{\mathrm{se}},r_*,\frac1{32},\frac1{32\kappa_*}
 \right\}.
\end{equation}
Thus in every dimension
\begin{equation}\label{eq:final_microscopic_scale_dependency}
 \eta_0=\eta_0(d,\kappa_*)
\end{equation}
and the final microscopic scale is independent of $N$, $\rho$, the
configuration, and the particle separation.  All estimates below are stated
for $0<\eta<\eta_0$.

\begin{proposition}[Renormalized energy with variable truncation radii]
\label{prop:yukawa_energy_variable_radii}
Let $d\ge2$, $\kappa_*>0$, $N\ge2$, $0<\kappa\le\kappa_*$, and let
$\rho$ satisfy the background bound \eqref{eq:static_background_bound}.  Let
$X_N\notin\Delta_N$, $0<\eta<\eta_0$, and let $0<r_i\le\eta$ be such that
the balls $B(x_i,r_i)$ are pairwise disjoint.  We set
\[
 \delta_i:=\delta_{x_i}^{\kappa,r_i},\qquad
 M_i:=M_d(\kappa r_i),\qquad
 \bar M:=\frac1N\sum_iM_i,
\]
where $\bar M$ is used only to keep track of the two-dimensional additive
normalization.  Define
\[
 h_{\mathbf r}
 :=\frac1N\sum_i g_{\kappa,r_i}(\cdot-x_i)-V_\rho^\kappa,
\qquad
 \mathcal E_{\mathbf r}
 :=\|h_{\mathbf r}\|_{H^1_\kappa}^2
 =I_\kappa\!\left[\frac1N\sum_i\delta_i-\rho\right]\ge0.
\]
Then
\begin{equation}\label{eq:renormalized_energy_expansion}
 \mathcal E_{\mathbf r}
 =\cF_N^{\kappa,\sharp}(X_N,\rho)
 +\frac1{N^2}\sum_i s_{\kappa,d}^{\sharp}(r_i)
 +\mathcal R_{\mathbf r},
\end{equation}
where, with
\[
 J_i:=\int V_\rho^\kappa\,\dd\delta_i,
 \qquad
 D_i:=V_\rho^\kappa(x_i)-J_i,
\]
\begin{equation}\label{eq:renormalized_energy_remainder}
 \mathcal R_{\mathbf r}
 =\frac1{N^2}\sum_i\mathfrak b_{\kappa,d}(r_i)
 +\frac2N\sum_iD_i
 =\frac{c_{\kappa,d}(1-\bar M)}N
 +\frac2N\sum_iD_i.
\end{equation}
Moreover, with $C=C(d,\kappa_*,\|\rho\|_{L^\infty})$,
\begin{equation}\label{eq:yukawa_energy_variable_radii_bound}
 |\mathcal R_{\mathbf r}|
 \le C(\eta^2+\kappa^2\eta^2).
\end{equation}

For the nearest-neighbor truncation radii
\begin{equation}\label{eq:variable_radii_yukawa}
 r_i(X_N):=\min\left\{\eta,\frac1{16}\min_{j\ne i}|x_i-x_j|\right\},
\end{equation}
the doubled balls are pairwise disjoint, and there exist constants
$B=B(d,\kappa_*,\|\rho\|_{L^\infty})\ge1$ and
$C=C(d,\kappa_*,\|\rho\|_{L^\infty})$ such that
\begin{equation}\label{eq:yukawa_self_term_control}
 \frac1{N^2}\sum_i s_{\kappa,d}^{\sharp}(r_i)
 \le C\left(\cF_N^{\kappa,\sharp}(X_N,\rho)
 +B\varepsilon_N^\kappa(\eta)\right),
\end{equation}
\begin{equation}\label{eq:yukawa_energy_with_error_nonnegative}
 \cF_N^{\kappa,\sharp}(X_N,\rho)
 +B\varepsilon_N^\kappa(\eta)\ge\varepsilon_N^\kappa(\eta),
\end{equation}
and
\begin{equation}\label{eq:variable_radius_field_energy}
 \mathcal E_{\mathbf r}
 +\frac1{N^2}\sum_i s_{\kappa,d}^{\sharp}(r_i)
 \le C\left(\cF_N^{\kappa,\sharp}
 +B\varepsilon_N^\kappa(\eta)\right).
\end{equation}
\end{proposition}

\begin{proof}
Because the balls are pairwise disjoint, $|x_i-x_j|\ge r_i+r_j$ for
$i\ne j$.  Lemma~\ref{lem:yukawa_smeared_exterior_exactness} therefore gives the
exact interaction energy identity
$ I_\kappa[\delta_i,\delta_j]=g_{\kappa,d}(x_i-x_j)$.
Here we use the fact that the smeared charges preserve the exterior Yukawa
potential despite their loss of total mass.  Hence
\[
 \mathcal E_{\mathbf r}
 =\frac1{N^2}\sum_{i\ne j}g_{\kappa,d}(x_i-x_j)
 +\frac1{N^2}\sum_i s_{\kappa,d}(r_i)
 -\frac2N\sum_iJ_i+I_\kappa[\rho].
\]
Using
$\cF_N^{\kappa,\sharp}=\cF_N^{g_{\kappa,d}}-c_{\kappa,d}/N$ and
$s_{\kappa,d}=s_{\kappa,d}^{\sharp}-c_{\kappa,d}M_i$, together with
$N^{-1}\sum_iM_i=\bar M$, gives
\eqref{eq:renormalized_energy_expansion}--
\eqref{eq:renormalized_energy_remainder}.  By the background truncation estimate
\eqref{eq:background_truncation_bound},
$0\le D_i\le C\|\rho\|_\infty r_i^2$.  Moreover,
\eqref{eq:additive_normalization_defect_bound} gives
\[
 \frac1{N^2}\sum_i|\mathfrak b_{\kappa,d}(r_i)|
 \le \frac{C\eta^2}{N}.
\]
Therefore
\[
 |\mathcal R_{\mathbf r}|\le C\eta^2,
\]
which is stronger than \eqref{eq:yukawa_energy_variable_radii_bound}.  The
factor $|\log\kappa|$ is controlled by the uniform bound
\eqref{eq:additive_normalization_defect_bound} on the correction due to the
additive normalization.

For \eqref{eq:variable_radii_yukawa}, the doubled balls are disjoint.  Before
estimating the diagonal correction terms, we establish the positivity needed in both
dimensions by the same argument with a common truncation radius.  Set
\[
 R:=
 \begin{cases}
  A\eta,&d\ge3,\\
  \eta,&d=2.
 \end{cases}
\]
By the choices of $\eta_0$, Lemma~\ref{lem:common_truncation_energy_comparison}
applies at this radius.  Since $R$ is a fixed multiple of $\eta$, there is
$B_0\ge1$, depending only on $d$, $\kappa_*$, and $\|\rho\|_{L^\infty}$,
such that
\begin{equation}\label{eq:common_truncation_energy_lower_bound}
 0\le L_R\le
 \cF_N^{\kappa,\sharp}(X_N,\rho)
 +B_0\varepsilon_N^\kappa(\eta).
\end{equation}
Thus, before estimating any configuration-dependent diagonal correction
term, we have
\begin{equation}\label{eq:common_radius_positive_margin}
 \cF_N^{\kappa,\sharp}(X_N,\rho)
 +(B_0+1)\varepsilon_N^\kappa(\eta)
 \ge \varepsilon_N^\kappa(\eta).
\end{equation}
The indices for which $r_i=\eta$ are handled uniformly in all dimensions.  Estimate~\eqref{eq:microscopic_self_energy_bound}
gives
\begin{equation}\label{eq:nearest_neighbor_self_energy_control}
 \frac1{N^2}\sum_{i:\ r_i=\eta}s_{\kappa,d}^{\sharp}(r_i)
 \le C\frac{\ell_d(\eta)}N
 \le C\varepsilon_N^\kappa(\eta).
\end{equation}
For every remaining index $r_i<\eta$, choose $j(i)$ so that
$|x_i-x_{j(i)}|=16r_i$.  The corresponding ordered pairs have distinct first
labels and satisfy $|x_i-x_{j(i)}|\le16\eta$.  Consequently, for every
nonnegative function $\Phi$,
\begin{equation}\label{eq:selected_ordered_pair_multiplicity}
 \sum_{i:\,r_i<\eta}\Phi(x_i-x_{j(i)})
 \le \sum_{\substack{i\ne j\, :\,|x_i-x_j|\le16\eta}}
       \Phi(x_i-x_j).
\end{equation}
Only the treatment of the nearest-neighbor diagonal correction terms depends
on the dimension.

\smallskip
\noindent\emph{Case $d\ge3$.}
Since $r_i\le\eta<\eta_0$, the parameter choice above gives
$\kappa r_i\le\kappa_*\eta_0\le1/32$.  Hence
\eqref{eq:close_pair_self_energy_bound} applies and
\[
 s_{\kappa,d}^{\sharp}(r_i)=s_{\kappa,d}(r_i)
 \le Cg_{\kappa,d}(x_i-x_{j(i)}).
\]
At the common radius $R=A\eta$, the parameter choice above gives
$G_{\kappa,R}(z)\le\frac12g_{\kappa,d}(z)$ whenever $|z|\le16\eta$.
Hence, with $L_R$ from Lemma~\ref{lem:common_truncation_energy_comparison},
\[
 \frac1{N^2}\sum_{i\ne j:\ |x_i-x_j|\le16\eta}
 g_{\kappa,d}(x_i-x_j)
 \le2L_R
 \le C\left(\cF_N^{\kappa,\sharp}(X_N,\rho)
 +B_0\varepsilon_N^\kappa(\eta)\right).
\]
Summing over $r_i<\eta$ and using
\eqref{eq:selected_ordered_pair_multiplicity} with
$\Phi=g_{\kappa,d}$ therefore gives
\[
 \frac1{N^2}\sum_{i:\ r_i<\eta}s_{\kappa,d}^{\sharp}(r_i)
 \le C\left(\cF_N^{\kappa,\sharp}(X_N,\rho)
 +B\varepsilon_N^\kappa(\eta)\right),
\]
with $B$ chosen sufficiently large.  Together with
\eqref{eq:nearest_neighbor_self_energy_control}, this proves
\eqref{eq:yukawa_self_term_control} in dimensions $d\ge3$.

\smallskip
\noindent\emph{Case $d=2$.}
We use $G_{\kappa,\eta}$ and $L_\eta$ from
Lemma~\ref{lem:common_truncation_energy_comparison}.  Since
$zK_1(z)\to1$ as $z\downarrow0$ and $\kappa_*\eta_0\le1/32$, there is
$c_0>0$, independent of $\kappa$, such that
\[
 -g_{\kappa,2}'(r)\ge\frac{c_0}{2\pi r}
 \qquad(0<r\le\eta_0).
\]
If $16r_i\ge\eta$, then $r_i\ge\eta/16$ and
\eqref{eq:microscopic_self_energy_bound} gives
$
 s_{\kappa,2}^{\sharp}(r_i)\le C\ell_2(\eta).
$

If $16r_i<\eta$, then \eqref{eq:common_truncation_kernel_order} and the fact that
$g_{\kappa,\eta}$ is constant on $B_\eta$ give
\begin{align*}
 (g_{\kappa,2}-G_{\kappa,\eta})(x_i-x_{j(i)})
 \ge g_{\kappa,2}(16r_i)-g_{\kappa,2}(\eta)=\int_{16r_i}^{\eta}-g_{\kappa,2}'(s)\,\dd s
 \ge c\log\frac{\eta}{16r_i}.
\end{align*}
Combining this with
$\ell_2(r_i)\le C\bigl(\ell_2(\eta)+\log(\eta/(16r_i))\bigr)$ and
\eqref{eq:microscopic_self_energy_bound} yields, in both subcases,
\begin{equation}\label{eq:two_dimensional_self_loss_comparison}
 s_{\kappa,2}^{\sharp}(r_i)
 \le C\left(\ell_2(\eta)
 +(g_{\kappa,2}-G_{\kappa,\eta})(x_i-x_{j(i)})\right).
\end{equation}
Since $g_{\kappa,2}-G_{\kappa,\eta}\ge0$, summing
\eqref{eq:two_dimensional_self_loss_comparison} and using
\eqref{eq:selected_ordered_pair_multiplicity} yields
\[
 \frac1{N^2}\sum_{i:\ r_i<\eta}s_{\kappa,2}^{\sharp}(r_i)
 \le C\left(\frac{\ell_2(\eta)}N+L_\eta\right).
\]
By \eqref{eq:common_truncation_energy_lower_bound} and
\eqref{eq:nearest_neighbor_self_energy_control}, this proves
\eqref{eq:yukawa_self_term_control} in $d=2$ as well.

Choose the final constant $B$ larger than all constants required in the
preceding self-energy estimates and also larger than $B_0+1$.  Since
$\varepsilon_N^\kappa(\eta)\ge0$, increasing its coefficient preserves the
lower bound already established in
\eqref{eq:common_radius_positive_margin}; hence
\[
 \cF_N^{\kappa,\sharp}+B\varepsilon_N^\kappa(\eta)
 \ge \varepsilon_N^\kappa(\eta).
\]
This proves \eqref{eq:yukawa_energy_with_error_nonnegative}.  Then
\eqref{eq:renormalized_energy_expansion}, the remainder bound, and
\eqref{eq:yukawa_self_term_control} give
\eqref{eq:variable_radius_field_energy} directly.

\end{proof}

\section{Commutator estimate}\label{sec:yukawa_commutator}

Fix $d\ge2$ and let $\rho$ satisfy
\eqref{eq:static_background_bound}.  Throughout this section,
$\eta_0=\eta_0(d,\kappa_*)$ denotes the microscopic scale fixed in
\eqref{eq:final_microscopic_scale_dependency}.  Generic constants $C$ depend only on $(d,\kappa_*,\Lambda)$ and may change
from line to line, while constants $C_d$ depend only on the dimension.  We fix
$B=B(d,\kappa_*,\Lambda)$ as in
Proposition~\ref{prop:yukawa_energy_variable_radii}.

We combine the energy estimate for variable truncation radii from
Proposition~\ref{prop:yukawa_energy_variable_radii} with the stress-energy
identity with interface terms from Section~\ref{sec:yukawa_main}.  First we
compare the commutator for the point charges with the finite-energy field of
the smeared charges.  We then average the resulting errors over a common
scaling of the truncation radii and select one admissible scale.

\subsection{Comparison with smeared charges}

Let $N\ge2$, $0<\kappa\le\kappa_*$, $X_N\notin\Delta_N$ and
$0<\eta<\eta_0$.  Let the truncation radii be given by
\eqref{eq:variable_radii_yukawa}, and let $B$ be the constant in
Proposition~\ref{prop:yukawa_energy_variable_radii}.  Define
\begin{equation}\label{eq:commutator_shift}
 \mathcal S:=\cF_N^{\kappa,\sharp}(X_N,\rho)
 +B\varepsilon_N^\kappa(\eta).
\end{equation}
By Proposition~\ref{prop:yukawa_energy_variable_radii},
\begin{equation}\label{eq:modulated_energy_with_error_positive}
 \mathcal S\ge\varepsilon_N^\kappa(\eta)\ge0.
\end{equation}

For \(t\in[1/2,1]\), set \(r_i(t)=tr_i\) and abbreviate
\[
 \delta_i^t:=\delta_{x_i}^{\kappa,r_i(t)}.
\]
Here \(t\) is an auxiliary scaling parameter, not the physical time; it
rescales all truncation radii by the same factor.  For \(i\ne j\),
\[
 2r_i(t)+2r_j(t)\le2r_i+2r_j\le\frac12|x_i-x_j|,
\]
so the enlarged balls \(B(x_i,2r_i(t))\) are pairwise disjoint.  Define
\begin{equation}
\label{eq:point_and_smeared_sources}
 \widetilde\nu_t:=\frac1N\sum_i\delta_i^t-\rho,
 \qquad
 \widetilde h_t:=g_{\kappa,d}*\widetilde\nu_t
 =\frac1N\sum_i g_{\kappa,r_i(t)}(\cdot-x_i)-V_\rho^\kappa.
\end{equation}
The signed source is generally non-neutral, with total mass
$N^{-1}\sum_iM_d(\kappa r_i(t))-1$, but it belongs to $H^{-1}_\kappa$ for
fixed $\kappa>0$.  The same source is used in both the comparison between the
point charges and the smeared charges and the energy estimate.  Set
\[
 \mathcal E_t:=\|\widetilde h_t\|_{H^1_\kappa}^2
 =I_\kappa[\widetilde\nu_t].
\]
Each smeared charge and the bounded background have finite Yukawa energy, so
all stress pairings below are finite.

At fixed \(0<\kappa\le\kappa_*\), the common scaling parameter enters
continuously in the energy topology.  Indeed, if \(t_n\to t\), then all
radii \(r_i(t_n)\) and \(r_i(t)\) lie in \([r_i/2,r_i]\).  Dominated
convergence, together with the exponential Yukawa decay, therefore gives
\[
 g_{\kappa,r_i(t_n)}\longrightarrow g_{\kappa,r_i(t)}
 \quad\text{strongly in }H^1_\kappa(\R^d).
\]
After translation and finite summation, \(t\mapsto\widetilde h_t\) is
continuous in $H^1_\kappa$, and hence so are
$\mathcal E_t=\|\widetilde h_t\|_{H^1_\kappa}^2$ and the self-energies.
Moreover,
\[
 \|T_\kappa[\widetilde h_t]-T_\kappa[\widetilde h_s]\|_{L^1}
 \le C\bigl(\|\widetilde h_t\|_{H^1_\kappa}+\|\widetilde h_s\|_{H^1_\kappa}\bigr)
       \|\widetilde h_t-\widetilde h_s\|_{H^1_\kappa},
\]
so, for every fixed globally Lipschitz field \(u\),
\begin{equation}
\label{eq:scaled_truncation_stress_continuity}
 t\longmapsto \mathcal B_\kappa(\widetilde\nu_t;u)
 \quad\text{is continuous on }[1/2,1].
\end{equation}
The continuity assertion is used only at fixed \(\kappa>0\); the uniform
commutator estimates below rely on the quantitative bounds established above.

We next compare the commutator for the point charges with the stress-energy
pairing for the smeared charges.  The required weak derivative follows from
the Green function identity \eqref{eq:smeared_charge_resolvent}.  Since
\(g_{\kappa,d}\in L^1_{\rm loc}\) and \(\delta_0^{\kappa,r}\) is finite and
compactly supported, distributional differentiation gives
\begin{equation}
\label{eq:yukawa_weak_derivative_convolution}
 \nabla(g_{\kappa,d}*\delta_0^{\kappa,r})
 =\nabla g_{\kappa,d}*\delta_0^{\kappa,r}
 =\nabla g_{\kappa,r}
 \qquad\text{in }\mathcal D'(\R^d;\R^d).
\end{equation}
Thus, after translation,
\begin{equation}
\label{eq:yukawa_translated_weak_derivative}
 \int_{\R^d}\nabla_z g_{\kappa,d}(z-y)\,\dd\delta_x^{\kappa,r}(z)
 =\nabla g_{\kappa,r}(x-y)
 \qquad\text{in }\mathcal D'_y(\R^d;\R^d),
\end{equation}
and the right-hand side equals \(\nabla g_{\kappa,d}(x-y)\) whenever
\(|x-y|>r\).  If \(\rho\in L^1\cap L^\infty\), then the pairing is
absolutely convergent because
\[
 |\nabla g_{\kappa,d}(w)|
 \le C\bigl(|w|^{1-d}\mathbf1_{\{|w|<1\}}+\mathbf1_{\{|w|\ge1\}}\bigr),
 \qquad M_d(\kappa r)\le1,
\]
and hence Fubini gives
\begin{equation}
\label{eq:yukawa_weak_derivative_background_pairing}
 \int_{\R^d}\rho(y)
 \left[\int\nabla_z g_{\kappa,d}(z-y)\,\dd\delta_x^{\kappa,r}(z)\right]\dd y
 =\int_{\R^d}\rho(y)\nabla g_{\kappa,r}(x-y)\dd y.
\end{equation}

For the remainder of this subsection, write
\[
 \cH_N^{g_{\kappa,d}}(u):=\cH_N^{g_{\kappa,d}}(X_N,\rho;u),
 \qquad
 h_N:=\frac1N\sum_jg_{\kappa,d}(\cdot-x_j)-V_\rho^\kappa,
 \qquad
 h_N^i:=h_N-\frac1Ng_{\kappa,d}(\cdot-x_i).
\]
Both $h_N$ and $h_N^i$ are singular point-charge potentials.  Neither belongs
globally to $H^1_\kappa$ because of the untruncated self-interaction
singularities.  Only $h_N^i$ away from $x_i$ and its gradient are used below.
By contrast, $\widetilde h_t$ is the finite-energy potential generated by the
smeared charges and the unchanged background.

We can now express the first variation of the point charges as the stress-energy
pairing of the smeared charges plus explicit error terms.

\begin{lemma}
\label{lem:point_smeared_charge_identity}
With the preceding notation, let \(u\in C_c^\infty(\R^d;\R^d)\).  Define
\begin{equation}
\label{eq:point_smeared_first_variation_term}
\begin{aligned}
 \mathcal R_t(u)
 :=&\frac2N\sum_i\int (u(x_i)-u)\cdot\nabla h_N^i\,\dd\delta_i^t
 -\frac2{N^2}\sum_i\int
       (u-u(x_i))\cdot\nabla g_{\kappa,r_i(t)}(\cdot-x_i)\,\dd\delta_i^t\\
 &+\frac2N\sum_i\int
       (u(x)-u(x_i))\cdot\nabla(g_{\kappa,r_i(t)}-g_{\kappa,d})(x-x_i)
       \rho(x)\,\dd x.
\end{aligned}
\end{equation}
For almost every \(t\in[1/2,1]\), surface pairings are understood using the
arithmetic mean of the one-sided traces from
Proposition~\ref{prop:average_trace_yukawa_green_stress}.  Then
\begin{equation}
\label{eq:point_smeared_charge_identity}
 \cH_N^{g_{\kappa,d}}(X_N,\rho;u)
 =\mathcal B_\kappa(\widetilde\nu_t;u)+\mathcal R_t(u).
\end{equation}
\end{lemma}

\begin{proof}
We first rewrite the commutator for the point charges.  Since
\(g_{\kappa,d}\) is radial, \(\nabla g_{\kappa,d}\) is odd; exchanging the
indices in the particle--particle term and the variables in the
background--background term gives
\begin{equation}
\label{eq:point_smeared_desymmetrization}
 \cH_N^{g_{\kappa,d}}(u)
 =\frac2N\sum_i u(x_i)\cdot\nabla h_N^i(x_i)
 -2\int u\cdot\nabla h_N\,\rho\,\dd x.
\end{equation}

We set \(B_i^t:=B(x_i,r_i(t))\).  By the disjointness of
\(B(x_i,2r_i(t))\), the truncations centered at \(x_j\), \(j\ne i\), are
inactive on \(B_i^t\).  Hence
\begin{equation}
\label{eq:local_truncated_field}
 \nabla\widetilde h_t
 =\nabla h_N^i+\frac1N\nabla g_{\kappa,r_i(t)}(\cdot-x_i)
 \quad\text{in }B_i^t,
\end{equation}
and
\begin{equation}
\label{eq:point_smeared_field_difference}
 \nabla(h_N-\widetilde h_t)
 =\frac1N\sum_j\nabla(g_{\kappa,d}-g_{\kappa,r_j(t)})(\cdot-x_j)
 \quad\text{a.e.}
\end{equation}

To apply Proposition~\ref{prop:average_trace_yukawa_green_stress}, note from
\eqref{eq:general_dimensional_smeared_charge} that
\begin{equation}
\label{eq:smeared_charge_source_decomposition}
 \dd\delta_i^t
 =A_d(\kappa r_i(t))\,\dd\sigma_{r_i(t),x_i}
 +\kappa^2g_{\kappa,d}(r_i(t))\mathbf1_{B_i^t}\,\dd x.
\end{equation}
Therefore
\begin{equation}
\label{eq:point_smeared_source_difference}
 L_\kappa\widetilde h_t
 =\widetilde\nu_t
 =f_t\,\dd x+\sum_i\lambda_i^t
   \mathcal H^{d-1}\lfloor_{\partial B_i^t},
\end{equation}
with
\begin{equation}
\label{eq:source_pairing_derivative}
 f_t=\frac{\kappa^2}{N}\sum_i g_{\kappa,d}(r_i(t))
       \mathbf1_{B_i^t}-\rho,
 \qquad
 \lambda_i^t=\frac{A_d(\kappa r_i(t))}
 {N\Sd r_i(t)^{d-1}},
\end{equation}
since
\(\dd\sigma_{r_i(t),x_i}=(\Sd r_i(t)^{d-1})^{-1}
\dd\mathcal H^{d-1}\lfloor_{\partial B_i^t}\).
The hypotheses of Proposition~\ref{prop:average_trace_yukawa_green_stress}
are satisfied.  Indeed, \(\widetilde h_t\in H^1_\kappa\cap C^0\) and the
spheres \(\partial B_i^t\) are pairwise disjoint.  Near each sphere,
\(\widetilde h_t=N^{-1}g_{\kappa,r_i(t)}(\cdot-x_i)+H_i^t\), where
\(H_i^t\) is \(W^{2,p}\) across the interface and the truncated self-potential
has \(W^{2,p}\) extensions on both sides.  Hence, for any \(p>d\), the
proposition prescribes the average
\[
 \langle \nabla\widetilde h_t\rangle
 =\frac12(\nabla\widetilde h_t^++\nabla\widetilde h_t^-)
\]
on each truncation sphere.  Substituting \eqref{eq:source_pairing_derivative} in the
stress-energy identity with interface terms gives
\begin{equation}
\label{eq:explicit_surface_volume_pairing}
\begin{aligned}
 \mathcal B_\kappa(\widetilde\nu_t;u)
 &=\frac2N\sum_i A_d(\kappa r_i(t))
   \int_{\partial B_i^t}u\cdot\langle \nabla\widetilde h_t\rangle\,
      \dd\sigma_{r_i(t),x_i}\\
 &\quad+\frac2N\sum_i \kappa^2g_{\kappa,d}(r_i(t))
   \int_{B_i^t}u\cdot\nabla\widetilde h_t\,\dd x
   -2\int u\cdot\nabla\widetilde h_t\,\rho\,\dd x.
\end{aligned}
\end{equation}
Equivalently, with the same convention for averaging the one-sided traces on the surface part,
\begin{equation}
\label{eq:compactly_supported_source_pairing}
 \mathcal B_\kappa(\widetilde\nu_t;u)
 =\frac2N\sum_i\int u\cdot\nabla\widetilde h_t\,\dd\delta_i^t
 -2\int u\cdot\nabla\widetilde h_t\,\rho\,\dd x.
\end{equation}

Using \eqref{eq:local_truncated_field}--\eqref{eq:point_smeared_field_difference}
and subtracting \eqref{eq:compactly_supported_source_pairing} from
\eqref{eq:point_smeared_desymmetrization}, we obtain
\begin{align*}
 \cH_N^{g_{\kappa,d}}(u)-\mathcal B_\kappa(\widetilde\nu_t;u)
 &=\frac2N\sum_i\left[u(x_i)\cdot\nabla h_N^i(x_i)
      -\int u\cdot\nabla h_N^i\,\dd\delta_i^t\right]\\
 &\quad-\frac2{N^2}\sum_i\int
      u\cdot\nabla g_{\kappa,r_i(t)}(\cdot-x_i)\,\dd\delta_i^t\\
 &\quad+\frac2N\sum_i\int u(x)\cdot\nabla
      (g_{\kappa,r_i(t)}-g_{\kappa,d})(x-x_i)\rho(x)\,\dd x.
\end{align*}
To complete the identity, we subtract the constant values \(u(x_i)\).  By \eqref{eq:yukawa_translated_weak_derivative}, for \(j\ne i\),
\[
 \int\nabla_zg_{\kappa,d}(z-x_j)\,\dd\delta_i^t(z)
 =\nabla g_{\kappa,r_i(t)}(x_i-x_j)
 =\nabla g_{\kappa,d}(x_i-x_j),
\]
and hence
\begin{equation}
\label{eq:external_field_comparison_error}
 \nabla h_N^i(x_i)-\int\nabla h_N^i\,\dd\delta_i^t
 =\int\nabla(g_{\kappa,r_i(t)}-g_{\kappa,d})(x_i-y)\rho(y)\,\dd y.
\end{equation}
Moreover, radial symmetry gives
\begin{equation}
\label{eq:radial_self_average_zero}
 \int\nabla g_{\kappa,r_i(t)}(z-x_i)\,\dd\delta_i^t(z)=0.
\end{equation}
In this identity the value of the gradient on the surface part of
$\delta_i^t$ is the average of the interior and exterior traces
prescribed by Proposition~\ref{prop:average_trace_yukawa_green_stress}; the
surface contribution is radial and has zero spherical average, while the volume
contribution vanishes because $g_{\kappa,r_i(t)}$ is constant in the ball.
After adding and subtracting \(u(x_i)\) in the three terms above,
\eqref{eq:radial_self_average_zero} removes the self constant, while the
remaining two background constants cancel because
\(g_{\kappa,r_i(t)}-g_{\kappa,d}\) is radial and its gradient is odd.  The
remaining three terms are exactly those in
\eqref{eq:point_smeared_first_variation_term}, which proves
\eqref{eq:point_smeared_charge_identity}.
\end{proof}

We first record the uniform bounds in the scaling parameter that enter the
averaging argument.  They follow directly from the renormalized energy
estimate for variable truncation radii.

\smallskip\noindent\emph{Uniform estimates for the rescaled radii.}
For $t\in[1/2,1]$, the radii $r_i(t)=tr_i$ remain pairwise admissible and
$r_i(t)\le\eta$.  Estimate~\eqref{eq:microscopic_self_energy_bound} and
$\ell_d(tr)\le C_d\ell_d(r)$ and the two-sided comparison
\eqref{eq:microscopic_self_energy_bound} give
\begin{equation}\label{eq:scaled_self_term_comparison}
 s_{\kappa,d}^{\sharp}(tr_i)
 \le C_{d,\kappa_*} s_{\kappa,d}^{\sharp}(r_i),
 \qquad t\in[1/2,1].
\end{equation}
Apply the exact expansion \eqref{eq:renormalized_energy_expansion} to the
radii $r_i(t)$, with the background still equal to $\rho$.  Since
$\mathcal S\ge\varepsilon_N^\kappa(\eta)$ by
\eqref{eq:modulated_energy_with_error_positive}, the uniform remainder bound,
\eqref{eq:scaled_self_term_comparison}, and
\eqref{eq:yukawa_self_term_control} yield, uniformly in $t$,
\begin{equation}\label{eq:scaled_radius_energy_bound}
 \mathcal E_t+\frac1{N^2}\sum_i s_{\kappa,d}^{\sharp}(r_i(t))
 \le C\mathcal S.
\end{equation}
Equivalently,
\begin{equation}\label{eq:uniform_truncated_potential_bounds}
 \|\widetilde h_t\|_{H^1_\kappa}^2
 +\frac1{N^2}\sum_i s_{\kappa,d}^{\sharp}(r_i(t))
 \le C\mathcal S,
 \qquad t\in[1/2,1].
\end{equation}
No neutrality correction and no estimate involving a factor $\kappa^{-1}$ is used here.

\subsubsection{Averaging over truncation radii}

Averaging over the scaling parameter allows us to choose a truncation radius
for which the commutator error and the field energy are controlled
simultaneously.

\begin{lemma}[Averaged commutator estimate and selection of a truncation radius]
\label{lem:averaged_commutator_comparison}
There exists $C=C(d,\kappa_*,\Lambda)\ge1$ with the following properties.
For every $u\in C_c^\infty(\R^d;\R^d)$,
\begin{equation}\label{eq:averaged_commutator_comparison}
 \int_{1/2}^1
 \left|\cH_N^{g_{\kappa,d}}(X_N,\rho;u)-\mathcal B_\kappa(\widetilde\nu_t;u)\right|\,\dd t
 \le C\|\nabla u\|_{L^\infty}\mathcal S.
\end{equation}
For every globally Lipschitz $u:\R^d\to\R^d$, there exists
$t_*\in[1/2,1]$ such that, with $\widetilde r_i=t_*r_i$,
\begin{equation}\label{eq:common_truncation_scale_remainder}
\left|\cH_N^{g_{\kappa,d}}(X_N,\rho;u)
-\mathcal B_\kappa(\widetilde\nu_{t_*};u)\right|
\le C\|\nabla u\|_{L^\infty}\mathcal S,
\end{equation}
\begin{equation}\label{eq:common_truncation_scale_energy}
\mathcal E_{t_*}+\frac1{N^2}\sum_i s_{\kappa,d}^{\sharp}(\widetilde r_i)
\le C\mathcal S,
\end{equation}
and consequently
\begin{equation}\label{eq:selected_truncation_stress_bound}
 |\mathcal B_\kappa(\widetilde\nu_{t_*};u)|
 \le C\|\nabla u\|_{L^\infty}\mathcal S.
\end{equation}
\end{lemma}

\begin{proof}
All constants below have the dependence stated in the lemma.  Since
$r_i(t)\le\eta<\eta_0$ and $0<\kappa\le\kappa_*$, every estimate for $s_{\kappa,d}^{\sharp}$ used below is within the scale fixed in
Section~\ref{sec:truncation_energy}.  For almost every $t\in[1/2,1]$, Lemma~\ref{lem:point_smeared_charge_identity}
gives
\begin{equation}
 \cH_N^{g_{\kappa,d}}(u)-\mathcal B_\kappa(\widetilde\nu_t;u)
 =\mathcal R_t(u).
 \label{eq:averaged_comparison_remainder_start}
\end{equation}
We first verify measurability in the scaling parameter and then estimate the
surface, self-interaction, volume, and background truncation terms
separately.

\medskip
\noindent\emph{Measurability of the radius-dependent terms.}
By \eqref{eq:scaled_truncation_stress_continuity},
$t\mapsto\mathcal B_\kappa(\widetilde\nu_t;u)$ is continuous on $[1/2,1]$.  The
explicit source terms in $\mathcal R_t(u)$ are Borel functions of $t$: a surface
integral is parameterized by
$x_i+r_i(t)\omega$, $\omega\in\mathbb S^{d-1}$, whereas the volume terms are
Lebesgue integrals over $B(x_i,r_i(t))$ with Borel dependence on the radius.
Consequently every term in \eqref{eq:averaged_comparison_remainder_start}
has a Borel representative.  The estimates below provide an
$L^1(1/2,1)$ majorant, so Tonelli's and Fubini's theorems can be used without
any further choice of representatives.

The smeared charge has the decomposition
\begin{equation}
 \dd\delta_i^t
 =A_d(\kappa r_i(t))\,\dd\sigma_{r_i(t),x_i}
 +\kappa^2g_{\kappa,d}(r_i(t))
   \mathbf1_{B(x_i,r_i(t))}\,\dd x.
 \label{eq:averaged_smeared_charge_decomposition}
\end{equation}
Since $A_d(\kappa r)\le M_d(\kappa r)\le1$, the surface coefficient is
uniformly bounded.

\medskip
\noindent\emph{The self term.}
Consider the second term in \eqref{eq:point_smeared_first_variation_term}.  The volume part
of \eqref{eq:averaged_smeared_charge_decomposition} gives no contribution,
because $g_{\kappa,r}$ is constant on $B_r$.  On $\partial B_r$ the interior
gradient is zero and the exterior gradient is
$g_{\kappa,d}'(r)n$, so the average of the one-sided traces has magnitude
$-g_{\kappa,d}'(r)/2$.  Hence
\begin{align}
 r\int_{\partial B_r}|\langle\nabla g_{\kappa,r}\rangle|
 \,\dd\delta_0^{\kappa,r}
 =\frac r2 A_d(\kappa r)(-g_{\kappa,d}'(r))
 =\frac{A_d(\kappa r)^2}
        {2|\mathbb S^{d-1}|r^{d-2}}.
 \label{eq:averaged_self_trace_exact}
\end{align}
Since $A_d(\kappa r)\le1$ and, for every $d\ge2$ on the microscopic scale,
\[
 r^{2-d}\le \ell_d(r)\le C s_{\kappa,d}^{\sharp}(r)
\]
(with $r^{2-d}=1$ when $d=2$), \eqref{eq:averaged_self_trace_exact} gives, for every fixed $d\ge2$,
\begin{equation*}
 \left|\frac2{N^2}\sum_i\int
 (u-u(x_i))\cdot\nabla g_{\kappa,r_i(t)}\,\dd\delta_i^t\right|
 \le C\|\nabla u\|_{L^\infty}\frac1{N^2}\sum_i
 s_{\kappa,d}^{\sharp}(r_i(t)).
\end{equation*}
The uniform estimate \eqref{eq:uniform_truncated_potential_bounds} therefore controls this term by
$C\|\nabla u\|_{L^\infty}\mathcal S$, uniformly in $t$.

\medskip
\noindent\emph{The particle--background truncation error.}
The third term in \eqref{eq:point_smeared_first_variation_term} is supported in the
truncation ball.  Since
$|u(x)-u(x_i)|\le\|\nabla u\|_{L^\infty}|x-x_i|$ and the force bound valid in every dimension
\eqref{eq:yukawa_force_bound} gives
\[
 |\nabla g_{\kappa,d}(z)|
 \le \frac1{|\mathbb S^{d-1}|\,|z|^{d-1}},\qquad z\ne0,
\]
polar coordinates yield, for every fixed $d\ge2$,
\begin{align*}
 \left|\int
 (u-u(x_i))\cdot\nabla(g_{\kappa,r_i(t)}-g_{\kappa,d})(\cdot-x_i)
 \rho\,\dd x\right|
 &\le C_d\|\nabla u\|_{L^\infty}\|\rho\|_{L^\infty}
 \int_0^{r_i(t)} s\,\dd s\\
 &\le C_d\|\nabla u\|_{L^\infty}\|\rho\|_{L^\infty}r_i(t)^2.
\end{align*}
After summation with the prefactor $2/N$, this is bounded by
$C\|\nabla u\|_{L^\infty}\eta^2\le C\|\nabla u\|_{L^\infty}\mathcal S$.

\medskip
\noindent\emph{The surface contribution and averaging in the radius.}
For the first term in \eqref{eq:point_smeared_first_variation_term}, we define
\[
 E_i(r):=\int_{\partial B(x_i,r)}|\nabla h_N^i|^2
 \,\dd\mathcal H^{d-1}.
\]
Cauchy--Schwarz on the sphere, together with the normalization of $\sigma_r$,
gives
\begin{align*}
 r\int_{\partial B(x_i,r)}|\nabla h_N^i|\,\dd\sigma_{r,x_i}
 \le \frac{r}{(|\mathbb S^{d-1}|r^{d-1})^{1/2}}E_i(r)^{1/2}\notag\le C_dr^{-(d-3)/2}E_i(r)^{1/2}.
\end{align*}
With $r=r_i(t)=tr_i$ and $\dd t=\dd r/r_i$,
\begin{align}
 \int_{1/2}^1 r_i(t)
 \int_{\partial B(x_i,r_i(t))}|\nabla h_N^i|\,\dd\sigma_{r_i(t),x_i}\,\dd t\notag
 &\le\frac C{r_i}\int_{r_i/2}^{r_i}
 r^{-(d-3)/2}E_i(r)^{1/2}\,\dd r\notag\\
 &\le C_dr_i^{1-d/2}
 \left(\int_{B(x_i,r_i)\setminus B(x_i,r_i/2)}
 |\nabla h_N^i|^2\,\dd x\right)^{1/2}.
 \label{eq:single_radius_average}
\end{align}
Indeed, for $d\ne4$ the radial weight satisfies
\[
 \frac1{r_i}\left(\int_{r_i/2}^{r_i}r^{3-d}\,\dd r\right)^{1/2}
 \le C_d r_i^{1-d/2}.
\]
For $d=4$ the corresponding critical factor is
\[
 \frac1{r_i}\left(\int_{r_i/2}^{r_i}\frac{\dd r}{r}\right)^{1/2}
 =\frac{(\log2)^{1/2}}{r_i}
 =(\log2)^{1/2}r_i^{1-d/2}.
\]
Thus the logarithmic constant is absorbed into $C_d$, consistently with the
power displayed above.

At $t=1$, set
\[
 \widetilde h_{\mathbf r}:=\widetilde h_{t=1}
 =\frac1N\sum_i g_{\kappa,r_i}(\cdot-x_i)-V_\rho^\kappa.
\]
If $x\in B(x_i,r_i)$, the $i$th truncated self-potential is
constant.  For $j\ne i$, the nearest-neighbor definition gives
$r_i\le|x_i-x_j|/16$ and $r_j\le|x_i-x_j|/16$.  Hence
\[
 |x-x_j|\ge |x_i-x_j|-r_i
 \ge\frac{15}{16}|x_i-x_j|>r_j.
\]
Therefore every other truncated potential agrees with the original Yukawa
potential on $B(x_i,r_i)$ and
\begin{equation*}
 \nabla h_N^i=\nabla\widetilde h_{\mathbf r}
 \qquad\text{a.e.\ in }B(x_i,r_i).
\end{equation*}
Moreover, \eqref{eq:microscopic_self_energy_bound} gives the single estimate
\begin{equation}\label{eq:radius_weight_by_self_energy}
 r_i^{2-d}\le \ell_d(r_i)\le C s_{\kappa,d}^{\sharp}(r_i),
 \qquad d\ge2,
\end{equation}
where $r_i^{2-d}=1$ in dimension two.
The doubled nearest-neighbor balls are pairwise disjoint; hence the base
balls, and therefore the annuli
$B(x_i,r_i)\setminus B(x_i,r_i/2)$, are pairwise disjoint.  We use the
Cauchy--Schwarz inequality over the particle labels in the form
\begin{equation}\label{eq:label_space_cauchy_schwarz}
 \frac1N\sum_i a_i b_i
 \le
 \left(\frac1{N^2}\sum_i a_i^2\right)^{1/2}
 \left(\sum_i b_i^2\right)^{1/2}.
\end{equation}
Taking $a_i=r_i^{1-d/2}$ and
$b_i=\|\nabla\widetilde h_{\mathbf r}\|_{L^2(B(x_i,r_i)\setminus
B(x_i,r_i/2))}$, and using
\eqref{eq:radius_weight_by_self_energy}, we may sum
\eqref{eq:single_radius_average} to obtain
\begin{align*}
 \int_{1/2}^1\frac1N\sum_i r_i(t)
 \int_{\partial B(x_i,r_i(t))}|\nabla h_N^i|\,\dd\sigma_{r_i(t),x_i}\,\dd t
 \le&
 C\left(\frac1{N^2}\sum_i s_{\kappa,d}^{\sharp}(r_i)\right)^{1/2}
 \left(\sum_i\int_{B(x_i,r_i)\setminus B(x_i,r_i/2)}
 |\nabla\widetilde h_{\mathbf r}|^2\,\dd x\right)^{1/2}\notag\\
 \le&
 C\left(\frac1{N^2}\sum_i s_{\kappa,d}^{\sharp}(r_i)\right)^{1/2}
 \|\widetilde h_{\mathbf r}\|_{H^1_\kappa}
 \le C\mathcal S.
\end{align*}
The last inequality follows from \eqref{eq:uniform_truncated_potential_bounds}
at $t=1$.

\medskip
\noindent\emph{The volume contribution.}
Let
\[
 e_i^t:=\int_{B(x_i,r_i(t))}|\nabla h_N^i|^2\,\dd x.
\]
By Cauchy--Schwarz in the ball, we obtain
\begin{equation}
 \int_{B(x_i,r_i(t))}|\nabla h_N^i|\,\dd x
 \le C_dr_i(t)^{d/2}(e_i^t)^{1/2}.
 \label{eq:averaged_bulk_ball_cs}
\end{equation}
For each $t\in[1/2,1]$, the doubled balls
$B(x_i,2r_i(t))$ are pairwise disjoint.  The same local argument as above
shows that, on $B(x_i,r_i(t))$,
$\nabla h_N^i=\nabla\widetilde h_t$.  Hence
\begin{equation}
 \sum_i e_i^t\le\|\nabla\widetilde h_t\|_{L^2}^2
 \le\|\widetilde h_t\|_{H^1_\kappa}^2.
 \label{eq:averaged_bulk_disjoint_energy}
\end{equation}

The mass of the volume part of the smeared charge is
\[
 B_d(\kappa r)
 =\kappa^2g_{\kappa,d}(r)|B_r|
 =\frac{|\mathbb S^{d-1}|}{d}\,\kappa^2g_{\kappa,d}(r)r^d,
\]
and Lemma~\ref{lem:yukawa_operator_truncation} gives
$0\le B_d(\kappa r)\le M_d(\kappa r)\le1$.  Since
\[
 r^{1-d/2}=(r^{2-d})^{1/2}
 \le C\bigl(s_{\kappa,d}^{\sharp}(r)\bigr)^{1/2}
\]
by \eqref{eq:radius_weight_by_self_energy} (with $r^{1-d/2}=1$ in
dimension two), we obtain
\begin{equation}\label{eq:volume_coefficient_by_self_energy}
 \kappa^2g_{\kappa,d}(r)r^{1+d/2}
 =\frac{d}{|\mathbb S^{d-1}|}B_d(\kappa r)r^{1-d/2}
 \le C_d\bigl(s_{\kappa,d}^{\sharp}(r)\bigr)^{1/2}.
\end{equation}  Combining
\eqref{eq:averaged_bulk_ball_cs},
\eqref{eq:averaged_bulk_disjoint_energy}, and
\eqref{eq:volume_coefficient_by_self_energy}, and then applying
\eqref{eq:label_space_cauchy_schwarz}, yields in every dimension $d\ge2$
\begin{align*}
 \frac{\kappa^2}{N}\sum_i r_i(t)g_{\kappa,d}(r_i(t))
 \int_{B(x_i,r_i(t))}|\nabla h_N^i|\,\dd x\notag\le C_d
 \left(\frac1{N^2}\sum_i s_{\kappa,d}^{\sharp}(r_i(t))\right)^{1/2}
 \left(\sum_i e_i^t\right)^{1/2}
 \le C\mathcal S.
\end{align*}
Thus the interior volume contribution is bounded by
$C\|\nabla u\|_{L^\infty}\mathcal S$ in every fixed dimension $d\ge2$.

Combining these four estimates and integrating in $t$ gives
\begin{equation}\label{eq:averaged_point_smeared_charge_bound}
 \int_{1/2}^1
 \left|\cH_N^{g_{\kappa,d}}(u)-\mathcal B_\kappa(\widetilde\nu_t;u)\right|\,\dd t
 \le C\|\nabla u\|_{L^\infty}\mathcal S.
\end{equation}
This proves the averaged assertion \eqref{eq:averaged_commutator_comparison}.
Every estimate so far depends only on the microscopic scale $\eta_0$ and the
bounds specified in the lemma, not on the individual radii or particle separations.

\medskip
\noindent\emph{Extension to Lipschitz fields and selection of a radius.}
It remains to prove the second assertion.
If $\|\nabla u\|_{L^\infty}=0$, then $u$ is constant; both the
commutator and the stress-energy pairing vanish, so any
$t_*\in[1/2,1]$ is admissible and \eqref{eq:common_truncation_scale_energy}
follows from \eqref{eq:scaled_radius_energy_bound}.  The stress bound
\eqref{eq:selected_truncation_stress_bound} is then immediate.  We may
therefore assume $\|\nabla u\|_{L^\infty}>0$.

We first extend the estimate from smooth compactly supported vector fields to
the given Lipschitz field and then select a common truncation parameter by
averaging.  Surface terms are interpreted using
Proposition~\ref{prop:average_trace_yukawa_green_stress}.  For the extension,
we use cutoff and mollification.  Let $\zeta$ be a mollifier.  Choose $\chi\in C_c^\infty$ with $\chi=1$ on $B_1$
and set
\[
 u_R(x):=\chi(x/R)\bigl(u(x)-u(0)\bigr),
 \qquad
 u_{R,\varepsilon}:=\zeta_\varepsilon*u_R.
\]
Since $|u(x)-u(0)|\le\|\nabla u\|_{L^\infty}|x|$, differentiation of
the cutoff gives
\begin{equation}
\label{eq:truncated_field_approximation_gradient}
 \|\nabla u_R\|_{L^\infty}
 +\|\nabla u_{R,\varepsilon}\|_{L^\infty}
 \le C_\chi\|\nabla u\|_{L^\infty},
\end{equation}
with a constant independent of $R$ and $\varepsilon$.  For fixed $R$,
$u_{R,\varepsilon}\to u_R$ locally uniformly and
$\nabla u_{R,\varepsilon}\to\nabla u_R$ almost everywhere.  Hence, for
fixed $t$, dominated convergence and
$T_\kappa[g_{\kappa,d}*\widetilde\nu_t]\in L^1$ give
\[
 \mathcal B_\kappa(\widetilde\nu_t;u_{R,\varepsilon})
 \longrightarrow \mathcal B_\kappa(\widetilde\nu_t;u_R).
\]
The three terms in \eqref{eq:point_commutator_explicit_lipschitz}
depend only on differences of the transport field.  The uniform Lipschitz
bound \eqref{eq:truncated_field_approximation_gradient} therefore gives
the same integrable majorant as in the absolute convergence estimate following
\eqref{eq:point_commutator_explicit_lipschitz}, and dominated convergence
yields
$\cH_N^{g_{\kappa,d}}(u_{R,\varepsilon})\to
 \cH_N^{g_{\kappa,d}}(u_R)$.

Next let $R\to\infty$.  At every fixed particle location and for almost
every background point, $u_R\to u-u(0)$ and
$\nabla u_R\to\nabla u$.  Estimate
\eqref{eq:truncated_field_approximation_gradient} again dominates the
stress and commutator integrands.  Therefore
\[
 \cH_N^{g_{\kappa,d}}(u_R)-\mathcal B_\kappa(\widetilde\nu_t;u_R)
 \longrightarrow
 \cH_N^{g_{\kappa,d}}(u)-\mathcal B_\kappa(\widetilde\nu_t;u)
\]
for almost every $t\in[1/2,1]$.  The constant vector $u(0)$ contributes to
neither expression.  The estimates are uniform in $R$ and $\varepsilon$, so
Fatou's lemma yields
\[
 \int_{1/2}^1
 |\cH_N^{g_{\kappa,d}}(u)-\mathcal B_\kappa(\widetilde\nu_t;u)|\,\dd t
 \le C\|\nabla u\|_{L^\infty}\mathcal S.
\]
Let $G\subset[1/2,1]$ be the full-measure set on which the source identities,
the one-sided traces, and the stress-energy identity are defined.
By \eqref{eq:scaled_truncation_stress_continuity}, the map
$
 t\longmapsto
 |\cH_N^{g_{\kappa,d}}(u)-\mathcal B_\kappa(\widetilde\nu_t;u)|
$
is continuous, and hence Borel measurable, on \([1/2,1]\).  Since
\([1/2,1]\setminus G\) is null, the preceding integral estimate is
unchanged when restricted to \(G\).  As the interval has length \(1/2\),
there must exist \(t_*\in G\) such that
\[
 |\cH_N^{g_{\kappa,d}}(u)-\mathcal B_\kappa(\widetilde\nu_{t_*};u)|
 \le 2C\|\nabla u\|_{L^\infty}\mathcal S;
\]
otherwise the same quantity would exceed
$2C\|\nabla u\|_{L^\infty}\mathcal S$ almost everywhere on \(G\),
contradicting the integral bound.  Thus $t_*$ can be chosen in $G$, so all
source and trace identities hold at this value.  The uniform energy estimate
\eqref{eq:scaled_radius_energy_bound} also holds at the same \(t_*\).  Setting \(\widetilde r_i=t_*r_i\), the preceding integral estimate gives
\eqref{eq:common_truncation_scale_remainder}; the same uniform estimate gives
\eqref{eq:common_truncation_scale_energy}.  Finally,
\eqref{eq:finite_energy_stress_bound} and
\eqref{eq:common_truncation_scale_energy} give
\eqref{eq:selected_truncation_stress_bound}.  This proves the lemma.
\end{proof}

\begin{proof}[Proof of Theorem~\ref{thm:yukawa_commutator}]
Fix $B=B(d,\kappa_*,\Lambda)$ as in
Proposition~\ref{prop:yukawa_energy_variable_radii} and let
$\mathcal S$ be \eqref{eq:commutator_shift}.  Then
\eqref{eq:modulated_energy_with_error_positive} gives
$\mathcal S\ge\varepsilon_N^\kappa(\eta)$, which is exactly
\eqref{eq:yukawa_static_shift}.  Lemma~\ref{lem:averaged_commutator_comparison}
gives, with the same $t_*$,
\[
 |\cH_N^{g_{\kappa,d}}(X_N,\rho;u)|
 \le
 |\cH_N^{g_{\kappa,d}}(X_N,\rho;u)
      -\mathcal B_\kappa(\widetilde\nu_{t_*};u)|
 +|\mathcal B_\kappa(\widetilde\nu_{t_*};u)|
 \le C\|\nabla u\|_{L^\infty}\mathcal S.
\]
This proves \eqref{eq:yukawa_static_main}.  The uniformity follows from
\eqref{eq:final_microscopic_scale_dependency} and the dependence of the
constants recorded at the beginning of this section.  The stress constant in
\eqref{eq:finite_energy_stress_bound} depends only on $d$.
\end{proof}

\begin{proof}[Proof of Corollary~\ref{cor:yukawa_commutator_optimized_scale}]
Put $c_0:=\frac12\min\{1,\eta_0\}$, so that
$\eta_N=c_0N^{-1/2}$ in $d=2$ and
$\eta_N=c_0N^{-1/d}$ in $d\ge3$.  Thus $\eta_N<\eta_0$ for every
$N\ge2$.  If $d\ge3$, then
\[
 \frac{\eta_N^{2-d}}N=c_0^{2-d}N^{-2/d},
 \qquad
 \eta_N^2=c_0^2N^{-2/d},
 \qquad
 \kappa^2\eta_N^2\le\kappa_*^2c_0^2N^{-2/d}.
\]
If $d=2$, then
\[
 \frac{1+\log_+(1/\eta_N)}N
 \asymp_{\eta_0}\frac{1+\log N}{N},
 \qquad
 \eta_N^2=\frac{c_0^2}{N},
 \qquad
 \kappa^2\eta_N^2\le\frac{\kappa_*^2c_0^2}{N}.
\]
Consequently there are constants
$0<c_1=c_1(d,\kappa_*)\le C_1=C_1(d,\kappa_*)$ such that
\begin{equation}\label{eq:optimized_error_two_sided}
 c_1a_{N,d}\le\varepsilon_N^\kappa(\eta_N)
 \le C_1a_{N,d}
 \qquad(0<\kappa\le\kappa_*).
\end{equation}
Apply Theorem~\ref{thm:yukawa_commutator} at $\eta=\eta_N$ and absorb
$B C_1$ into the displayed constant $C$.  The lower bound in the theorem,
together with the left-hand inequality in
\eqref{eq:optimized_error_two_sided}, gives the quantitative lower bound
$c\,a_{N,d}$ after increasing $C$ if necessary.  The commutator estimate
follows from the upper bound in \eqref{eq:optimized_error_two_sided}.
\end{proof}

\section{Mean-field dynamics and propagation of chaos}\label{sec:mean_field_dynamics}

This section combines the commutator estimate with the exact modulated energy
identity and a quadratic transport estimate.  Throughout, fix
$0<\kappa\le\kappa_*$ and write $g=g_{\kappa,d}$.

Let \(\rho\) be a probability density solving \eqref{eq:pde_intro} on
\([0,T]\) and assume
\begin{equation}
\label{eq:rho_assumption}
  \rho\in C\bigl([0,T];(\mathcal P_2(\R^d),W_2)\bigr).
\end{equation}
Write \(u_\rho=-\nabla g*\rho\) and assume that the quantity \(\Lambda_T\)
defined in \eqref{eq:dynamic_solution_bound} is finite.  These assumptions are
imposed on the limiting solution.  The $L^\infty$ bound controls the
background terms, while the $W^{1,\infty}$ bound on $u_\rho$ enters the
characteristic comparison.  Continuity in $W_2$ on the compact interval
ensures finite second moments, but these moments enter neither the commutator
estimate nor the Gronwall constant.  The theorem is conditional on these
bounds.  Subsection~\ref{subsec:yukawa_classical_limiting_solutions}
constructs classical solutions that satisfy them uniformly in the screening
parameter.

By \eqref{eq:dynamic_solution_bound}, the background bound
\eqref{eq:static_background_bound} holds uniformly in time.  Fix throughout
\[
 B=B(d,\kappa_*,\Lambda_T),\qquad
 \eta_0=\eta_0(d,\kappa_*)
\]
as in Theorem~\ref{thm:yukawa_commutator}.  Once the common bound
$\Lambda_T$ is fixed, these choices are independent of the particular value
of $\kappa\in(0,\kappa_*]$ and are valid at every $t\in[0,T]$.

\subsection{Particle flow and microscopic energy dissipation}
We first establish global well-posedness of the Yukawa particle system in
every dimension $d\ge2$.  The microscopic energy dissipation is the standard
gradient flow identity; compare, for example,
\cite{Duerinckx2016,JiangQiaoWuZhangCoulombRiesz2026}.  We record it together
with the collision barrier and moment estimate needed for global continuation.
The proof uses only positivity, radial monotonicity, and divergence of the
Yukawa kernel at the origin.  In
dimension two the microscopic energy in
this subsection always uses the positive, unshifted kernel $g_{\kappa,2}$;
the additive normalization $g_{\kappa,2}^{\sharp}$ is used only for the
modulated energy normalization and the Coulomb limit.

\begin{proposition}[Global well-posedness, energy dissipation, and moment bounds for the Yukawa particle system]\label{prop:yukawa_particle_flow_moment}
Fix $d\ge2$, $N\ge2$, and $\kappa>0$, and let
\(X_N^0\in(\R^d)^N\setminus\Delta_N\).  Then
\eqref{eq:particle_system_intro} has a unique global smooth solution in
\((\R^d)^N\setminus\Delta_N\).  With
\[
 K_i(X_N):=\frac1N\sum_{j\ne i}\nabla g_{\kappa,d}(x_i-x_j),
\]
one has, for every $t\ge0$,
\begin{equation}
\label{eq:yukawa_particle_energy_dissipation_integrated}
 E_{N,\kappa}(X_N(t))
 +\int_0^t\frac1N\sum_{i=1}^N|K_i(X_N(s))|^2\,\dd s
 =E_{N,\kappa}(X_N^0),
\end{equation}
and hence
\begin{equation}
\label{eq:yukawa_particle_energy_dissipation_differential}
 \frac{\dd}{\dd t}E_{N,\kappa}(X_N(t))
 =-\frac1N\sum_{i=1}^N|K_i(X_N(t))|^2\le0.
\end{equation}
Moreover,
\begin{equation}
\label{eq:yukawa_particle_moment_bound}
 m_N(X_N(t))^{1/2}
 \le m_N(X_N^0)^{1/2}
      +t^{1/2}E_{N,\kappa}(X_N^0)^{1/2},
\end{equation}
and therefore
\begin{equation}
\label{eq:yukawa_particle_moment_bound_squared}
 m_N(X_N(t))
 \le 2m_N(X_N^0)+2tE_{N,\kappa}(X_N^0).
\end{equation}
In particular, for every $T<\infty$ the right-hand side gives a bound uniform
on $0\le t\le T$ whose numerical constant is independent of $d$, $N$, and
$\kappa$.

If \(\rho_N^0\) is symmetric, \(\rho_N^0(\Delta_N)=0\), and
\[
 \int\bigl(m_N+E_{N,\kappa}\bigr)\,\dd\rho_N^0<\infty,
\]
then \(\rho_N(t):=(\Phi_t^{N,\kappa})_\#\rho_N^0\) is symmetric and
\[
 \sup_{0\le t\le T}\int m_N\,\dd\rho_N(t)<\infty
 \qquad(T<\infty).
\]
\end{proposition}

\begin{proof}
The vector field is smooth on $(\R^d)^N\setminus\Delta_N$, so there is a
unique maximal smooth solution.  Since the particle system is the gradient
flow of $E_{N,\kappa}$, the standard finite-dimensional gradient flow identity
gives
\[
 \frac{\dd}{\dd t}E_{N,\kappa}(X_N(t))
 =-\frac1N\sum_i|K_i(X_N(t))|^2,
\]
which proves \eqref{eq:yukawa_particle_energy_dissipation_differential} and,
after integration, \eqref{eq:yukawa_particle_energy_dissipation_integrated}.
Because $g_{\kappa,d}$ is positive and strictly decreasing, and
$g_{\kappa,d}(r)\to+\infty$ as $r\downarrow0$, if
$d_N(t):=\min_{i\ne j}|x_i(t)-x_j(t)|$ then
\[
 \frac1{N^2}g_{\kappa,d}(d_N(t))
 \le E_{N,\kappa}(X_N(t))\le E_{N,\kappa}(X_N^0).
\]
Hence $d_N(t)$ stays bounded away from zero on every finite time interval, so
no finite-time collision is possible.

For $m_N=N^{-1}\sum_i|x_i|^2$, Cauchy--Schwarz and the dissipation identity give
\[
 \frac{\dd}{\dd t}m_N^{1/2}
 \le \left(\frac1N\sum_i|K_i|^2\right)^{1/2}
 \quad\text{a.e.\ on }\{m_N>0\},
\]
and the usual $m_N+\delta$ regularization followed by Cauchy--Schwarz in time
yields \eqref{eq:yukawa_particle_moment_bound};
\eqref{eq:yukawa_particle_moment_bound_squared} follows from
$(a+b)^2\le2a^2+2b^2$.  The collision barrier and this moment bound keep the
trajectory in a compact subset of $(\R^d)^N\setminus\Delta_N$ on every finite
interval, so the ODE continuation theorem gives global existence.  Finally,
smooth dependence and permutation equivariance of the flow imply symmetry of
$(\Phi_t^{N,\kappa})_\#\rho_N^0$, and integrating
\eqref{eq:yukawa_particle_moment_bound_squared} gives the asserted moment bound for the $N$-particle law.
\end{proof}

\smallskip\noindent\emph{Coulomb case for the particle flow.}
The same argument covers $\kappa=0$ when $d\ge3$.  In dimension two the
unscreened logarithmic interaction is not bounded below at infinity, so the
preceding collision bound and moment estimate do not yield the same global
statement for the particle dynamics.  No corresponding two-dimensional particle result is used below.

\subsection{N-particle laws and Wasserstein distance}
\label{subsec:particle_distribution_preliminaries}

\smallskip\noindent\emph{Initial distributions with finite Lyapunov functional.}
Fix $0<\eta<\eta_0$ and let $\rho_N^0\in\mathcal P_2((\R^d)^N)$ be
symmetric with $\mathfrak E_N^{\kappa,0}(\eta)<\infty$.  The corrected
modulated energy
$\cF_N^{\kappa,\sharp}+B\varepsilon_N^\kappa(\eta)$ is nonnegative and is
set equal to $+\infty$ on $\Delta_N$.  Hence
$\rho_N^0(\Delta_N)=0$.  Moreover,
\[
 \cF_N^{g_{\kappa,d}}(X_N,\rho(0))
 =2E_{N,\kappa}(X_N)
  -\frac2N\sum_iV_{\rho(0)}^\kappa(x_i)+I_\kappa[\rho(0)].
\]
For fixed $\kappa>0$, the screened background potential is bounded under the
$L^\infty$ hypothesis.  Since the additive normalization is finite, the
finiteness of the shifted modulated energy integral therefore implies
$\int E_{N,\kappa}\,\dd\rho_N^0<\infty$.  Finally,
$\rho_N^0\in\mathcal P_2$ gives $\int m_N\,\dd\rho_N^0<\infty$.
Thus Proposition~\ref{prop:yukawa_particle_flow_moment} applies.  Smooth
dependence of the flow implies that
$(t,X_N^0)\mapsto\Phi_t^{N,\kappa}(X_N^0)$ is Borel on
$[0,\infty)\times((\R^d)^N\setminus\Delta_N)$.  If a map on the whole
configuration space is needed, we extend the flow arbitrarily to the closed
set $\Delta_N$, for instance by the identity.  The extension is jointly Borel.
Because $\rho_N^0(\Delta_N)=0$, the push-forward below is independent of how
the flow is defined on $\Delta_N$.  We therefore set
$
  \rho_N(t)=(\Phi_t^{N,\kappa})_\#\rho_N^0.
$
In particular,
\begin{equation*}
  \sup_{0\le t\le T}\int \frac1N\sum_{i=1}^N |x_i|^2\,\dd \rho_N(t)(X_N)
  <\infty
\end{equation*}
whenever
\(
  \int m_N\,\dd \rho_N^0
  +\int E_{N,\kappa}\,\dd \rho_N^0<\infty
\).
We set $g_{\kappa,d}(0)=+\infty$.  At fixed $\kappa>0$,
$g_{\kappa,d}\in L^1(\R^d)$ and translation is continuous in $L^1$; hence,
for bounded $\rho(t)$,
\[
 |V_\rho^\kappa(x+h)-V_\rho^\kappa(x)|
 \le \|\rho(t)\|_{L^\infty}
      \|g_{\kappa,d}(\cdot+h)-g_{\kappa,d}\|_{L^1}\longrightarrow0.
\]
Thus the particle--background term is continuous in $X_N$.  Monotone
truncation of the positive particle--particle interaction shows that
$X_N\mapsto\cF_N^{g_{\kappa,d}}(X_N,\rho(t))$ is lower semicontinuous as an
extended-real-valued function.  It is finite on
$(\R^d)^N\setminus\Delta_N$ and equals $+\infty$ on $\Delta_N$.  Since
$\cF_N^{\kappa,\sharp}=\cF_N^{g_{\kappa,d}}-c_{\kappa,d}/N$, the same is
true for the modulated energy with the chosen additive normalization.  By
Theorem~\ref{thm:yukawa_commutator},
\[
 \cF_N^{\kappa,\sharp}+B\varepsilon_N^\kappa(\eta)\ge0.
\]
The same near-field truncation, together with narrow continuity of $\rho$ and
the uniform $L^\infty$ bound, gives jointly Borel representatives for the
particle--background term and $I_\kappa[\rho(t)]$.  Hence
$\cF_N^{\kappa,\sharp}+B\varepsilon_N^\kappa(\eta)$ is a nonnegative
extended-real Borel function of $(t,X_N)$.  This is the measurability needed
for the coupling argument below.

We extend $\cD_N^\kappa(\cdot,\rho(t))$ by $+\infty$ on the collision set;
off $\Delta_N$ it is given by \eqref{eq:intro_mean_square_force_error}.  The
particle force is continuous away from collisions.  The background force
$(t,x)\mapsto\nabla g_{\kappa,d}*\rho(t)(x)$ has a Borel representative as
well.  Indeed, truncate the kernel near the origin and at infinity, use narrow
continuity of $\rho$ for the truncated convolution, and then remove the
truncations using the uniform near-field bound and the integrable Yukawa tail.
Thus $(t,X_N)\mapsto\cD_N^\kappa(X_N,\rho(t))$ is a nonnegative Borel
function.
Since $\rho_N(t)(\Delta_N)=0$, the value assigned on $\Delta_N$ never affects
the integral with respect to the $N$-particle law.

\subsection{Modulated energy and Wasserstein stability}\label{subsec:modulated_energy_wasserstein}

We now combine the exact modulated energy identity with the quadratic
transport cost.

\subsubsection{Time derivative of the modulated energy}

Differentiation of singular modulated energies along continuity equations is
standard; see \cite[Lemma~2.1]{Serfaty2020} and
\cite[Section~2]{Duerinckx2016}, as well as \cite{Nguyen2021}.  The Yukawa
kernel has the same local Coulomb singularity and a better tail, so the same
regularization argument applies.  We record only the perturbed identity needed
in both the Yukawa mean-field theorem and the direct three-dimensional
Yukawa--Coulomb comparison.

Write $g=g_{\kappa,d}$ and set
\[
 K_i:=\frac1N\sum_{j\ne i}\nabla g(x_i-x_j),
 \qquad
 R_i:=K_i-\nabla g*\rho(x_i)=K_i+u_\rho(x_i),
 \qquad
 \cD_N:=\frac1N\sum_i|R_i|^2.
\]

\begin{lemma}[Exact perturbed chain rule for the modulated energy]
\label{lem:exact_perturbed_modulated_energy_chain_rule}
Let $g=g_{\kappa,d}$ with $d\ge2$ and $\kappa>0$; when $d\ge3$ we also
allow $\kappa=0$ with the Coulomb kernel.  Let
$\rho\in C([0,T];\mathcal P(\R^d))\cap L^\infty(0,T;L^\infty)$ solve
\[
 \partial_t\rho+\nabla\cdot(\rho u_\rho)=0,
 \qquad u_\rho=-\nabla g*\rho\in L^1(0,T;W^{1,\infty}),
\]
and assume that
$X_N(t)\notin\Delta_N$ on $[0,T]$ and that the particle trajectories belong
to $W^{1,\infty}(0,T;\R^d)$ and satisfy
\begin{equation}
\label{eq:general_perturbed_particle_velocity}
 \dot x_i=-K_i-Z_i
\end{equation}
for some $Z_i\in L^1_{\mathrm{loc}}(0,T;\R^d)$.  Then
$t\mapsto\cF_N^g(X_N(t),\rho(t))$ is absolutely continuous and, for almost
every time,
\begin{equation}
\label{eq:general_perturbed_modulated_energy_identity}
 \frac{\dd}{\dd t}\cF_N^g(X_N,\rho)
 =\cH_N^g(X_N,\rho;u_\rho)
 -2\cD_N
 -\frac2N\sum_{i=1}^N R_i\cdot Z_i.
\end{equation}
Equivalently, for every $0\le t_1\le t_2\le T$,
\begin{equation}
\label{eq:modulated_energy_integrated_all_times}
\begin{aligned}
 \cF_N^g(X_N(t_2),\rho(t_2))-\cF_N^g(X_N(t_1),\rho(t_1))
 =\int_{t_1}^{t_2}\!\left[
 \cH_N^g(X_N,\rho;u_\rho)-2\cD_N
 -\frac2N\sum_iR_i\cdot Z_i\right]\dd t.
\end{aligned}
\end{equation}
\end{lemma}

\begin{proof}
We give the calculation because the perturbed form is used twice below and it
also fixes the signs and normalizations.  Write
\[
 V(t,x):=(g*\rho(t))(x),\qquad u_\rho=-\nabla V,
\]
and, for the moment, suppress the time variable.  Since $X_N(t)$ is continuous
and collision free on the compact interval $[0,T]$, there is a number
$\delta>0$ such that
\[
 \min_{0\le t\le T}\min_{i\ne j}|x_i(t)-x_j(t)|\ge\delta.
\]
Thus the particle--particle part of the energy is a classical absolutely
continuous function of time.  The remaining differentiations can be justified
by first cutting off $g$ near the origin and at infinity, applying the
continuity equation to the resulting smooth compactly supported kernels, and
then removing the cutoffs.  Indeed, for Yukawa interactions $g$ and $\nabla g$
have the Coulomb local singularities $|x|^{2-d}$ and $|x|^{1-d}$ (with the
usual logarithmic modification in $d=2$) and exponential decay at infinity.
For the Coulomb kernel in $d\ge3$, both $g$ and $\nabla g$ are locally
integrable and bounded on $\{|x|\ge1\}$.  Since $\rho\in L^\infty\cap L^1$
and $u_\rho\in W^{1,\infty}$, dominated convergence applies to all terms
below.  This standard regularization also proves the asserted absolute
continuity.

The off-diagonal expansion is
\begin{equation}\label{eq:chain_rule_expanded_energy}
 \cF_N^g(X_N,\rho)
 =\frac1{N^2}\sum_{i\ne j}g(x_i-x_j)
 -\frac2N\sum_{i=1}^N V(x_i)
 +\int_{\R^d}V(x)\rho(x)\,\dd x.
\end{equation}
For the first term,
\begin{equation}\label{eq:chain_rule_particle_particle_derivative}
 \frac{\dd}{\dd t}
 \left(\frac1{N^2}\sum_{i\ne j}g(x_i-x_j)\right)
 =\frac2N\sum_{i=1}^N K_i\cdot\dot x_i.
\end{equation}
Next, the continuity equation gives, for almost every $t$ and every fixed $x$,
\begin{equation}\label{eq:chain_rule_V_time_derivative}
 \partial_tV(t,x)
 =-\int_{\R^d}\nabla g(x-y)\cdot u_\rho(t,y)\rho(t,y)\,\dd y.
\end{equation}
Set
\[
 J_i:=\int_{\R^d}\nabla g(x_i-y)\cdot u_\rho(y)\rho(y)\,\dd y.
\]
Since $\nabla V(x_i)=-u_\rho(x_i)$,
\begin{equation}\label{eq:chain_rule_particle_background_derivative}
 \frac{\dd}{\dd t}\left(-\frac2N\sum_iV(x_i)\right)
 =\frac2N\sum_i u_\rho(x_i)\cdot\dot x_i
 +\frac2N\sum_iJ_i.
\end{equation}
The same regularized continuity equation calculation, now in both variables,
yields
\begin{equation}\label{eq:chain_rule_background_background_derivative}
 \frac{\dd}{\dd t}\int V\rho
 =2\int \nabla V\cdot u_\rho\,\rho
 =-2\int |u_\rho|^2\rho.
\end{equation}
Combining \eqref{eq:chain_rule_particle_particle_derivative}--
\eqref{eq:chain_rule_background_background_derivative} and using
$R_i=K_i+u_\rho(x_i)$ gives
\begin{equation}\label{eq:chain_rule_derivative_intermediate}
 \frac{\dd}{\dd t}\cF_N^g(X_N,\rho)
 =\frac2N\sum_iR_i\cdot\dot x_i
 +\frac2N\sum_iJ_i
 -2\int |u_\rho|^2\rho.
\end{equation}
Inserting $\dot x_i=-K_i-Z_i$ therefore gives
\begin{equation}\label{eq:chain_rule_derivative_with_perturbation_intermediate}
 \frac{\dd}{\dd t}\cF_N^g(X_N,\rho)
 =-\frac2N\sum_iR_i\cdot K_i
 -\frac2N\sum_iR_i\cdot Z_i
 +\frac2N\sum_iJ_i
 -2\int |u_\rho|^2\rho.
\end{equation}

It remains to identify the first, third, and fourth terms with the commutator
minus $2\cD_N$.  Expanding the off-diagonal commutator and using its symmetry
under interchange of $x$ and $y$, we obtain
\begin{align}
 \cH_N^g(X_N,\rho;u_\rho)
 &=\frac2N\sum_i u_\rho(x_i)\cdot K_i
   -\frac2N\sum_i\int
     \bigl(u_\rho(x_i)-u_\rho(y)\bigr)
     \cdot\nabla g(x_i-y)\rho(y)\,\dd y
 \notag\\
 &\qquad
   +\iint
     \bigl(u_\rho(x)-u_\rho(y)\bigr)\cdot\nabla g(x-y)
     \rho(x)\rho(y)\,\dd x\dd y
 \notag\\
 &=\frac2N\sum_i u_\rho(x_i)\cdot R_i
   +\frac2N\sum_iJ_i
   -2\int|u_\rho|^2\rho.
 \label{eq:chain_rule_commutator_expansion}
\end{align}
Here we used $\nabla g*\rho=-u_\rho$ in the last equality.  Since
$\cD_N=N^{-1}\sum_i|R_i|^2$ and $R_i-u_\rho(x_i)=K_i$,
\begin{equation}\label{eq:chain_rule_commutator_minus_dissipation}
 \cH_N^g(X_N,\rho;u_\rho)-2\cD_N
 =-\frac2N\sum_iR_i\cdot K_i
   +\frac2N\sum_iJ_i
   -2\int|u_\rho|^2\rho.
\end{equation}
Comparing \eqref{eq:chain_rule_derivative_with_perturbation_intermediate}
with \eqref{eq:chain_rule_commutator_minus_dissipation} proves
\eqref{eq:general_perturbed_modulated_energy_identity}.  The right-hand side
belongs to $L^1(0,T)$.  Indeed, the collision-free compact trajectory keeps
$K_i$ and hence $R_i$ bounded, while $u_\rho$ is bounded; moreover,
\eqref{eq:general_perturbed_particle_velocity} and
$\dot x_i\in L^\infty(0,T)$ imply that $Z_i=-\dot x_i-K_i$ is itself bounded
on $(0,T)$.  Hence the absolutely continuous representative satisfies the
integrated identity \eqref{eq:modulated_energy_integrated_all_times} for every
$0\le t_1\le t_2\le T$.
\end{proof}

In particular, $Z_i\equiv0$ is the exact modulated energy identity used for
the Yukawa mean-field theorem, whereas the perturbation term in
Lemma~\ref{lem:exact_perturbed_modulated_energy_chain_rule} is the only
additional term needed in the Yukawa--Coulomb comparison.

\subsubsection{Wasserstein stability}

For $u\in L^1_tW^{1,\infty}_x$, the continuity equation is transported by the
unique bi-Lipschitz Carath\'eodory flow; see, for instance,
\cite[Chapter~8]{AmbrosioGigliSavare2008}.  Thus, if
$\rho:[0,T]\to\mathcal P(\R^d)$ is narrowly continuous and solves
$\partial_t\rho+\nabla\cdot(\rho u)=0$, and if $\Psi_{t,s}$ denotes the
flow generated by $u$, then for every $0\le s,t\le T$,
\begin{equation}
\label{eq:limiting_solution_characteristic_representation}
 \rho(t)=(\Psi_{t,s})_\#\rho(s).
\end{equation}
In particular, writing $\Psi_t:=\Psi_{t,0}$, one has
$\rho(t)=(\Psi_t)_\#\rho(0)$.  We use this standard representative in the
coupling argument below.

We combine the standard Dobrushin quadratic coupling estimate with the
modulated energy dissipation.  See \cite{Dobrushin1979} for the original
coupling argument and \cite{JiangQiaoWuZhangCoulombRiesz2026} for the same
dissipative combination in the Coulomb and Riesz settings.  For completeness,
we record the short calculation in our normalization.

Let
\[
 \cQ_N(t):=\frac1N\sum_i|x_i(t)-y_i(t)|^2,
 \qquad
 \cS_N(t):=\mathcal M_N(t)+B\varepsilon_N\ge0,
 \qquad
 \mathcal L_N(t):=\cQ_N(t)+\cS_N(t),
\]
where
$\mathcal M_N'=\cH_N^g-2\cD_N$ and
$|\cH_N^g|\le\gamma(t)\cS_N$.
For characteristic trajectories $\dot y_i=u_\rho(y_i)$ and particle
trajectories $\dot x_i=-K_i$, recall that
$R_i=K_i+u_\rho(x_i)$, so that
$\dot x_i=u_\rho(x_i)-R_i$.  Therefore, for almost every time,
\begin{align*}
 \frac{\dd}{\dd t}\cQ_N
 &=\frac2N\sum_i(x_i-y_i)\cdot
   \bigl(u_\rho(x_i)-u_\rho(y_i)-R_i\bigr)\\
 &\le 2\|\nabla u_\rho\|_{L^\infty}\cQ_N
   +\cQ_N+\cD_N,
\end{align*}
where Young's inequality was used in the last term.  On the other hand,
\[
 \cS_N'=\mathcal M_N'
 \le \gamma(t)\cS_N-2\cD_N.
\]
Adding the two inequalities gives
\begin{equation}
\label{eq:abstract_transfer_differential}
 \frac{\dd}{\dd t}\mathcal L_N(t)
 \le
 \bigl(1+2\|\nabla u_\rho(t)\|_{L^\infty}+\gamma(t)\bigr)
 \mathcal L_N(t)-\cD_N(t).
\end{equation}
With
\[
 \Gamma(t):=\int_0^t
 \bigl(1+2\|\nabla u_\rho(s)\|_{L^\infty}+\gamma(s)\bigr)\,\dd s,
\]
multiplication by $e^{-\Gamma}$ and integration give
\begin{equation}
\label{eq:abstract_transfer_dissipation}
 e^{-\Gamma(t)}\mathcal L_N(t)
 +\int_0^t e^{-\Gamma(s)}\cD_N(s)\,\dd s
 \le \mathcal L_N(0),
\end{equation}
and in particular
\begin{equation}
\label{eq:abstract_transfer_pathwise}
 \cQ_N(t)\le\mathcal L_N(t)\le e^{\Gamma(t)}\mathcal L_N(0).
\end{equation}
When the corrected modulated energy term is viewed as a function of the
current particle configuration rather than along one trajectory, we write it
as $\cS_N(t,X_N)$.

Now let $\pi_0$ be any coupling of $\rho_N^0$ and
$\rho(0)^{\otimes N}$, and push it forward by the particle flow together with
the coordinatewise characteristic flow.  The resulting $\pi_t$ couples
$\rho_N(t)$ and $\rho(t)^{\otimes N}$.  Its transport cost is
$N\int\cQ_N(t)\,\dd\pi_0$, while $\cS_N(t)$ depends only on the particle
coordinate.  The pathwise estimate therefore yields
\begin{equation}
\label{eq:abstract_transfer_distribution}
 \frac1N W_2^2(\rho_N(t),\rho(t)^{\otimes N})
 +\int \cS_N(t,X_N)\,\dd\rho_N(t)(X_N)
 \le e^{\Gamma(t)}
 \int\bigl(\cQ_N(0)+\cS_N(0)\bigr)\,\dd\pi_0.
\end{equation}
Likewise, integrating \eqref{eq:abstract_transfer_dissipation} with respect
to $\pi_0$ and using that $\cD_N(t)$ depends only on the particle
coordinate gives
\begin{equation}
\label{eq:abstract_transfer_distribution_dissipation}
 \int_0^T\int\cD_N(t)\,\dd\rho_N(t)\,\dd t
 \le e^{\Gamma(T)}
 \int\bigl(\cQ_N(0)+\cS_N(0)\bigr)\,\dd\pi_0.
\end{equation}
No differentiation of the $N$-particle law is needed.  The argument is
pathwise, along collision-free particle trajectories and the Lipschitz
characteristics of the limiting solution.  Joint Borel measurability and the
nonnegativity of $\cS_N$ and $\cD_N$ allow integration by Tonelli's theorem,
while the push-forward identities convert trajectory integrals into integrals
against $\rho_N(t)$.  In particular, the left-hand sides of
\eqref{eq:abstract_transfer_distribution}--
\eqref{eq:abstract_transfer_distribution_dissipation} are finite whenever
the initial Lyapunov functional is finite.

By Subsection~\ref{subsec:particle_distribution_preliminaries} and
Theorem~\ref{thm:yukawa_commutator}, the corrected modulated energy integrand
in \eqref{eq:intro_initial_error} is a nonnegative Borel function.  Therefore
the initial error has the coupling representation
\begin{equation}
\label{eq:initial_error_coupling_representation}
\mathfrak E_N^{\kappa,0}(\eta)
=\inf_{\pi_0\in\Pi(\rho_N^0,\rho(0)^{\otimes N})}
 \int\left[
 \frac1N\sum_{i=1}^N|x_i-y_i|^2
 +\cF_N^{\kappa,\sharp}(X_N,\rho(0))
 +B\varepsilon_N^\kappa(\eta)
 \right]\,\dd\pi_0.
\end{equation}
The shifted modulated energy depends only on the particle coordinate $X_N$,
so its integral is the same for every coupling with first marginal
$\rho_N^0$.  The minimization in
\eqref{eq:initial_error_coupling_representation} therefore reduces to the
quadratic Wasserstein problem.  Since both marginals belong to
$\mathcal P_2$, an optimal $W_2$ coupling exists and the infimum is attained.

For symmetric probability measures $\rho_N$ and $\sigma_N$ on $(\R^d)^N$
with finite second moments, we use the standard symmetric projection estimate
\begin{equation}
\label{eq:exact_symmetric_marginal_projection}
 W_2^2(\rho_{N:k},\sigma_{N:k})
 \le \frac{k}{N}W_2^2(\rho_N,\sigma_N),
 \qquad 1\le k\le N.
\end{equation}
This is the quadratic-cost version of the Wasserstein projection inequality on
product spaces; see \cite[Proposition~2.6 and the discussion following
it]{HaurayMischler2014}, whose argument applies verbatim to $W_2$.

\subsection{Mean-field convergence for the Yukawa flow}\label{subsec:yukawa_meanfield_proof}

We first prove Theorem~\ref{thm:yukawa_meanfield} for an arbitrary admissible
radius $0<\eta<\eta_0$.  We then choose the optimized radius $\eta_N$ from
Corollary~\ref{cor:yukawa_commutator_optimized_scale}.  By construction,
\begin{equation}\label{eq:admissible_radius_scale_bound}
 \varepsilon_N^\kappa(\eta_N)\le C a_{N,d}.
\end{equation}

\begin{proof}[Proof of Theorem~\ref{thm:yukawa_meanfield}]
The dynamic bound \eqref{eq:dynamic_solution_bound} implies the
hypothesis \eqref{eq:static_background_bound} at every time.
By Subsection~\ref{subsec:particle_distribution_preliminaries}, finiteness of
$\mathfrak E_N^{\kappa,0}(\eta)$ implies that $\rho_N^0$ assigns zero mass to
the collision set and has finite expected microscopic Yukawa energy.
Proposition~\ref{prop:yukawa_particle_flow_moment} therefore yields the global
particle flow.  The unshifted and normalized modulated energies differ only by
the time-independent constant $c_{\kappa,d}/N$.  Lemma~\ref{lem:exact_perturbed_modulated_energy_chain_rule} gives
\[
 \frac{\dd}{\dd t}\cF_N^{g_{\kappa,d}}
 =\cH_N^{g_{\kappa,d}}-2\cD_N^\kappa.
\]
Since $c_{\kappa,d}/N$ is time independent, the same identity holds with
$\cF_N^{\kappa,\sharp}$ in place of the uncorrected modulated energy.  We therefore apply the combined Dobrushin and modulated energy estimate
\eqref{eq:abstract_transfer_distribution}--
\eqref{eq:abstract_transfer_distribution_dissipation} with
\[
 \mathcal M_N(t)=\cF_N^{\kappa,\sharp}(X_N(t),\rho(t)),
 \qquad
 \varepsilon_N=\varepsilon_N^\kappa(\eta),
 \qquad
 \gamma(t)=C\|\nabla u_\rho(t)\|_{L^\infty},
\]
where $C=C(d,\kappa_*,\Lambda_T)$ is the constant in
Theorem~\ref{thm:yukawa_commutator}.  The lower bound gives
$\cS_N\ge0$.  If $C_{\rm com}=C(d,\kappa_*,\Lambda_T)$ denotes the
commutator constant, then
\[
 \gamma(t)=C_{\rm com}\|\nabla u_\rho(t)\|_{L^\infty},
 \qquad
 \Gamma(T)\le T\bigl[1+(2+C_{\rm com})\Lambda_T\bigr],
\]
so the Gronwall factor is uniform for $0<\kappa\le\kappa_*$.
Choose an optimal initial coupling in
\eqref{eq:initial_error_coupling_representation}.  The resulting Lyapunov and
dissipation estimates give \eqref{eq:yukawa_meanfield_full_lyapunov}, with
$C_T=C(d,T,\kappa_*,\Lambda_T)$.  Since both $\rho_N(t)$ and
$\rho(t)^{\otimes N}$ are symmetric, applying
\eqref{eq:exact_symmetric_marginal_projection} gives
\eqref{eq:yukawa_meanfield_k_marginal}.  Finally
\eqref{eq:admissible_radius_scale_bound} gives the optimized scale.
\end{proof}

\subsection{Uniform kernel estimates and classical solutions}
\label{subsec:yukawa_classical_limiting_solutions}

The local classical solutions constructed below satisfy the regularity
assumptions of Theorem~\ref{thm:yukawa_meanfield}.  The kernel estimates in
this subsection provide the uniform velocity bounds needed for both the
nonlinear flow and the Coulomb limit.
Define
\begin{equation}\label{eq:alpha_screening_parameter}
 \alpha_d(\kappa):=
 \begin{cases}
 \kappa,&d=2,\\
 \kappa^2,&d\ge3.
 \end{cases}
\end{equation}

\begin{lemma}
\label{lem:yukawa_coulomb_kernel_estimates}
Fix $d\ge2$ and $\kappa_*>0$.  For $0<\kappa\le\kappa_*$ set
$q_{\kappa,d}:=g_{\kappa,d}^{\sharp}-g_{0,d}$, where
$g_{\kappa,d}^{\sharp}=g_{\kappa,d}$ for $d\ge3$.  Then, for every $x\ne0$,
\begin{equation}\label{eq:force_difference_general_pointwise}
 |\nabla q_{\kappa,d}(x)|
 \le
 \begin{cases}
 C\kappa,&d=2,\\
 C_d\kappa^2|x|^{3-d},&d\ge3.
 \end{cases}
\end{equation}
For $d\ge3$ and $m=1,2,3$,
\begin{equation}\label{eq:screened_correction_derivatives_general}
 |D^m q_{\kappa,d}(x)|\le C_{d,\kappa_*}
 \begin{cases}
 \kappa^2|x|^{4-d-m},&0<|x|\le1,\\
 |x|^{2-d-m},&|x|>1,
 \end{cases}
\end{equation}
where the case $(d,m)=(3,1)$ is understood through the sharper bound below.
For $d=2$,
\begin{equation}\label{eq:two_dimensional_screened_correction_derivatives}
 |D^2q_{\kappa,2}(x)|\le C_{\kappa_*}
 \begin{cases}
 \kappa^2(1+|\log(\kappa|x|)|),&0<\kappa|x|\le1,\\
 |x|^{-2},&\kappa|x|\ge1,
 \end{cases}
\end{equation}
\begin{equation}\label{eq:two_dimensional_screened_correction_third}
 |D^3q_{\kappa,2}(x)|\le C_{\kappa_*}
 \begin{cases}
 \kappa^2|x|^{-1},&0<\kappa|x|\le1,\\
 |x|^{-3},&\kappa|x|\ge1.
 \end{cases}
\end{equation}
Moreover,
\begin{equation}\label{eq:uniform_third_derivative_kernel_difference}
 \sup_{0<\kappa\le\kappa_*}\|D^3q_{\kappa,d}\|_{L^1(\R^d)}<\infty,
\end{equation}
and every $\rho\in L^1(\R^d)\cap L^\infty(\R^d)$ satisfies
\begin{equation}\label{eq:kernel_difference_convolution_bound}
 \|\nabla q_{\kappa,d}*\rho\|_{L^\infty}
 \le C_d\alpha_d(\kappa)
       (\|\rho\|_{L^1}+\|\rho\|_{L^\infty}).
\end{equation}
The constant in the convolution estimate is independent of an upper bound on
$\kappa$; the parameter $\kappa_*$ enters only the higher derivative
estimates.
In dimension $d=3$ one has the sharper pointwise estimate
\begin{equation}\label{eq:single_screened_coulomb_force_difference}
 \sup_{x\ne0}|\nabla g_{\kappa,3}(x)-\nabla g_{0,3}(x)|
 \le\frac{\kappa^2}{8\pi}.
\end{equation}
Consequently, if
\[
 Z_i^\kappa(X_N)
 :=\frac1N\sum_{j\ne i}\nabla(g_{\kappa,3}-g_{0,3})(x_i-x_j),
\]
then
\begin{equation}\label{eq:particle_force_difference_square}
 \frac1N\sum_{i=1}^N|Z_i^\kappa(X_N)|^2
 \le\frac{\kappa^4}{64\pi^2}
\end{equation}
for every collision-free configuration in $(\R^3)^N$.
\end{lemma}

\begin{proof}
Assume first $d\ge3$.  With
$p_t(x)=(4\pi t)^{-d/2}e^{-|x|^2/(4t)}$, the resolvent representation gives
\[
 q_{\kappa,d}(x)
 =-\int_0^\infty(1-e^{-\kappa^2t})p_t(x)\,\dd t.
\]
The Gaussian derivative estimate
\[
 |D^mp_t(x)|\le C_{d,m}t^{-(d+m)/2}e^{-|x|^2/(8t)}
\]
therefore yields
\begin{equation}\label{eq:screening_heat_kernel_derivative_integral}
 |D^mq_{\kappa,d}(x)|
 \le C_{d,m}\int_0^\infty
 \min\{1,\kappa^2t\}t^{-(d+m)/2}e^{-|x|^2/(8t)}\,\dd t.
\end{equation}
If $|x|\le1$ and $d+m>4$, use
$1-e^{-\kappa^2t}\le\kappa^2t$ and the change of variables
$u=|x|^2/t$ to obtain
$|D^mq_{\kappa,d}(x)|\le C\kappa^2|x|^{4-d-m}$.
If $|x|>1$, use the bound by $1$ in
\eqref{eq:screening_heat_kernel_derivative_integral} to obtain
$|D^mq_{\kappa,d}(x)|\le C|x|^{2-d-m}$.
This proves \eqref{eq:screened_correction_derivatives_general} except for
$(d,m)=(3,1)$.  In that case
$g_{\kappa,3}(r)=e^{-\kappa r}/(4\pi r)$ gives, with $z=\kappa r$,
\[
 |\nabla q_{\kappa,3}(x)|
 =\frac{1-(1+z)e^{-z}}{4\pi r^2}
 \le\frac{\kappa^2}{8\pi},
\]
because $1-(1+z)e^{-z}=\int_0^z\tau e^{-\tau}\,\dd\tau\le z^2/2$.
This proves both the missing case and
\eqref{eq:single_screened_coulomb_force_difference}.  For $d\ge4$, the same
heat-kernel calculation with $m=1$ gives
$|\nabla q_{\kappa,d}(x)|\le C_d\kappa^2|x|^{3-d}$, so
\eqref{eq:force_difference_general_pointwise} holds for every $d\ge3$.
Furthermore,
\[
 \int_{|x|\le1}|D^3q_{\kappa,d}(x)|\,\dd x
 \le C_d\kappa^2\int_0^1\dd r,
 \qquad
 \int_{|x|>1}|D^3q_{\kappa,d}(x)|\,\dd x
 \le C_d\int_1^\infty r^{-2}\,\dd r,
\]
which proves \eqref{eq:uniform_third_derivative_kernel_difference} for
$d\ge3$.

For $d=2$, constants disappear after differentiation.  Put
\[
 Q(z):=K_0(z)+\log z-(\log2-\gamma_E),
 \qquad q_{\kappa,2}(x)=\frac1{2\pi}Q(\kappa|x|).
\]
Using $K_0'=-K_1$ and $K_1'=-K_0-z^{-1}K_1$,
\[
 Q'=-K_1+z^{-1},\qquad Q''=K_0+z^{-1}K_1-z^{-2}.
\]
The small-argument expansion gives $Q'(z)=O(z|\log z|)$, while
$Q'(z)=z^{-1}+O(e^{-z}z^{-1/2})$ as $z\to\infty$; hence $Q'$ is bounded
and \eqref{eq:force_difference_general_pointwise} follows in dimension two.
For a radial function the two third-order tensor coefficients are
\[
 \frac{Q''}z-\frac{Q'}{z^2}
 =\frac{K_0}z+\frac{2K_1}{z^2}-\frac2{z^3},
 \qquad
 Q'''-\frac{3Q''}z+\frac{3Q'}{z^2}
 =-K_1-\frac{4K_0}z-\frac{8K_1}{z^2}+\frac8{z^3}.
\]
The Bessel expansions at zero and infinity give
\eqref{eq:two_dimensional_screened_correction_derivatives}--
\eqref{eq:two_dimensional_screened_correction_third}.  In particular,
\[
 \int_{\kappa r\le1}\frac{\kappa^2}{r}\,\dd x
 +\int_{\kappa r\ge1}r^{-3}\,\dd x\le C\kappa,
\]
which proves \eqref{eq:uniform_third_derivative_kernel_difference} in $d=2$.

Finally, \eqref{eq:kernel_difference_convolution_bound} follows directly from
\eqref{eq:force_difference_general_pointwise}, with a constant independent of
$\kappa_*$.  In $d=2,3$ the pointwise Yukawa--Coulomb force difference is globally bounded by
$C\alpha_d(\kappa)$.  If $d\ge4$, split into $|x-y|\le1$ and $|x-y|>1$:
$|x-y|^{3-d}$ is locally integrable in the near field and is at most one in
the far field.  Thus the near part is controlled by $\|\rho\|_{L^\infty}$
and the far part by $\|\rho\|_{L^1}$.  Finally,
\eqref{eq:single_screened_coulomb_force_difference} gives, for every $i$,
\[
 |Z_i^\kappa(X_N)|
 \le \frac{N-1}{N}\frac{\kappa^2}{8\pi}
 \le \frac{\kappa^2}{8\pi},
\]
and averaging the squares yields
\eqref{eq:particle_force_difference_square}.
\end{proof}

\begin{proposition}[Uniform estimates for the velocity field]
\label{prop:yukawa_velocity_field}
Fix $d\ge2$, $0<\beta<1$, and $\kappa_*>0$.  The following estimates hold
uniformly for $0\le\kappa\le\kappa_*$.  At $\kappa=0$ we use the
logarithmic kernel in $d=2$ and the Newtonian kernel in $d\ge3$.

\smallskip
\noindent\textup{(i) Estimate for bounded densities.}
For every
$\rho\in L^1(\R^d)\cap L^\infty(\R^d)\cap C^{0,\beta}(\R^d)$,
\begin{equation}\label{eq:standard_velocity_regularity_criterion}
 \|\nabla g_{\kappa,d}*\rho\|_{W^{1,\infty}}
 \le C_{d,\beta,\kappa_*}
 \bigl(\|\rho\|_{L^1}+\|\rho\|_{L^\infty}+[\rho]_{C^{0,\beta}}\bigr).
\end{equation}

\smallskip
\noindent\textup{(ii) Local estimates on compactly supported densities.}
Fix $R>0$.  For
$f\in C_c^{1,\beta}(\R^d)$ with
$\operatorname{supp}f\subset\overline{B_R}$, set
\[
 u_f^\kappa:=-\nabla g_{\kappa,d}*f,
 \qquad
 \vartheta_f^\kappa:=\nabla\cdot u_f^\kappa
 =f-\kappa^2g_{\kappa,d}*f,
\]
with the second term omitted when $\kappa=0$.  There is
$C=C(d,R,\beta,\kappa_*)$ such that
\begin{equation}
\label{eq:yukawa_velocity_field_strong}
 \|u_f^\kappa\|_{C^{1,\beta}}
 +\|\vartheta_f^\kappa\|_{C^{1,\beta}}
 \le C\|f\|_{C^{1,\beta}}.
\end{equation}
If $f,g\in C_c^{1,\beta}(\R^d)$ have support in $\overline{B_R}$, then
\begin{equation}
\label{eq:yukawa_velocity_field_weak}
 \|u_f^\kappa-u_g^\kappa\|_{L^\infty}
 +\|\vartheta_f^\kappa-\vartheta_g^\kappa\|_{L^\infty}
 \le C\|f-g\|_{L^\infty}.
\end{equation}
For every $0<\beta'<\beta$,
\begin{equation}
\label{eq:yukawa_velocity_field_holder_continuity}
 \|u_f^\kappa-u_g^\kappa\|_{C^{1,\beta'}}
 +\|\vartheta_f^\kappa-\vartheta_g^\kappa\|_{C^{1,\beta'}}
 \le C_{d,R,\beta',\kappa_*}\|f-g\|_{C^{1,\beta'}}.
\end{equation}
\end{proposition}

\begin{proof}
The Hessian of the Coulomb kernel is a Calder\'on--Zygmund kernel with
zero angular mean.  Distributionally,
\[
 \partial_{ij}g_{0,d}
 =\operatorname{p.v.}K_{ij}-\frac1d\delta_{ij}\delta_0,
 \qquad
 \int_{\mathbb S^{d-1}}K_{ij}(\omega)\,\dd\omega=0.
\]
The H\"older cancellation estimate for Calder\'on--Zygmund operators
\cite{Stein1970}, together with the elementary near- and far-field estimate
for $\nabla g_{0,d}*\rho$, gives
\eqref{eq:standard_velocity_regularity_criterion} at $\kappa=0$.
For $\kappa>0$, additive constants do not affect the velocity and
$\nabla g_{\kappa,d}=\nabla g_{0,d}+\nabla q_{\kappa,d}$.  The force
correction is controlled by \eqref{eq:kernel_difference_convolution_bound}.
The bounds for $D^2q_{\kappa,d}$ in
\eqref{eq:screened_correction_derivatives_general} and
\eqref{eq:two_dimensional_screened_correction_derivatives} are uniformly
integrable in the near field and uniformly bounded in the far field.
This controls $D^2q_{\kappa,d}*\rho$ by
$\|\rho\|_{L^\infty}+\|\rho\|_{L^1}$ and proves part \textup{(i)}.

For part \textup{(ii)}, the Coulomb H\"older estimate gives
\[
 \|\nabla g_{0,d}*f\|_{C^{1,\beta}}
 \le C_{d,R,\beta}\bigl(\|f\|_{L^1}+\|f\|_{C^{0,\beta}}\bigr).
\]
The singular parts of $D^2g_{\kappa,d}$ and $D^2g_{0,d}$ cancel.  More
precisely, their distributional contributions supported at the origin are
identical, so $D^3q_{\kappa,d}$ has no distributional term supported at the
origin.  In particular,
$D^2q_{\kappa,d}\in W^{1,1}_{\mathrm{loc}}(\R^d)$, and for compactly
supported $f$ one has
$\nabla(D^2q_{\kappa,d}*f)=D^3q_{\kappa,d}*f$ distributionally.  By
Lemma~\ref{lem:yukawa_coulomb_kernel_estimates}, $D^3q_{\kappa,d}$ is
uniformly in $L^1$, while $\nabla q_{\kappa,d}$ and $D^2q_{\kappa,d}$ have
the required local $L^1$ bounds.  Since $f$ is supported in
$\overline{B_R}$,
\[
 \|\nabla q_{\kappa,d}*f\|_{L^\infty}
 +\|D^2q_{\kappa,d}*f\|_{L^\infty}
 \le C_{d,R,\kappa_*}(\|f\|_{L^1}+\|f\|_{L^\infty}),
\]
and
\[
 \|\nabla(D^2q_{\kappa,d}*f)\|_{L^\infty}
 \le \|D^3q_{\kappa,d}\|_{L^1}\|f\|_{L^\infty}.
\]
Hence $D^2q_{\kappa,d}*f$ is bounded and Lipschitz, which gives the
$C^{0,\beta}$ bound needed in
\eqref{eq:yukawa_velocity_field_strong}.  The estimate for
$\vartheta_f^\kappa$ follows from
\[
 \vartheta_f^\kappa=f-\kappa^2g_{\kappa,d}*f,
 \qquad
 \nabla\vartheta_f^\kappa=\nabla f+\kappa^2u_f^\kappa,
\]
together with $\kappa^2\int g_{\kappa,d}=1$ for $\kappa>0$.

For the difference estimates, local integrability of the force and the
common support give
\[
 \|u_f^\kappa-u_g^\kappa\|_{L^\infty}
 \le C_{d,R,\kappa_*}\|f-g\|_{L^\infty}.
\]
The formula for $\vartheta_f^\kappa-\vartheta_g^\kappa$ and
$\kappa^2\int g_{\kappa,d}=1$ yield the second term in
\eqref{eq:yukawa_velocity_field_weak}.  Finally, applying the Coulomb
H\"older estimate with exponent $\beta'<\beta$ to $f-g$, and using the same
$q_{\kappa,d}$ bounds, proves
\eqref{eq:yukawa_velocity_field_holder_continuity}.
\end{proof}

Let $\rho=\rho^\kappa$ be a solution of \eqref{eq:pde_intro} satisfying
\eqref{eq:rho_assumption} and, for some $0<\beta<1$,
\begin{equation}
\label{eq:holder_velocity_criterion_bound}
 1+\sup_{0\le t\le T}\left[
   \|\rho^\kappa(t)\|_{L^\infty}
   +[\rho^\kappa(t)]_{C^{0,\beta}}
 \right]<\infty.
\end{equation}
Part \textup{(i)} then implies \eqref{eq:dynamic_solution_bound}, with a
constant independent of $\kappa\in[0,\kappa_*]$.  Hence
Theorem~\ref{thm:yukawa_meanfield} applies for every
$0<\kappa\le\kappa_*$, uniformly over families satisfying a common bound in
\eqref{eq:holder_velocity_criterion_bound}.  The corresponding estimate at
$\kappa=0$ is used in the Coulomb limit below.

Local well-posedness for related aggregation equations is treated in
\cite{Laurent2007,CozziGieKelliher2017}.  The proposition below additionally
provides a common existence time and regularity bounds uniform for
$0\le\kappa\le\kappa_*$.  These uniform bounds are used both in the Yukawa
mean-field estimate and in the Coulomb limit.

\begin{proposition}[Uniform local classical solutions]
\label{prop:yukawa_classical_solutions}
Fix an integer $d\ge2$, $0<\beta<1$, and $\kappa_*>0$.  Let
$\rho_0\in C_c^{1,\beta}(\R^d)$ satisfy
\[
 \rho_0\ge0,\qquad \int_{\R^d}\rho_0\,\dd x=1.
\]
There exists $T_*>0$, depending only on $d$, $\rho_0$, $\beta$, and $\kappa_*$,
such that, for every $0\le\kappa\le\kappa_*$, the Cauchy problem
\eqref{eq:pde_intro} with the Coulomb kernel from
\eqref{eq:coulomb_kernel_general_intro} when $\kappa=0$ has a
unique classical solution $\rho^\kappa$ satisfying
\begin{equation}
\label{eq:yukawa_classical_solution_class}
 \rho^\kappa\in L^\infty_{\mathrm{loc}}\bigl([0,T_*);C^{1,\beta}(\R^d)\bigr),
 \qquad
 \rho^\kappa\in C\bigl([0,T_*);C^{1,\beta'}(\R^d)\bigr)
 \cap C^1\bigl([0,T_*);C^{0,\beta'}(\R^d)\bigr)
\end{equation}
for every $0<\beta'<\beta$.
Here ``classical'' means that the equation holds pointwise with the time and
space regularity displayed above.  Uniqueness holds within the class of
nonnegative unit-mass solutions satisfying these bounds and having compact
support on each compact time interval.  Moreover, for every $0<T<T_*$ there is a radius
$R_T=R_T(d,\beta,\kappa_*,\rho_0,T)<\infty$, independent of
$\kappa\in[0,\kappa_*]$, such that
\begin{equation*}
 \operatorname{supp}\rho^\kappa(t)\subset B_{R_T}
 \qquad(0\le t\le T,\ 0\le\kappa\le\kappa_*).
\end{equation*}
For every $0<T<T_*$,
\begin{equation}
\label{eq:yukawa_classical_solution_bound}
 \sup_{0\le\kappa\le\kappa_*}\sup_{0\le t\le T}
 \left[
   \|\rho^\kappa(t)\|_{L^\infty}
   +[\rho^\kappa(t)]_{C^{0,\beta}}
   +\int_{\R^d}|x|^2\rho^\kappa(t,x)\,\dd x
   +\|u_{\rho^\kappa}^\kappa(t)\|_{W^{1,\infty}}
 \right]<\infty.
\end{equation}
In particular, for $0<\kappa\le\kappa_*$ these solutions satisfy
\eqref{eq:rho_assumption}--\eqref{eq:dynamic_solution_bound} on every
$[0,T]\subset[0,T_*)$.
\end{proposition}

The proof is given in Appendix~\ref{app:local_classical_solutions}.

\begin{corollary}[Yukawa mean-field limit for tensorized initial data]
\label{cor:yukawa_classical_chaos}
Under the assumptions of Proposition~\ref{prop:yukawa_classical_solutions},
fix $0<T<T_*$.  There exists
$C_T=C(d,\beta,\kappa_*,\rho_0,T)<\infty$ such that, for every
$N\ge2$ and every $0<\kappa\le\kappa_*$, if $\rho^\kappa$ is the
classical solution and the particle system starts from
$\rho_N^0=\rho_0^{\otimes N}$, then
\begin{equation*}
\begin{aligned}
 &\sup_{t\le T}\left\{
 \frac1N W_2^2\bigl(\rho_N(t),\rho^\kappa(t)^{\otimes N}\bigr)
 +\int\!\left[\cF_N^{\kappa,\sharp}(X_N,\rho^\kappa(t))
       +B\varepsilon_N^\kappa(\eta_N)\right]\dd\rho_N(t)\right\}\\
 &\qquad
 +\int_0^T\!\int \cD_N^\kappa(X_N,\rho^\kappa(t))\,\dd\rho_N(t)\,\dd t
 \le C_Ta_{N,d},
\end{aligned}
\end{equation*}
and, for every $1\le k\le N$, we have
\begin{equation*}
 \sup_{t\le T}
 W_2^2\bigl(\rho_{N:k}(t),\rho^\kappa(t)^{\otimes k}\bigr)
 \le C_Tk a_{N,d}.
\end{equation*}
\end{corollary}

\begin{proof}
The uniform bounds in Proposition~\ref{prop:yukawa_classical_solutions}
verify the hypotheses of Theorem~\ref{thm:yukawa_meanfield}.  It remains only
to check that the product distribution is well prepared.  The transport term in
\eqref{eq:intro_initial_error} vanishes for
$\rho_N^0=\rho_0^{\otimes N}$.  The finite-$N$ off-diagonal modulated energy is not itself nonnegative:
removing the labeled diagonal self-interactions leaves a negative
$O(N^{-1})$ expectation for independent samples.  This is why
Theorem~\ref{thm:yukawa_commutator} is formulated for the corrected quantity
$\cF_N^{\kappa,\sharp}+Ca_{N,d}$.  Independence, Fubini's theorem, and the
off-diagonal shift identity \eqref{eq:additive_shift_off_diagonal_ledger}
give
\[
 \mathbb E_{\rho_0^{\otimes N}}
 \cF_N^{\kappa,\sharp}(X_N,\rho_0)
 =-\frac1N\iint g_{\kappa,d}^{\sharp}(x-y)
 \rho_0(x)\rho_0(y)\,\dd x\dd y.
\]
For $d\ge3$ the double integral is uniformly bounded by
\eqref{eq:static_background_bound_dge3}.  In dimension two,
$\rho_0\in\mathcal P_2$ has a finite logarithmic moment, so
\eqref{eq:two_dimensional_normalized_background_energy_bound} gives the same
uniform bound for the normalized kernel.  Together with
$\varepsilon_N^\kappa(\eta_N)\le Ca_{N,d}$ and $N^{-1}\le a_{N,d}$, this yields
\[
 \mathfrak E_N^{\kappa,0}(\eta_N)\le C_Ta_{N,d}
\]
uniformly for $0<\kappa\le\kappa_*$.  Theorem~\ref{thm:yukawa_meanfield}
now gives the uniform-in-time bound on the expected modulated energy, the time
integral of $\cD_N$, and the marginal estimates.
\end{proof}

\section{The Coulomb limit}
\label{sec:coulomb_limit}

We first examine the decomposition of the Yukawa kernel into the Coulomb
kernel and a remainder and identify the obstructions to controlling the
Yukawa modulated energy by existing Coulomb estimates.  We then establish a
quantitative Coulomb limit using the uniform Yukawa estimates and the kernel
bounds from Lemma~\ref{lem:yukawa_coulomb_kernel_estimates}.  Throughout,
$\alpha_d(\kappa)$ is defined by \eqref{eq:alpha_screening_parameter}.

\subsection{Coulomb perturbation arguments and their limitations}
\label{sec:coulomb_perturbation_comparison}

Recall the notation $q_{\kappa,d}$ from
Lemma~\ref{lem:yukawa_coulomb_kernel_estimates}, and set
$F_{\kappa,d}:=\nabla q_{\kappa,d}$.  We also write
\[
 u_\rho^\kappa:=-\nabla g_{\kappa,d}*\rho,
 \qquad
 u_\rho^{\mathrm C}:=-\nabla g_{0,d}*\rho.
\]
For $d\ge3$ no additive normalization is needed; in dimension two the chosen
normalization changes neither $F_{\kappa,2}$ nor the velocity field.

\subsubsection{Comparison with the sharp Coulomb commutator estimate}

We compare our result with the global first-order estimate of
\cite[Theorem~1.1, estimate~(1.7)]{RosenzweigSerfaty2024}.  In the notation of
that paper, the microscopic scale is
\[
 \lambda_{N,\rho}:=(N\|\rho\|_{L^\infty})^{-1/d},
\]
and the theorem is stated under the condition $\lambda_{N,\rho}<1$.  Their modulated energy contains a factor $1/2$ in front of the off-diagonal
quadratic form, and their Riesz and logarithmic kernels use a different fixed
normalization from $g_{0,d}$ in \eqref{eq:coulomb_kernel_general_intro}.
Multiplying their inequality by the appropriate dimensional constant converts
the left-hand side to the Coulomb commutator used here.  Their modulated
energy is one half of our quadratic form; in $d=2$, the difference in additive
normalization contributes only an $O(N^{-1})$ shift.

At the Coulomb exponent $s=d-2$ and for $\|\rho\|_{L^\infty}\le\Lambda$, the
additive term in \cite[Theorem~1.1]{RosenzweigSerfaty2024} is bounded by
\[
 C_{d,\Lambda}N^{-2/d}\qquad(d\ge3),
\]
whenever $\lambda_{N,\rho}<1$.  In dimension two, their logarithmic shift and
remaining error are bounded by
\[
 C_\Lambda\frac{1+\log N}{N}
\]
under the same condition; the required integrability of the logarithmic energy
follows, for example, from bounded density and a finite logarithmic moment.
Together with the lower bound in
\cite[Proposition~2.1, estimate~(2.29)]{RosenzweigSerfaty2024}, this gives the
Coulomb analogue of the estimate used below whenever
$\lambda_{N,\rho}<1$, with correction $a_{N,d}$ after adjusting constants.

This identifies the finite-$N$ scale above with the sharp Coulomb theory.  For
the arguments below, however, we use the Coulomb bound valid for all $N$
obtained in Subsection~\ref{subsec:static_coulomb_endpoint} by passing
$\kappa\downarrow0$ in Theorem~\ref{thm:yukawa_commutator}.  This formulation
also avoids translating between different kernel and energy normalizations.
Evaluating a Coulomb commutator at $u=u^\kappa$ also requires a common
$W^{1,\infty}$ bound on the Yukawa velocity.  This is assumed in
Theorem~\ref{thm:yukawa_meanfield} and, for the smooth data considered below,
follows uniformly in $\kappa$ from Propositions~\ref{prop:yukawa_velocity_field}
and~\ref{prop:yukawa_classical_solutions}.  We next compare the force
correction with the assumptions on additional regular interactions in
\cite{Serfaty2020}.

\subsubsection{Regularity of the Yukawa--Coulomb force difference}

For fixed $\kappa>0$ in dimension two,
\eqref{eq:force_difference_general_pointwise} and
\eqref{eq:two_dimensional_screened_correction_derivatives} show that
$F_{\kappa,2}=\nabla q_{\kappa,2}$ is bounded, belongs to $\dot H^1$, and is
H\"older continuous for every exponent $0<\alpha<1$.  Thus, at fixed
$\kappa$, the two-dimensional Yukawa system may be viewed as a logarithmic
Coulomb system with an additional regular interaction.  The framework of
\cite[Theorem~1 and Lemma~2.3]{Serfaty2020} then applies to data well prepared
in the Coulomb modulated energy.  Our uniform Yukawa estimate is proved
independently of this reduction.

For $d\ge3$, this reduction is not directly available because the
Yukawa--Coulomb force correction is not continuous at the origin.  In $d=3$, the explicit formula gives
\begin{equation}\label{eq:d3_screened_correction_local_expansion}
 F_{\kappa,3}(x)
 =\frac{\kappa^2}{8\pi}\frac{x}{|x|}+O_\kappa(|x|)
 \qquad (|x|\downarrow0),
\end{equation}
so no continuous extension exists.  For $d\ge4$, the standard small-argument
Bessel expansions give
\begin{equation}\label{eq:higher_dimensional_screened_correction_force_expansion}
 F_{\kappa,d}(x)=
 \begin{cases}
 \dfrac{\kappa^2}{8\pi^2}\dfrac{x}{|x|^2}
 +O_\kappa\!\left(|x|(1+|\log|x||)\right),&d=4,\\[2mm]
 \dfrac{\kappa^2}{2(d-2)|\mathbb S^{d-1}|}
 \dfrac{x}{|x|^{d-2}}+o_\kappa(|x|^{3-d}),&d\ge5,
 \end{cases}
\end{equation}
which is unbounded at the origin.  The result of \cite{Serfaty2020} for an
additional regular interaction requires, in particular, H\"older continuity of
the added force and therefore does not directly apply in dimensions
$d\ge3$.

\subsubsection{Low-frequency comparison of the Coulomb and Yukawa energies}

For a smooth compactly supported neutral density $\nu$,
\[
 I_0[\nu]=(2\pi)^{-d}\int_{\R^d}
 \frac{|\widehat\nu(\xi)|^2}{|\xi|^2}\,\dd\xi,
 \qquad
 I_\kappa[\nu]=(2\pi)^{-d}\int_{\R^d}
 \frac{|\widehat\nu(\xi)|^2}{|\xi|^2+\kappa^2}\,\dd\xi.
\]
The comparison already fails for every fixed $\kappa>0$: there is no finite
constant
$C_\kappa$ such that
\[
 I_0[\nu]\le C_\kappa I_\kappa[\nu]
\]
for all nonzero smooth compactly supported neutral densities $\nu$.  Indeed,
fix one such $\nu$ and set $\nu_R(x)=R^{-d}\nu(x/R)$.  Then
\begin{equation}\label{eq:dilation_coulomb_yukawa_energy}
 I_0[\nu_R]=R^{2-d}I_0[\nu],
 \qquad
 I_\kappa[\nu_R]
 =R^{2-d}(2\pi)^{-d}\int_{\R^d}
 \frac{|\widehat\nu(\zeta)|^2}
 {|\zeta|^2+(\kappa R)^2}\,\dd\zeta.
\end{equation}
Neutrality makes $I_0[\nu]$ finite.  Moreover, dominated convergence and
Plancherel give
\[
 (\kappa R)^2R^{d-2}I_\kappa[\nu_R]
 \longrightarrow \|\nu\|_{L^2}^2>0,
\]
and therefore
\[
 \frac{I_0[\nu_R]}{I_\kappa[\nu_R]}
 \sim \frac{I_0[\nu]}{\|\nu\|_{L^2}^2}(\kappa R)^2
 \longrightarrow\infty.
\]
Thus the Coulomb energy cannot be bounded globally by the Yukawa energy even
for fixed $\kappa>0$; in particular, no such bound can hold uniformly as
$\kappa\downarrow0$.

The same scaling also exhibits the cancellation between the Coulomb term and
the remainder.  If $d\ge3$, $u(x)=x$, and
\[
 \mathscr H^g(\nu;u)
 :=\iint (u(x)-u(y))\cdot\nabla g(x-y)\nu(x)\nu(y)\,\dd x\dd y,
\]
then homogeneity gives
$\mathscr H^{g_{0,d}}(\nu;u)=(2-d)I_0[\nu]$.  For smooth compactly supported
$\nu$, the stress bound \eqref{eq:finite_energy_stress_bound} gives
$|\mathscr H^{g_{\kappa,d}}(\nu;u)|\le C_dI_\kappa[\nu]$.  Consequently,
\[
 \mathscr H^{q_{\kappa,d}}(\nu_R;u)
 =-(2-d)I_0[\nu_R]+O(I_\kappa[\nu_R]).
\]
Estimating the two commutator components separately loses this cancellation
and introduces the Coulomb energy.

The preceding regularity obstruction and the low-frequency incomparability of
the two quadratic energies are already sufficient to show why estimating the
Coulomb and remainder pieces separately does not close in the natural Yukawa
energy.  The direct Yukawa commutator avoids precisely this loss.

Once the Yukawa estimate is available uniformly in $\kappa$, the Coulomb
limit can be taken directly.

\subsection{Limit of the static commutator estimate}\label{subsec:static_coulomb_endpoint}

Throughout this subsection, the superscript $\mathrm C$ denotes the unscreened
Coulomb quantity (logarithmic Coulomb when $d=2$), while the superscript $0$
is reserved for initial data.  We use the dimension-dependent kernel
$g_{0,d}$ fixed in \eqref{eq:coulomb_kernel_general_intro}.  When $d\ge3$ we
write
\[
 E_N^{\mathrm C}(X_N):=\frac1{2N^2}\sum_{i\ne j}g_{0,d}(x_i-x_j),
\]
and use $V_\rho^{\mathrm C}$ and $I_0[\rho]$ with the convention in
\eqref{eq:global_unscreened_background_convention}.

For fixed $N$, a collision-free configuration $X_N$, a bounded probability
density $\rho$, and a globally Lipschitz field $u$, the Yukawa functionals
converge to their Coulomb counterparts as $\kappa\downarrow0$:
\[
 \cF_N^{\kappa,\sharp}(X_N,\rho)\longrightarrow
 \cF_N^{g_{0,d}}(X_N,\rho),
 \qquad
 \cH_N^{g_{\kappa,d}}(X_N,\rho;u)\longrightarrow
 \cH_N^{g_{0,d}}(X_N,\rho;u).
\]
In dimension two we additionally assume
$\int_{\R^2}\log(2+|x|)\rho(x)\,\dd x<\infty$ and use the normalized
representative in \eqref{eq:renormalized_2d_limit}.  All statements in this
subsection involving the logarithmic Coulomb energy are made under this
finite-logarithmic-moment assumption.  For the energy,
\eqref{eq:renormalized_2d_limit} and the global majorant
\eqref{eq:renormalized_2d_global_majorant} give dominated convergence in the
particle--background and background--background terms.  The finite
particle--particle sum converges pointwise away from the diagonal.  Hence
$\cF_N^{\kappa,\sharp}(X_N,\rho)\to\cF_N^{g_{0,2}}(X_N,\rho)$.  For the
commutator, the additive normalization is irrelevant because gradients remove
constants.  When $d\ge3$ one has
\[
 0<g_{\kappa,d}(z)\le g_{0,d}(z)=\frac{|z|^{2-d}}{(d-2)\Sd},
 \qquad
 |\nabla g_{\kappa,d}(z)|\le C_d|z|^{1-d},
\]
while in $d=2$,
\[
 |\nabla g_{\kappa,2}(z)|\le(2\pi|z|)^{-1}.
\]
After using
$|u(x)-u(y)|\le\|\nabla u\|_{L^\infty}|x-y|$, the Lipschitz cancellation
leaves commutator kernels with standard integrable majorants in both the near
and far fields.
Dominated convergence therefore gives the commutator limit.  This separates
the only place where the logarithmic additive normalization matters (the
energy) from the convergence of the force, where it does not.

Consequently, Theorem~\ref{thm:yukawa_commutator} also yields a first-order
Coulomb estimate at $\kappa=0$ with the sharp finite-$N$ scaling.  To see
this, fix $N$, $X_N$, $\rho$, $u$, and an admissible truncation radius
$\eta$, apply the theorem with $\kappa_*=1$, and then let
$\kappa\downarrow0$.  The resulting Coulomb inequalities have correction
$N^{-1}\ell_d(\eta)+\eta^2$.  Optimizing $\eta$ as in
Corollary~\ref{cor:yukawa_commutator_optimized_scale} yields constants $C,B$
depending only on $d$ and an upper bound for $\|\rho\|_{L^\infty}$ such that
\begin{equation}\label{eq:optimized_coulomb_energy_shift}
 \cF_N^{g_{0,d}}(X_N,\rho)+B a_{N,d}\ge0,
\end{equation}
and
\begin{equation}\label{eq:optimized_coulomb_commutator}
 |\cH_N^{g_{0,d}}(X_N,\rho;u)|
 \le C\|\nabla u\|_{L^\infty}
 \bigl(\cF_N^{g_{0,d}}(X_N,\rho)+B a_{N,d}\bigr).
\end{equation}
Thus, with our normalization, the Yukawa estimate yields a first-order
Coulomb bound with the same sharp $N$-scaling as in
\cite{RosenzweigSerfaty2024}.  We use
\eqref{eq:optimized_coulomb_energy_shift}--
\eqref{eq:optimized_coulomb_commutator} in the direct three-dimensional
comparison below.  Here $N$ is fixed while $\kappa\downarrow0$, so this
endpoint argument does not interchange the limits $\kappa\downarrow0$ and
$N\to\infty$.

\subsection{Quantitative Coulomb limit}

We use two complementary comparisons.  In dimension three, the
Yukawa--Coulomb force difference is bounded pointwise, allowing a direct
comparison at the particle level.  In every dimension, convolution against a
bounded density controls the force difference and yields a continuum
comparison.

\subsubsection{Comparison of the particle dynamics in dimension three}

The proof applies the exact perturbed chain rule for the modulated energy from
Lemma~\ref{lem:exact_perturbed_modulated_energy_chain_rule} with
$Z_i^\kappa$, the difference between the Yukawa and Coulomb particle forces.

\begin{theorem}[Simultaneous mean-field and Coulomb limit at the particle level in $d=3$]
\label{thm:simultaneous_yukawa_coulomb}
Assume $d=3$.  Let $\rho^{\mathrm C}$ be a probability density solving, in the
sense of distributions,
\[
 \partial_t\rho^{\mathrm C}+\nabla\cdot(\rho^{\mathrm C}u^{\mathrm C})=0,
 \qquad u^{\mathrm C}=-\nabla g_{0,3}*\rho^{\mathrm C},
\]
on $[0,T]$ and assume
\[
 \rho^{\mathrm C}\in C([0,T];(\mathcal P_2,W_2)),
 \qquad
 \Lambda^{\mathrm C}:=1+\sup_{t\le T}\left[
   \|\rho^{\mathrm C}(t)\|_{L^\infty}
   +\|u^{\mathrm C}(t)\|_{W^{1,\infty}}
 \right]<\infty.
\]
Choose $B\ge1$, depending only on $\Lambda^{\mathrm C}$, so that the
endpoint lower bound \eqref{eq:optimized_coulomb_energy_shift} with $d=3$
holds uniformly for $t\in[0,T]$.  For $N\ge2$ and a symmetric
$\rho_N^0\in\mathcal P_2((\R^3)^N)$, define
\[
 \mathfrak E_N^{\mathrm C,0}
 :=\frac1N W_2^2(\rho_N^0,\rho^{\mathrm C}(0)^{\otimes N})
 +\int\left(\cF_N^{g_{0,3}}(X_N,\rho^{\mathrm C}(0))
        +BN^{-2/3}\right)\,\dd \rho_N^0(X_N).
\]
There exists $C_T=C(T,\Lambda^{\mathrm C})<\infty$ such that, whenever
$\mathfrak E_N^{\mathrm C,0}<\infty$, the Yukawa $N$-particle law
$\rho_N^\kappa(t)$ is well defined for every $\kappa>0$ and, for
$1\le k\le N$,
\begin{align}
 \sup_{0\le t\le T}\frac1N W_2^2
 \bigl(\rho_N^\kappa(t),\rho^{\mathrm C}(t)^{\otimes N}\bigr)
 &\le C_T\left(\mathfrak E_N^{\mathrm C,0}+\kappa^4\right),
 \label{eq:simultaneous_yukawa_coulomb_full}\\
 \sup_{0\le t\le T}W_2^2
 \bigl(\rho_{N:k}^\kappa(t),\rho^{\mathrm C}(t)^{\otimes k}\bigr)
 &\le C_Tk\left(\mathfrak E_N^{\mathrm C,0}+\kappa^4\right).
 \notag
\end{align}
In particular, if $N_m\to\infty$, $\kappa_m\downarrow0$, and
$\mathfrak E_{N_m}^{\mathrm C,0}\to0$, then the normalized Wasserstein error
of the $N$-particle law tends to zero.  The same holds for every fixed
$k$-particle marginal, with no relation required between the two rates.
\end{theorem}

\begin{proof}
Apply the endpoint Coulomb estimates
\eqref{eq:optimized_coulomb_energy_shift}--\eqref{eq:optimized_coulomb_commutator}
with $d=3$, $\rho=\rho^{\mathrm C}(t)$, and $u=u^{\mathrm C}(t)$.  The common
$L^\infty$ and $W^{1,\infty}$ bounds give a constant
$C=C(\Lambda^{\mathrm C})$ such that, for every $N\ge2$, $t\in[0,T]$, and
collision-free $X_N$,
\begin{equation}
\label{eq:coulomb_limit_time_uniform_bounds}
 \cF_N^{g_{0,3}}(X_N,\rho^{\mathrm C}(t))+BN^{-2/3}\ge0,
\end{equation}
and
\begin{equation}
\label{eq:coulomb_limit_time_uniform_commutator}
 |\cH_N^{g_{0,3}}(X_N,\rho^{\mathrm C}(t);u^{\mathrm C}(t))|
 \le C\left(
 \cF_N^{g_{0,3}}(X_N,\rho^{\mathrm C}(t))
 +BN^{-2/3}\right).
\end{equation}
Under the stated $L^\infty$ bound, the Coulomb background potential is
bounded and continuous.  After setting $g_{0,3}(0)=+\infty$, the positive
particle--particle term is lower semicontinuous.  Hence the corrected modulated
energy in $\mathfrak E_N^{\mathrm C,0}$ is a nonnegative
extended-real-valued Borel function, and finiteness of
$\mathfrak E_N^{\mathrm C,0}$ implies $\rho_N^0(\Delta_N)=0$.  Moreover,
\[
 \cF_N^{g_{0,3}}(X_N,\rho^{\mathrm C}(0))
 =2E_N^{\mathrm C}(X_N)-\frac2N\sum_i
 V_{\rho^{\mathrm C}(0)}^{\mathrm C}(x_i)+I_0[\rho^{\mathrm C}(0)],
\]
so the expected Coulomb microscopic energy is finite.  Since
$0<g_{\kappa,3}\le g_{0,3}$ and $\rho_N^0\in\mathcal P_2$,
Proposition~\ref{prop:yukawa_particle_flow_moment} defines the global
collision-free Yukawa $N$-particle law for every $\kappa>0$.

Along the Yukawa particle dynamics set
\[
 K_i^\kappa:=\frac1N\sum_{j\ne i}\nabla g_{\kappa,3}(x_i-x_j),
 \qquad
 K_i^{\mathrm C}:=\frac1N\sum_{j\ne i}\nabla g_{0,3}(x_i-x_j),
\]
\[
 Z_i^\kappa:=K_i^\kappa-K_i^{\mathrm C},
 \qquad
 R_i^{\mathrm C}:=K_i^{\mathrm C}-\nabla g_{0,3}*\rho^{\mathrm C}(x_i),
 \qquad
 \cD_N^{\mathrm C}:=\frac1N\sum_i|R_i^{\mathrm C}|^2.
\]
Let
\[
 \dot y_i=u^{\mathrm C}(t,y_i),
 \qquad
 \cQ_N^{\mathrm C}:=\frac1N\sum_i|x_i-y_i|^2,
 \qquad
 \cS_N^{\mathrm C}:=\cF_N^{g_{0,3}}(X_N,\rho^{\mathrm C})
 +BN^{-2/3}.
\]
The Yukawa trajectory satisfies
\[
 \dot x_i=-K_i^\kappa=-K_i^{\mathrm C}-Z_i^\kappa,
\]
which is precisely the perturbation convention $\dot x_i=-K_i-Z_i$ in
\eqref{eq:general_perturbed_modulated_energy_identity} for the Coulomb kernel.
Hence
\begin{equation*}
 \frac{\dd}{\dd t}\cF_N^{g_{0,3}}(X_N,\rho^{\mathrm C})
 =\cH_N^{g_{0,3}}(X_N,\rho^{\mathrm C};u^{\mathrm C})
 -2\cD_N^{\mathrm C}
 -\frac2N\sum_i R_i^{\mathrm C}\cdot Z_i^\kappa.
\end{equation*}
Together with \eqref{eq:coulomb_limit_time_uniform_commutator}, this yields
\[
 \frac{\dd}{\dd t}\cS_N^{\mathrm C}
 \le C\cS_N^{\mathrm C}-2\cD_N^{\mathrm C}
      -\frac2N\sum_i R_i^{\mathrm C}\cdot Z_i^\kappa.
\]
By Young's inequality and \eqref{eq:particle_force_difference_square},
\[
 -\frac2N\sum_i R_i^{\mathrm C}\cdot Z_i^\kappa
 \le \frac1N\sum_i|R_i^{\mathrm C}|^2
      +\frac1N\sum_i|Z_i^\kappa|^2
 =\cD_N^{\mathrm C}+\frac1N\sum_i|Z_i^\kappa|^2
 \le \cD_N^{\mathrm C}+\frac{\kappa^4}{64\pi^2}.
\]
Combining this estimate with the preceding differential inequality gives
\begin{equation*}
 \frac{\dd}{\dd t}\cS_N^{\mathrm C}
 \le C_T\cS_N^{\mathrm C}-\cD_N^{\mathrm C}+C\kappa^4.
\end{equation*}
Since $R_i^{\mathrm C}=K_i^{\mathrm C}+u^{\mathrm C}(x_i)$, the Yukawa
trajectory also satisfies
\[
 \dot x_i-u^{\mathrm C}(y_i)
 =-R_i^{\mathrm C}+u^{\mathrm C}(x_i)-u^{\mathrm C}(y_i)-Z_i^\kappa.
\]
Writing $z_i=x_i-y_i$, we therefore have
\begin{align*}
 \frac{\dd}{\dd t}\cQ_N^{\mathrm C}
 &=\frac2N\sum_i z_i\cdot
 \bigl[-R_i^{\mathrm C}+u^{\mathrm C}(x_i)-u^{\mathrm C}(y_i)
       -Z_i^\kappa\bigr]\\
 &\le \frac1N\sum_i|R_i^{\mathrm C}|^2
   +\frac1N\sum_i|Z_i^\kappa|^2
   +(2+2\|\nabla u^{\mathrm C}\|_{L^\infty})\cQ_N^{\mathrm C}\\
 &\le \cD_N^{\mathrm C}
   +(2+2\|\nabla u^{\mathrm C}\|_{L^\infty})\cQ_N^{\mathrm C}
   +C\kappa^4.
\end{align*}
Here the two copies of $\cQ_N^{\mathrm C}$ come from applying
$2|z_i||R_i^{\mathrm C}|\le|z_i|^2+|R_i^{\mathrm C}|^2$ and
$2|z_i||Z_i^\kappa|\le|z_i|^2+|Z_i^\kappa|^2$ separately.
All coefficients in the two differential inequalities are controlled by
$\Lambda^{\mathrm C}$ and numerical constants.  The dependence on the
screening parameter is confined to the explicit term $\kappa^4$.  Therefore
\[
 \frac{\dd}{\dd t}(\cQ_N^{\mathrm C}+\cS_N^{\mathrm C})
 \le C_T(\cQ_N^{\mathrm C}+\cS_N^{\mathrm C})+C\kappa^4.
\]
Here $\cQ_N^{\mathrm C}\ge0$ by definition and
$\cS_N^{\mathrm C}\ge0$ by \eqref{eq:coulomb_limit_time_uniform_bounds}.
Gronwall's inequality therefore controls their sum.  Dropping the nonnegative
term $\cS_N^{\mathrm C}(t)$ gives
\begin{equation}
\label{eq:perturbed_coulomb_pathwise_result}
 \cQ_N^{\mathrm C}(t)
 \le C_T\left(\cQ_N^{\mathrm C}(0)+\cS_N^{\mathrm C}(0)+\kappa^4\right).
\end{equation}

Let $\Psi_t^{\mathrm C}$ be the characteristic flow of $u^{\mathrm C}$ and
$\Phi_t^{N,\kappa}$ the Yukawa particle flow.  By
the characteristic representation \eqref{eq:limiting_solution_characteristic_representation},
\[
 \rho^{\mathrm C}(t)=(\Psi_t^{\mathrm C})_\#\rho^{\mathrm C}(0).
\]
We use the Borel extension of the Yukawa flow from
Subsection~\ref{subsec:particle_distribution_preliminaries}.  For any
$\pi_0\in\Pi(\rho_N^0,\rho^{\mathrm C}(0)^{\otimes N})$, set
$
 \pi_t:=(\Phi_t^{N,\kappa},(\Psi_t^{\mathrm C})^{\otimes N})_\#\pi_0.
$
Then
\[
 \pi_t\in\Pi(\rho_N^\kappa(t),\rho^{\mathrm C}(t)^{\otimes N}).
\]
For $\pi_0$-almost every initial configuration, the pathwise estimate
\eqref{eq:perturbed_coulomb_pathwise_result} holds for all $t\in[0,T]$.
Its terms are Borel functions of the initial coordinates, so the estimate may
be integrated against $\pi_0$.  Moreover,
\[
 \int\cQ_N^{\mathrm C}(t)\,\dd\pi_0
 =\int\frac1N\sum_i|x_i-y_i|^2\,\dd\pi_t
 \ge \frac1N W_2^2\bigl(\rho_N^\kappa(t),
          \rho^{\mathrm C}(t)^{\otimes N}\bigr).
\]
Since $\cS_N^{\mathrm C}(0)$ depends only on the particle coordinate, choosing
an optimal initial coupling gives
\[
 \int\bigl(\cQ_N^{\mathrm C}(0)+\cS_N^{\mathrm C}(0)\bigr)\,\dd\pi_0
 =\mathfrak E_N^{\mathrm C,0}.
\]
This proves \eqref{eq:simultaneous_yukawa_coulomb_full}.  The marginal
estimate follows from symmetry of the Yukawa $N$-particle law and
\eqref{eq:exact_symmetric_marginal_projection}.
\end{proof}

\subsubsection{Comparison of the mean-field equations in all dimensions}

For $d\ge4$, the singular pointwise force difference rules out the preceding
direct perturbative comparison at the particle level.  After convolution with bounded densities the
singularity is locally integrable, so we compare the continuum flows first and
then use tensorization with the Yukawa mean-field theorem.  We formulate this
continuum comparison for every $d\ge2$: in $d=3$ it complements the stronger
direct particle theorem, while in $d=2$ it yields the logarithmic Coulomb
limit.

\smallskip\noindent\emph{Neutral Coulomb energy and the $\dot H^{-1}$ chain rule.}
The static identities needed below are standard.  If
$\nu\in L^1(\R^d)\cap L^\infty(\R^d)$ is neutral and, in $d=2$,
$\int |x|\,|\nu(x)|\,\dd x<\infty$, then, with
$\phi=g_{0,d}*\nu$,
\begin{equation}
\label{eq:neutral_coulomb_energy_identity}
 I_0[\nu]
 =(2\pi)^{-d}\int_{\R^d}\frac{|\widehat\nu(\xi)|^2}{|\xi|^2}\,\dd\xi
 =\int_{\R^d}|\nabla\phi|^2\,\dd x\ge0.
\end{equation}
For $u\in W^{1,\infty}(\R^d;\R^d)$, the cancelled Coulomb commutator
\begin{equation}
\label{eq:neutral_coulomb_commutator_definition}
 \mathscr H_{0,d}(\nu;u)
 :=\iint
 (u(x)-u(y))\cdot\nabla g_{0,d}(x-y)\nu(x)\nu(y)\,\dd x\dd y
\end{equation}
is absolutely convergent and satisfies the stress-energy identity for finite Coulomb energy
\begin{equation}
\label{eq:neutral_coulomb_commutator_stress}
 \mathscr H_{0,d}(\nu;u)
 =\int_{\R^d}\nabla u:
 \bigl(2\nabla\phi\otimes\nabla\phi-|\nabla\phi|^2\mathrm{Id}\bigr)\,\dd x,
\end{equation}
and hence
\begin{equation}
\label{eq:neutral_coulomb_commutator_bound}
 |\mathscr H_{0,d}(\nu;u)|
 \le C_d\|\nabla u\|_{L^\infty}I_0[\nu].
\end{equation}
These are the standard Coulomb electric-field and stress-energy identities;
see, for example, \cite[Lemma~4.3]{Serfaty2020} and the weak--strong
Coulomb argument of \cite{Duerinckx2016}.  Under the stated assumptions the
extension from smooth neutral densities follows by the usual cutoff and
mollification approximation.  In dimension two, neutrality together with the
first moment gives
$|\widehat\nu(\xi)|\le |\xi|\int |x||\nu(x)|\,\dd x$ near the origin, so
\eqref{eq:neutral_coulomb_energy_identity} has no low-frequency divergence.

The only time-dependent fact that we shall use is the following elementary
Hilbert-space chain rule.  We record it because the flux arising in the
Yukawa--Coulomb comparison is not simply of the form $\nu u$.

\begin{lemma}[$\dot H^{-1}$ chain rule for a continuity equation]
\label{lem:Hminusone_continuity_chain_rule}
Let $\nu\in L^\infty(0,T;\dot H^{-1}(\R^d))$ and
$J\in L^2((0,T)\times\R^d;\R^d)$ satisfy
\begin{equation}
\label{eq:Hminusone_chain_rule_continuity_equation}
 \partial_t\nu+\nabla\cdot J=0
 \qquad\text{in }\mathscr D'((0,T)\times\R^d).
\end{equation}
Then $\nu$ has an absolutely continuous representative with values in
$\dot H^{-1}$, in fact
$\nu\in W^{1,2}(0,T;\dot H^{-1})$, and for almost every $t$,
\begin{equation}
\label{eq:neutral_coulomb_chain_rule}
 \frac{\dd}{\dd t}\|\nu(t)\|_{\dot H^{-1}}^2
 =2\int_{\R^d}\nabla(-\Delta)^{-1}\nu(t,x)\cdot J(t,x)\,\dd x.
\end{equation}
\end{lemma}

\begin{proof}
By Fourier duality,
\[
 \|\nabla\cdot J(t)\|_{\dot H^{-1}}\le \|J(t)\|_{L^2}
 \qquad\text{for a.e. }t.
\]
Hence \eqref{eq:Hminusone_chain_rule_continuity_equation} gives
$\partial_t\nu\in L^2(0,T;\dot H^{-1})$, and therefore
$\nu\in W^{1,2}(0,T;\dot H^{-1})$.  The Hilbert-space chain rule yields
\[
 \frac{\dd}{\dd t}\|\nu\|_{\dot H^{-1}}^2
 =2(\nu,\partial_t\nu)_{\dot H^{-1}}.
\]
Writing $\phi=(-\Delta)^{-1}\nu$ and using
$\partial_t\nu=-\nabla\cdot J$ gives
\[
 (\nu,\partial_t\nu)_{\dot H^{-1}}
 =\langle-\nabla\cdot J,\phi\rangle
 =\int_{\R^d}J\cdot\nabla\phi,
\]
which is finite because $J,\nabla\phi\in L^2$.
\end{proof}

\begin{proposition}[Continuum Yukawa-to-Coulomb stability]
\label{prop:continuum_screening_stability}
Fix $d\ge2$, $T>0$, and $\Lambda_T\ge1$.  There exists
$C_T=C(d,T,\Lambda_T)<\infty$ such that the following holds for every
$\kappa>0$.  Let $\rho^\kappa$ and $\rho^{\mathrm C}$ solve on $[0,T]$,
with the same initial probability density $\rho_0\in\mathcal P_2(\R^d)$,
\[
 \partial_t\rho^\kappa+\nabla\cdot(\rho^\kappa u^\kappa)=0,
 \qquad u^\kappa=-\nabla g_{\kappa,d}*\rho^\kappa,
\]
\[
 \partial_t\rho^{\mathrm C}+\nabla\cdot(\rho^{\mathrm C} u^{\mathrm C})=0,
 \qquad u^{\mathrm C}=-\nabla g_{0,d}*\rho^{\mathrm C},
\]
where $g_{0,d}$ is logarithmic in $d=2$ and Newtonian in $d\ge3$.  Assume
\[
 \rho^\kappa,\rho^{\mathrm C}\in C([0,T];(\mathcal P_2,W_2)),
\]
\[
 \sup_{t\le T}\left(
 \|\rho^\kappa(t)\|_\infty+\|\rho^{\mathrm C}(t)\|_\infty
 +\|u^\kappa(t)\|_{W^{1,\infty}}+\|u^{\mathrm C}(t)\|_{W^{1,\infty}}
 \right)\le\Lambda_T.
\]
Set $\nu:=\rho^\kappa-\rho^{\mathrm C}$ and
$w:=-\nabla g_{0,d}*\nu$.  Then
\begin{equation}
\label{eq:continuum_screening_stability_strong}
 \sup_{t\le T}\left[
 W_2^2(\rho^\kappa(t),\rho^{\mathrm C}(t))+I_0[\nu(t)]\right]
 +\int_0^T\!\int_{\R^d}\rho^\kappa|w|^2\,\dd x\dd t
 \le C_T\alpha_d(\kappa)^2.
\end{equation}
In particular,
\[
 \sup_{t\le T}W_2^2(\rho^\kappa(t),\rho^{\mathrm C}(t))
 \le C_T\alpha_d(\kappa)^2.
\]
Thus the squared transport and Coulomb energy errors are $O(\kappa^2)$ in
$d=2$ and $O(\kappa^4)$ in $d\ge3$, and the same order controls the
integrated Coulomb field discrepancy generated by $\nu$ in
\eqref{eq:continuum_screening_stability_strong}.  The constant is uniform for
any family of solution pairs satisfying the same bound $\Lambda_T$.  For the
classical solutions constructed in Proposition~\ref{prop:yukawa_classical_solutions},
the required uniform bounds hold on every compact subinterval of the common
existence interval.
\end{proposition}

\begin{proof}
Set
\[
 \nu:=\rho^\kappa-\rho^{\mathrm C},
 \qquad \phi:=g_{0,d}*\nu,
 \qquad w:=-\nabla\phi,
\]
and
\[
 e_\kappa:=-(\nabla g_{\kappa,d}-\nabla g_{0,d})*\rho^\kappa.
\]
For each $x$,
\begin{equation}\label{eq:continuum_screening_velocity_decomposition}
 u^\kappa=u^{\mathrm C}+w+e_\kappa.
\end{equation}
Because the constant in
\eqref{eq:kernel_difference_convolution_bound} is independent of $\kappa_*$,
the estimate applies to every fixed $\kappa>0$.  Since $\rho^\kappa$ has unit mass and is
bounded by $\Lambda_T$, it gives
\begin{equation}\label{eq:continuum_screening_error_bound}
 \|e_\kappa(t)\|_{L^\infty}\le C_T\alpha_d(\kappa)
 \qquad(0\le t\le T).
\end{equation}

\smallskip
\noindent\emph{Coulomb energy and exact chain rule.}
The signed density $\nu$ is neutral and uniformly bounded in
$L^1\cap L^\infty$, with $\|\nu(t)\|_{L^1}\le2$ and
$\|\nu(t)\|_{L^\infty}\le2\Lambda_T$.  In dimension two, continuity in
$W_2$ on the compact time interval gives uniform second moments and hence
\begin{equation}
\label{eq:continuum_screening_first_moment_bound}
 \sup_{t\le T}\int_{\R^2}|x|\,|\nu(t,x)|\,\dd x
 \le \sup_{t\le T}\left(
    \int|x|\rho^\kappa(t,x)\,\dd x+
    \int|x|\rho^{\mathrm C}(t,x)\,\dd x\right)<\infty.
\end{equation}
The preceding neutral Coulomb identities give
\[
 \mathscr F(t):=I_0[\nu(t)]
 =\int_{\R^d}|\nabla\phi(t,x)|^2\,\dd x\ge0.
\]
The uniform $L^1\cap L^\infty$ bounds, together with
\eqref{eq:continuum_screening_first_moment_bound} when $d=2$, also give
\begin{equation}
\label{eq:continuum_screening_coulomb_energy_uniform_bound}
 \sup_{t\le T}\mathscr F(t)\le C_T.
\end{equation}
Define
\begin{equation}
\label{eq:continuum_screening_dissipation_definition}
 \mathscr D(t):=\int_{\R^d}\rho^\kappa(t,x)|w(t,x)|^2\,\dd x.
\end{equation}
It is finite and satisfies
$\mathscr D(t)\le\Lambda_T\mathscr F(t)$.

Subtracting the two continuity equations and using
\eqref{eq:continuum_screening_velocity_decomposition} gives
\begin{equation}
\label{eq:continuum_screening_difference_flux}
 \partial_t\nu+\nabla\cdot J=0,
 \qquad
 J:=\nu u^{\mathrm C}+\rho^\kappa(w+e_\kappa).
\end{equation}
To apply Lemma~\ref{lem:Hminusone_continuity_chain_rule}, it remains to check
that $J\in L^2((0,T)\times\R^d)$.  The bound
$\nu\in L^\infty_t(L^1\cap L^\infty)$ implies a uniform $L^2$ estimate,
and $u^{\mathrm C}\in L^\infty_tL^\infty_x$.  Moreover,
\begin{align}
 \|\rho^\kappa w\|_{L^2}^2
 &\le \|\rho^\kappa\|_{L^\infty}\mathscr D
 \le \Lambda_T^2\mathscr F,
 \label{eq:continuum_screening_flux_w_L2}\\
 \|\rho^\kappa e_\kappa\|_{L^2}^2
 &\le \|\rho^\kappa\|_{L^\infty}
       \int\rho^\kappa|e_\kappa|^2
 \le C_T\alpha_d(\kappa)^2,
 \label{eq:continuum_screening_flux_e_L2}
\end{align}
where we used unit mass and
\eqref{eq:continuum_screening_error_bound}.  Therefore
Lemma~\ref{lem:Hminusone_continuity_chain_rule} applies directly, without an auxiliary
spatial or Fourier regularization parameter.

Since $\nabla\phi=-w$, the exact chain rule and
\eqref{eq:continuum_screening_difference_flux} give, for almost every time,
\begin{align}
 \mathscr F'
 &=2\int\nabla\phi\cdot
    \bigl(\nu u^{\mathrm C}+\rho^\kappa(w+e_\kappa)\bigr)\,\dd x\notag\\
 &=\mathscr H_{0,d}(\nu;u^{\mathrm C})
   -2\mathscr D
   -2\int\rho^\kappa w\cdot e_\kappa\,\dd x.
 \label{eq:continuum_screening_energy_identity}
\end{align}
Here
\[
 2\int\nu u^{\mathrm C}\cdot\nabla\phi
 =\mathscr H_{0,d}(\nu;u^{\mathrm C})
\]
follows by symmetrizing the double integral; its absolute convergence follows
from the cancellation in \eqref{eq:neutral_coulomb_commutator_definition}.  The stress estimate
\eqref{eq:neutral_coulomb_commutator_bound} gives
\begin{equation}
\label{eq:continuum_coulomb_commutator_bound}
 |\mathscr H_{0,d}(\nu;u^{\mathrm C})|
 \le C_d\|\nabla u^{\mathrm C}\|_{L^\infty}\mathscr F.
\end{equation}
Finally, weighted Young's inequality and
\eqref{eq:continuum_screening_error_bound} yield
\begin{align*}
 2\left|\int\rho^\kappa w\cdot e_\kappa\right|
 &\le \frac12\mathscr D
      +2\int\rho^\kappa|e_\kappa|^2\\
 &\le \frac12\mathscr D+C_T\alpha_d(\kappa)^2.
\end{align*}
Inserting these bounds into
\eqref{eq:continuum_screening_energy_identity} leaves the quantitative
coercive inequality
\begin{equation}
\label{eq:continuum_screening_energy_differential_inequality}
 \mathscr F'
 \le C_T\mathscr F-\frac32\mathscr D
      +C_T\alpha_d(\kappa)^2.
\end{equation}
Only the positivity of the remaining dissipation is relevant; the particular
coefficient $3/2$ is immaterial.  This remaining term absorbs the corresponding
contribution in the transport estimate.

\smallskip
\noindent\emph{Quadratic transport estimate.}
Let $X_t^\kappa(a)$ and $X_t^{\mathrm C}(a)$ be the characteristic flows
starting from the same point $a$, and transport the diagonal initial coupling
by $(X_t^\kappa,X_t^{\mathrm C})$.  Denote the resulting coupling by $\pi_t$
and set
\[
 \mathscr Q(t):=\int_{\R^d\times\R^d}|x-y|^2\,\dd\pi_t(x,y)
 =\int_{\R^d}|X_t^\kappa(a)-X_t^{\mathrm C}(a)|^2\,\dd\rho_0(a).
\]
The velocity fields are globally Lipschitz and have at most linear growth,
uniformly on $[0,T]$.  Since $\rho_0\in\mathcal P_2$, both flow maps belong
to $AC([0,T];L^2(\rho_0))$.  Consequently $\mathscr Q\in AC([0,T])$, and
Cauchy--Schwarz justifies differentiation under the $\rho_0$-integral.  Using \eqref{eq:continuum_screening_velocity_decomposition},
\[
 u^\kappa(x)-u^{\mathrm C}(y)
 =w(x)+e_\kappa(x)+u^{\mathrm C}(x)-u^{\mathrm C}(y).
\]
The $x$-marginal of $\pi_t$ is $\rho^\kappa(t)$.  Writing
$\delta=x-y$ and differentiating under the coupling gives
\begin{align*}
 \mathscr Q'
 &=2\int \delta\cdot w(x)\,\dd\pi_t
   +2\int \delta\cdot e_\kappa(x)\,\dd\pi_t
   +2\int \delta\cdot\bigl(u^{\mathrm C}(x)-u^{\mathrm C}(y)\bigr)\,\dd\pi_t\\
 &\le \frac12\mathscr D+2\mathscr Q
   +\mathscr Q+\int\rho^\kappa|e_\kappa|^2
   +2\|\nabla u^{\mathrm C}\|_{L^\infty}\mathscr Q\\
 &\le \frac12\mathscr D+C_T\mathscr Q
      +C_T\alpha_d(\kappa)^2.
\end{align*}
Here we used
$2|\delta||w|\le\frac12|w|^2+2|\delta|^2$ for the first term and
$2|\delta||e_\kappa|\le|\delta|^2+|e_\kappa|^2$ for the second; because the
$x$-marginal of $\pi_t$ is $\rho^\kappa(t)$, the weighted $|w|^2$ term is
exactly $\mathscr D$.

Set $\mathscr Y:=\mathscr F+\mathscr Q$.  Adding the preceding inequality to
\eqref{eq:continuum_screening_energy_differential_inequality} yields
\begin{equation}
\label{eq:continuum_screening_combined_differential_inequality}
 \mathscr Y'+\mathscr D
 \le C_T\mathscr Y+C_T\alpha_d(\kappa)^2
 \qquad\text{for a.e. }t.
\end{equation}
Since the two solutions have the same initial density,
$\mathscr F(0)=\mathscr Q(0)=0$.  Multiplying
\eqref{eq:continuum_screening_combined_differential_inequality} by the
integrating factor $e^{-C_Tt}$ gives
\[
 \sup_{t\le T}\mathscr Y(t)+\int_0^T\mathscr D(t)\,\dd t
 \le C_T\alpha_d(\kappa)^2.
\]
Both $\mathscr F$ and $\mathscr Q$ are nonnegative, and
$W_2^2(\rho^\kappa(t),\rho^{\mathrm C}(t))\le\mathscr Q(t)$.  Recalling
\eqref{eq:continuum_screening_dissipation_definition} proves the stronger
estimate \eqref{eq:continuum_screening_stability_strong} and hence the
proposition.
\end{proof}

\subsubsection{Product data and simultaneous limits}

\begin{corollary}[Simultaneous Yukawa--Coulomb limit for tensorized initial data]
\label{cor:classical_simultaneous_yukawa_coulomb}
Assume the hypotheses of Proposition~\ref{prop:yukawa_classical_solutions} and
let $\rho^{\mathrm C}$ be the Coulomb solution on the common existence interval
$[0,T_*)$.  Fix $0<T<T_*$.  There exists
$C_T=C(d,\beta,\kappa_*,\rho_0,T)<\infty$ such that, for every
$N\ge2$, every choice $0<\kappa_N\le\kappa_*$, and every
$1\le k\le N$, the particle system starting from
$\rho_0^{\otimes N}$ satisfies
\begin{align*}
 \sup_{t\le T}\frac1N W_2^2
 (\rho_N^{\kappa_N}(t),(\rho^{\mathrm C}(t))^{\otimes N})
 &\le C_T\bigl(a_{N,d}+\alpha_d(\kappa_N)^2\bigr),
 \\
 \sup_{t\le T}W_2^2
 (\rho_{N:k}^{\kappa_N}(t),(\rho^{\mathrm C}(t))^{\otimes k})
 &\le C_Tk\bigl(a_{N,d}+\alpha_d(\kappa_N)^2\bigr).
\end{align*}
Thus the error is $N^{-2/d}+\kappa_N^4$ for $d\ge3$ and
$(1+\log N)/N+\kappa_N^2$ for $d=2$.  If $\kappa_N\downarrow0$, both the
normalized $N$-particle error and the error of every fixed marginal tend to
zero as $N\to\infty$, without any condition on the relative rates.
\end{corollary}

\begin{proof}
Corollary~\ref{cor:yukawa_classical_chaos} gives the Yukawa mean-field error
$C_Ta_{N,d}$ uniformly in $\kappa_N$, while
Proposition~\ref{prop:continuum_screening_stability} gives the continuum error
$C_T\alpha_d(\kappa_N)^2$.  Tensorization gives the exact identities
\[
 \frac1N W_2^2\bigl((\rho^{\kappa_N}(t))^{\otimes N},
                    (\rho^{\mathrm C}(t))^{\otimes N}\bigr)
 =W_2^2(\rho^{\kappa_N}(t),\rho^{\mathrm C}(t)),
\]
and
\[
 W_2^2\bigl((\rho^{\kappa_N}(t))^{\otimes k},
            (\rho^{\mathrm C}(t))^{\otimes k}\bigr)
 =k\,W_2^2(\rho^{\kappa_N}(t),\rho^{\mathrm C}(t)).
\]
For the full $N$-particle law, apply the triangle inequality through
$(\rho^{\kappa_N}(t))^{\otimes N}$, divide the squared estimate by $N$, and
use $(a+b)^2\le2a^2+2b^2$.  The same argument for the $k$-marginal uses
$(\rho^{\kappa_N}(t))^{\otimes k}$ and the second identity.  The resulting
universal factor is absorbed into $C_T$.
\end{proof}

\medskip
\noindent In dimension $d=3$, Theorem~\ref{thm:simultaneous_yukawa_coulomb}
gives the stronger direct comparison at the particle level for general
symmetric initial laws.  The initial Wasserstein and Coulomb modulated energy
errors remain explicitly contained in $\mathfrak E_N^{\mathrm C,0}$.  In $d=2$
and $d\ge4$, the result for smooth product data proceeds through the continuum
stability estimate.

\section{Discussion}

The principal result is Theorem~\ref{thm:yukawa_commutator}, a first-order
commutator inequality in the natural Yukawa modulated energy with constants
uniform as $\kappa\downarrow0$.  Its proof relies on a renormalization adapted
to the modified Helmholtz operator.  When the Yukawa potential is truncated to
a constant inside a ball, the resulting source consists of a positive surface
charge and a positive volume charge with total mass strictly less than one.
Nevertheless, the exterior Yukawa potential is unchanged.  This exact exterior
representation makes it possible to track the loss of mass explicitly in the
renormalized energy.  The stress-energy identity
with interface terms and the averaging over truncation radii then convert the
surface contributions into bulk field energy.  The resulting estimate is
uniform for $0<\kappa\le\kappa_*$ and in the particle configuration.

The additive error is of order $N^{-2/d}$ for $d\ge3$, matching the Coulomb
and Riesz scale at $s=d-2$, and of order $(1+\log N)/N$ in $d=2$.  Once the
commutator estimate is available, the dynamical conclusions follow from the
modulated energy identity and the normalized quadratic transport estimate: the
former contributes $-2\mathcal D_N$, while the derivative of the latter
contributes at most $+\mathcal D_N$.  The remaining dissipation yields
Wasserstein control of the
full $N$-particle law, control of the expected modulated energy, propagation of
chaos for each fixed marginal, and a time-integrated estimate of the
mean-square discrepancy between empirical and mean-field forces.  Thus the
screening-specific analysis is confined to the commutator estimate, whereas
the subsequent transport argument is more general.

The Coulomb limit $\kappa\downarrow0$ involves a different obstruction.  At
low frequency, the Coulomb and Yukawa quadratic forms are not uniformly
comparable; indeed, the Coulomb energy is not globally controlled by the
Yukawa energy even for fixed $\kappa>0$.  We therefore compare the continuum
solutions using the neutral Coulomb $\dot H^{-1}$ energy together with a
quadratic transport cost, rather than by comparing the two modulated energies
coercively.  This yields simultaneous limits $N\to\infty$ and
$\kappa_N\downarrow0$ for smooth product data without any condition on their
relative rates.  In dimension three, the bounded pointwise Yukawa--Coulomb
force difference also gives the direct comparison at the particle level for
general symmetric laws with finite initial error stated in
Theorem~\ref{thm:simultaneous_yukawa_coulomb}.

Several extensions remain open.  One is whether the truncation argument for
the modified Helmholtz operator extends to broader classes of screened elliptic
Green functions for which truncation changes the source mass.  Another is to
determine the optimal regularity of the reference velocity in the Yukawa
commutator estimate, in analogy with lower-regularity Coulomb theory.  Finally,
a direct Yukawa--Coulomb comparison at the particle level in dimensions
$d\ge4$ would require an argument beyond the pointwise force perturbation used
in $d=3$.

\appendix

\section{Yukawa kernels and the Riesz-type admissible class}
\label{app:yukawa_riesz_admissible}

We verify the large-scale comparison used in
Subsection~\ref{subsec:intro_related_work} by checking whether the Yukawa kernel
belongs to the $(s,\phi)$-admissible class of
\cite[Definition~2.8, equations~(2.31)--(2.33)]{HessChildsRosenzweigSerfaty2025}.

\begin{proposition}[Exclusion from the Riesz-type admissible class]
\label{prop:yukawa_not_riesz_type}
Fix $\kappa>0$.  If $d\ge3$ and $s=d-2$, then $g_{\kappa,d}$ is not an
$(s,\phi)$-admissible potential in the sense of
\cite[Definition~2.8]{HessChildsRosenzweigSerfaty2025}.  If $d=2$, no fixed
additive normalization $g_{\kappa,2}+c$, $c\in\R$, belongs to the corresponding
logarithmic admissible class.
\end{proposition}

\begin{proof}
Suppose first that $d\ge3$, so $s=d-2>0$.  Definition~2.8, in particular
(2.31)--(2.32), requires the scale weight to be bounded above and below by
constant multiples of $t^{d-s}$.  As noted in the paragraph immediately
following that definition in \cite{HessChildsRosenzweigSerfaty2025}, an
$(s,\phi)$-admissible potential with $s>0$ is consequently pointwise
comparable in physical and Fourier space with the exact Riesz potential.  In
particular, there exist positive constants $c_0$ and $C_0$ such that
\[
 c_0|x|^{2-d}\le g(x)\le C_0|x|^{2-d},
 \qquad x\ne0.
\]
By contrast, the large-argument asymptotics of the modified Bessel function
imply
\[
 g_{\kappa,d}(x)
 =C_{d,\kappa}|x|^{-(d-1)/2}e^{-\kappa|x|}
   \bigl(1+O(|x|^{-1})\bigr),
 \qquad |x|\to\infty.
\]
Hence
\[
 \frac{g_{\kappa,d}(x)}{|x|^{2-d}}\longrightarrow0
 \qquad\text{as }|x|\to\infty,
\]
which contradicts the required two-sided comparison.  Adding a nonzero fixed
constant cannot restore comparability: the shifted kernel tends to that
constant, whereas $|x|^{2-d}\to0$.

In dimension two, Definition~2.8, especially (2.31) and (2.33), together
with the paragraph immediately following the definition in
\cite{HessChildsRosenzweigSerfaty2025}, identifies the admissible potential
with the logarithmic Riesz kernel plus a more regular positive-definite
remainder of the form
\[
 h(x)=\int_0^\infty \varrho(t)\phi_t(x)\,\frac{\dd t}{t},
 \qquad
 \int_0^\infty \varrho(t)t^{-d}\,\frac{\dd t}{t}<\infty.
\]
Since $\phi$ is bounded, this remainder is bounded.  Thus such an admissible
representative has the large-scale behavior $-\log|x|+O(1)$ and tends to
$-\infty$.  On the other hand,
\[
 2\pi g_{\kappa,2}(x)=K_0(\kappa|x|)\longrightarrow0
 \qquad\text{exponentially as }|x|\to\infty,
\]
and adding any fixed constant gives a finite limit at infinity.  This excludes
every fixed additive normalization of the two-dimensional Yukawa kernel.
\end{proof}

Proposition~\ref{prop:yukawa_not_riesz_type} identifies the large-scale
obstruction to placing the Yukawa kernel in the Riesz-type admissible class.
The independent low-frequency comparison of the Coulomb and Yukawa quadratic
energies is established in
Subsection~\ref{sec:coulomb_perturbation_comparison}.

\section{Uniform local classical solutions}
\label{app:local_classical_solutions}

\begin{proof}[Proof of Proposition~\ref{prop:yukawa_classical_solutions}]
Let $R_0>0$ satisfy $\operatorname{supp}\rho_0\subset B_{R_0}$, and fix
$R>R_0+1$.  Throughout the proof, $C$ denotes a constant depending only on
$d,R,\beta,\kappa_*$ and the uniform bound defining the fixed-point set
below.  In particular, all such constants are uniform for
$0\le\kappa\le\kappa_*$.  Proposition~\ref{prop:yukawa_velocity_field}
provides both the H\"older estimates for the velocity field and the
$L^\infty$ stability estimate used in the contraction argument.

We set
$
 M:=2\bigl(1+\|\rho_0\|_{C^{1,\beta}}\bigr)
$
and define
\begin{equation}
\label{eq:yukawa_fixed_point_space}
\begin{aligned}
 \mathcal X_T:=\{f\in C([0,T];C_0(\R^d)):\;&
 f(0)=\rho_0,\ f(t)\in C_c^{1,\beta}(\R^d),\ f\ge0,\\
 &\int f(t)=1,\ \operatorname{supp}f(t)\subset\overline{B_R},\quad
 \sup_{t\le T}\|f(t)\|_{C^{1,\beta}}\le M\}.
\end{aligned}
\end{equation}
Equip $\mathcal X_T$ with
\[
 d_T(f,g):=\sup_{0\le t\le T}\|f(t)-g(t)\|_{L^\infty}.
\]
The space $(\mathcal X_T,d_T)$ is complete.  Indeed, any $d_T$-Cauchy
sequence converges uniformly on $[0,T]\times\R^d$ to some
$f\in C([0,T];C_0)$.  For each fixed $t$, the uniform $C^{1,\beta}$ bound and
the common compact support yield, by Arzel\`a--Ascoli, a subsequence whose
gradients converge uniformly.  Passing to the limit in
$\int f_n\partial_j\varphi=-\int(\partial_jf_n)\varphi$ identifies the
limit with $\partial_jf$, while lower semicontinuity preserves the
$C^{1,\beta}$ bound.  Nonnegativity, unit mass, the initial condition, and the
support constraint also pass to the limit.  Hence $f\in\mathcal X_T$.

For every $0<\beta'<\beta$, interpolation between $C^0$ and
$C^{1,\beta}$ gives
\[
 \|f(t)-f(s)\|_{C^{1,\beta'}}
 \le C\|f(t)-f(s)\|_{C^0}^{\theta}(2M)^{1-\theta}
\]
for some $\theta\in(0,1)$.  Hence $f\in C([0,T];C^{1,\beta'})$.  By
\eqref{eq:yukawa_velocity_field_holder_continuity}, the same time continuity
holds for $u_f^\kappa$ and $\vartheta_f^\kappa$, with uniform
$C^{1,\beta}$ bounds.

For $f\in\mathcal X_T$, let $X_f$ be the flow of $u_f^\kappa$,
\begin{equation*}
 \partial_tX_f(t,a)=u_f^\kappa(t,X_f(t,a)),
 \qquad X_f(0,a)=a.
\end{equation*}
By \eqref{eq:yukawa_velocity_field_strong}, there are constants $U,L$,
depending only on $(d,R,M,\beta,\kappa_*)$, such that for every
$f\in\mathcal X_T$,
\begin{equation}
\label{eq:yukawa_fixed_point_uniform_coefficients}
 \|u_f^\kappa\|_{L^\infty}\le U,
 \qquad
 \|\nabla u_f^\kappa\|_{C^{0,\beta}}\le L,
 \qquad
 \|\vartheta_f^\kappa\|_{C^{1,\beta}}\le C.
\end{equation}
Consequently
\begin{equation*}
 \|D_aX_f(t)\|_{L^\infty}\le e^{Lt},
 \qquad
 [D_aX_f(t)]_{C^{0,\beta}}
 \le C_{L,\beta}\,t e^{C_{L,\beta}t},
\end{equation*}
and analogous bounds hold for the inverse flow $A_f(t)=X_f(t)^{-1}$.
The second estimate follows by subtracting the equations for
$D_aX_f(t,a)$ and $D_aX_f(t,b)$ and applying Gronwall's inequality.

Let $\mathcal T_\kappa f$ be the solution of the linear continuity equation with
velocity $u_f^\kappa$ and initial datum $\rho_0$.  Along the flow,
\begin{equation}
\label{eq:yukawa_fixed_point_density_formula}
 (\mathcal T_\kappa f)(t,X_f(t,a))
 =\rho_0(a)\exp\left(-\int_0^t
 \vartheta_f^\kappa(s,X_f(s,a))\,\dd s\right).
\end{equation}
This formula preserves nonnegativity and unit mass.  The time continuity of
$u_f^\kappa$, $\vartheta_f^\kappa$, and the characteristic flow then implies
from \eqref{eq:yukawa_fixed_point_density_formula} that
$\mathcal T_\kappa f\in C([0,T];C_0(\R^d))$.  If
$T\le (R-R_0)/U$, then
\[
 \operatorname{supp}(\mathcal T_\kappa f)(t)
 \subset\overline{B_{R_0+Ut}}\subset\overline{B_R}.
\]
Differentiating \eqref{eq:yukawa_fixed_point_density_formula} with respect to
$a$ gives
\begin{equation}
\label{eq:yukawa_fixed_point_density_gradient_formula}
\begin{aligned}
 \nabla_a[(\mathcal T_\kappa f)(t,X_f(t,a))]
 =\exp\!\left(-\int_0^t\vartheta_f^\kappa(s,X_f(s,a))\,\dd s\right)\Bigg[
   \nabla\rho_0(a)-\rho_0(a)\int_0^t
   \nabla \vartheta_f^\kappa(s,X_f(s,a))D_aX_f(s,a)\,\dd s\Bigg],
\end{aligned}
\end{equation}
To pass from Lagrangian to Eulerian coordinates, compose the Lagrangian
profile in \eqref{eq:yukawa_fixed_point_density_formula} with the inverse flow
$A_f(t)=X_f(t)^{-1}$.  We use the composition and product estimates
\[
 [G\circ A]_{C^{0,\beta}}
 \le [G]_{C^{0,\beta}}\|DA\|_{L^\infty}^{\beta},
\qquad
 [D(G\circ A)]_{C^{0,\beta}}
 \le C\Bigl([DG]_{C^{0,\beta}}\|DA\|_{L^\infty}^{1+\beta}
 +\|DG\|_{L^\infty}[DA]_{C^{0,\beta}}\Bigr).
\]
Combining these estimates with
\eqref{eq:yukawa_fixed_point_uniform_coefficients}--
\eqref{eq:yukawa_fixed_point_density_gradient_formula} gives
\begin{equation*}
 \sup_{t\le T}\|\mathcal T_\kappa f(t)\|_{C^{1,\beta}}
 \le \|\rho_0\|_{C^{1,\beta}}+C
 T e^{CT}.
\end{equation*}
After decreasing $T$, the right-hand side is at most $M$.  Thus
$\mathcal T_\kappa$ maps $\mathcal X_T$ into itself.

We next prove that $\mathcal T_\kappa$ is a contraction.  Let
$Z_f(s;t,x)$ denote the backward characteristic ending at $x$ at time $t$:
\[
 \partial_sZ_f(s;t,x)=u_f^\kappa(s,Z_f(s;t,x)),
 \qquad Z_f(t;t,x)=x.
\]
From \eqref{eq:yukawa_velocity_field_weak} and the uniform
Lipschitz bound in \eqref{eq:yukawa_fixed_point_uniform_coefficients},
\begin{equation}
\label{eq:yukawa_fixed_point_flow_contraction}
 \sup_{0\le s\le t\le T}\sup_x
 |Z_f(s;t,x)-Z_g(s;t,x)|
 \le CTe^{LT}d_T(f,g).
\end{equation}
Set
\[
 J_f(t,x):=\int_0^t \vartheta_f^\kappa(s,Z_f(s;t,x))\,\dd s.
\]
Then \eqref{eq:yukawa_velocity_field_weak},
\eqref{eq:yukawa_fixed_point_uniform_coefficients}, and
\eqref{eq:yukawa_fixed_point_flow_contraction} imply
\begin{equation}
\label{eq:yukawa_fixed_point_exponent_contraction}
 \|J_f-J_g\|_{L^\infty([0,T]\times\R^d)}
 \le CTe^{CT}d_T(f,g).
\end{equation}
The backward form of \eqref{eq:yukawa_fixed_point_density_formula} is
\[
 (\mathcal T_\kappa f)(t,x)
 =\rho_0(Z_f(0;t,x))e^{-J_f(t,x)}.
\]
Since $\rho_0$ is Lipschitz and $J_f,J_g$ are uniformly bounded,
\eqref{eq:yukawa_fixed_point_flow_contraction}--
\eqref{eq:yukawa_fixed_point_exponent_contraction} give
\begin{equation*}
 d_T(\mathcal T_\kappa f,\mathcal T_\kappa g)
 \le CTe^{CT}d_T(f,g).
\end{equation*}
Choose $T$ smaller, if necessary, so that $CTe^{CT}<1$.  Banach's fixed-point
theorem then yields, for each $0\le\kappa\le\kappa_*$, a unique
$\rho^\kappa\in\mathcal X_T$ satisfying
$\mathcal T_\kappa\rho^\kappa=\rho^\kappa$.  Because all constants above
are uniform in $\kappa$, the same $T$ is valid throughout
$[0,\kappa_*]$.  The contraction argument also gives uniqueness in the stated
class.  Indeed, on any sufficiently short subinterval, two solutions with a
common support bound and common $C^{1,\beta}$ bound belong to the same
fixed-point class of the form \eqref{eq:yukawa_fixed_point_space} and hence
coincide.  Iteration yields uniqueness on the full common interval.

Fix $\kappa$ and write $\rho=\rho^\kappa$.  The interpolation estimate
above gives $\rho\in C([0,T];C^{1,\beta'})$ for every
$0<\beta'<\beta$.  Equation~\eqref{eq:yukawa_velocity_field_holder_continuity}
then gives the corresponding continuity of $u_\rho^\kappa$ and
$\vartheta_\rho^\kappa=\nabla\cdot u_\rho^\kappa$.  Therefore
$\partial_t\rho=-u_\rho^\kappa\cdot\nabla\rho-\rho\vartheta_\rho^\kappa$
belongs to $C([0,T];C^{0,\beta'})$, which proves
\eqref{eq:yukawa_classical_solution_class}.
Taking this common $T$ as $T_*$ establishes the asserted uniform local
existence and uniqueness.  Each solution can be continued as long as its
$C^{1,\beta}$ norm remains finite.

We finally verify the uniform bounds in
\eqref{eq:yukawa_classical_solution_bound}.  Since
$
 (-\Delta+\kappa^2)g_{\kappa,d}=\delta_0,
$
one has
\begin{equation*}
 \nabla\cdot u_\rho^\kappa
 =-\Delta g_{\kappa,d}*\rho
 =\rho-\kappa^2g_{\kappa,d}*\rho.
\end{equation*}
Along a characteristic $X(t,a)$,
\begin{equation}
\label{eq:yukawa_solution_characteristic_density}
 \frac{\dd}{\dd t}\rho(t,X(t,a))
 =-\rho(t,X(t,a))^2
  +\kappa^2\rho(t,X(t,a))
   (g_{\kappa,d}*\rho)(t,X(t,a)).
\end{equation}
For $\kappa>0$,
$
 \kappa^2\int_{\R^d}g_{\kappa,d}(x)\,\dd x=1.
$
Thus, with $M_\rho(t)=\|\rho(t)\|_{L^\infty}$,
$
 \kappa^2\|g_{\kappa,d}*\rho(t)\|_{L^\infty}\le M_\rho(t).
$
The upper Dini derivative at a spatial maximum therefore satisfies
\begin{equation*}
 D^+M_\rho(t)\le -M_\rho(t)^2+M_\rho(t)^2=0.
\end{equation*}
For $\kappa=0$, the screening term in
\eqref{eq:yukawa_solution_characteristic_density} is absent and
$D^+M_\rho(t)\le-M_\rho(t)^2\le0$.
Thus
\begin{equation*}
 \|\rho(t)\|_{L^\infty}\le\|\rho_0\|_{L^\infty}
 \qquad (0\le t<T_*).
\end{equation*}
For every $T<T_*$, the bound
$\rho\in L^\infty(0,T;C^{1,\beta})$ controls the H\"older seminorm of
$\rho$ on $[0,T]$.  Proposition~\ref{prop:yukawa_velocity_field} then yields
a $W^{1,\infty}$ bound for $u_\rho^\kappa$ that is uniform for
$0\le\kappa\le\kappa_*$.

Finally, \eqref{eq:yukawa_fixed_point_uniform_coefficients} gives a support
bound on every compact time interval that is uniform in
$\kappa\in[0,\kappa_*]$.  More precisely, if
$\operatorname{supp}\rho_0\subset B_{R_0}$, then for every $T<T_*$,
\begin{equation*}
 \operatorname{supp}\rho(t)\subset B_{R_0+C_Tt},
 \qquad 0\le t\le T.
\end{equation*}
The support bound immediately yields a uniform second-moment estimate.  Moreover,
since $\rho(t)=X(t,\cdot)_\#\rho_0$,
\[
 W_2\bigl(\rho(t),\rho(s)\bigr)
 \le \|X(t,\cdot)-X(s,\cdot)\|_{L^2(\rho_0)}
 \le C_T|t-s|.
\]
This proves \eqref{eq:yukawa_classical_solution_bound}.
\end{proof}

\section*{Statements and Declarations}
\textbf{Funding.} This work was supported in part by the National Natural
Science Foundation of China under Grant No.~12371180.\par
\textbf{Competing interests.} The authors declare that they have no competing
financial or non-financial interests.\par
\noindent\textbf{Data availability.} No datasets were generated or analyzed in the
course of this study.

\end{document}